\documentclass[12pt]{tac}
\usepackage{amsfonts}
\usepackage{amssymb}
\usepackage{amsmath}
\usepackage{tikz-cd}
\usepackage{enumerate}
\usepackage{relsize}
\usepackage{scalerel}
\usepackage{bm}
\usepackage{amsrefs}
\usepackage[makeroom]{cancel}
\usepackage{appendix}
\usepackage{hyperref}
\usepackage{bm}
\usepackage{setspace}

\usepackage{fullpage}

\newtheorem{theorem}{Theorem}[section]
\newtheorem{lemma}[theorem]{Lemma}
\newtheorem{proposition}[theorem]{Proposition}
\newtheorem{corollary}[theorem]{Corollary}

\theoremstyle{definition}
\newtheorem{definition}[theorem]{Definition}

\newtheorem{example}[theorem]{Example}

\theoremstyle{remark}

\newcommand{\D}{\mathcal D}
\newcommand{\E}{\mathcal E}

\renewcommand{\H}{\mathcal H}

\renewcommand{\P}{\mathcal P}

\newcommand{\X}{\mathcal X}
\newcommand{\Y}{\mathcal Y}
\newcommand{\Z}{\mathcal Z}
\newcommand{\qo}{{\boldsymbol{\preccurlyeq}}}
\newcommand{\qop}{{\boldsymbol{\succcurlyeq}}}

\newcommand{\CC}{\mathbb C}

\renewcommand{\tt}{\mathfrak t}

\newcommand{\qedhere}{}

\newcommand{\Tr}{\mathrm{Tr}}
\newcommand{\At}{\mathrm{At}}

\newcommand{\Rel}{\mathrm{Rel}}
\newcommand{\Eval}{\mathrm{Eval}}

\newcommand{\op}{\mathrm{op}}

\newcommand{\id}{\mathrm{id}}

\newcommand{\true}{\mathrm{true}}

\newcommand{\Set}{\mathbf{Set}}

\newcommand{\qRel}{\mathbf{qRel}}
\newcommand{\qSet}{\mathbf{qSet}}

\renewcommand{\name}[1]{\ulcorner #1\urcorner}
\newcommand{\coname}[1]{\llcorner #1\lrcorner}

\makeatletter
\newcommand{\circlearrow}{}
\DeclareRobustCommand{\circlearrow}{%
  \mathrel{\vphantom{\rightarrow}\mathpalette\circle@arrow\relax}%
}
\newcommand{\circle@arrow}[2]{%
  \m@th
  \ooalign{%
    \hidewidth$#1\circ\mkern1mu$\hidewidth\cr
    $#1\longrightarrow$\cr}%
}
\makeatother

\makeatletter
\def\slashedarrowfill@#1#2#3#4#5{%
  $\m@th\thickmuskip0mu\medmuskip\thickmuskip\thinmuskip\thickmuskip
  \relax#5#1\mkern-7mu%
  \cleaders\hbox{$#5\mkern-2mu#2\mkern-2mu$}\hfill
  \mathclap{#3}\mathclap{#2}%
  \cleaders\hbox{$#5\mkern-2mu#2\mkern-2mu$}\hfill
  \mkern-7mu#4$%
}
\def\rightslashedarrowfill@{%
  \slashedarrowfill@\relbar\relbar\mapstochar\rightarrow}
\newcommand\xslashedrightarrow[2][]{%
  \ext@arrow 0055{\rightslashedarrowfill@}{#1}{#2}}
\makeatother
\def\sto{\relbar\joinrel\mapstochar\joinrel\rightarrow}

\newcommand*{\vcbox}[1]{\begingroup
\setbox0=\hbox{#1}\parbox{\wd0}{\box0}\endgroup}

\title{Monoidal quantaloids}
\author{Gejza Jen\v{c}a and Bert Lindenhovius}
\dedication{Dedicated to Thomas Vetterlein for his friendship and guidance}    
\begin{document}

\maketitle

\begin{abstract}
We investigate how to add a symmetric monoidal structure to quantaloids in a compatible way. In particular, dagger compact quantaloids turn out to have properties that are similar to the category $\mathbf{Rel}$ of sets and binary relations. Examples of such quantaloids are the category $\mathbf{qRel}$ of quantum sets and binary relations, and the category $V$-$\mathbf{Rel}$ of sets and binary relations with values in a commutative quantale $V$. For both examples, the process of internalization structures is of interest. \emph{Discrete quantization}, a process of generalizing mathematical structures to the noncommutative setting can be regarded as the process of internalizing these structures in $\mathbf{qRel}$, whereas \emph{fuzzification}, the process of introducing degrees of truth or membership to concepts that are traditionally considered either true or false, can be regarded as the process of internalizing structures in $V$-$\mathbf{Rel}$. Hence, we investigate how to internalize power sets and preordered structures in dagger compact quantaloids.
\end{abstract}

\tableofcontents

\section{Introduction}

The work we present here evolved from the research on mathematical quantization via quantum relations. Here, \emph{mathematical quantization}, also briefly called \emph{quantization}, refers to the process of generalizing mathematical structures to the noncommutative setting, typically in terms of operators on Hilbert spaces. For example, because of Gelfand duality between locally compact Hausdorff spaces and commutative C*-algebras, one can regard general C*-algebras as noncommutative generalizations of locally compact Hausdorff spaces, and many theorems on locally compact Hausdorff spaces can be generalized to arbitrary C*-algebras. This observation is crucial in the program of \emph{noncommutative geometry} \cite{connes:ncg}, in which the concepts and tools of geometry are generalized to the noncommutative setting. Another example is provided by von Neumann algebras, which can be regarded as noncommutative generalizations of measure spaces. The reason why one could be interested in such noncommutative generalizations is because most quantum phenomena can be described in terms of noncommutative structures. Since many of these phenomena have classical counterparts; noncommutative generalizations of the mathematical structures describing these classical counterparts often can be used to describe the quantum phenomena. For example, \emph{complete partial orders (cpos)}, i.e., posets in which every monotonically ascending sequence has a supremum, can be used to model programming languages with recursion. Recently, cpos were quantized, and the resulting \emph{quantum cpos} were used to model quantum programming languages with recursion \cite{klm-quantumcpos,klm-qpl}.
 
The quantization method employed for the quantization of cpos is called \emph{discrete quantization}, which is based on a noncommutative generalization of sets and binary relations. The starting point of this approach are \emph{quantum relations} between von Neumann algebras, introduced by Weaver in \cite{Weaver10} and distilled from his joint work with Kuperberg on the quantization of metric spaces \cite{kuperbergweaver:quantummetrics}. Quantum relations can be regarded as noncommutative analogues of binary measurable relations between measure spaces.

As mentioned above, von Neumann algebras are noncommutative generalizations of measure spaces. Just like sets can be regarded as a subclass of measure spaces (by equipping them with the Dirac measure), one can identify a subclass of von Neumann algebras that are noncommutative generalizations of sets. These von Neumann algebras are called \emph{hereditarily atomic}, and are by definition isomorphic to (possibly infinite) sums of complex-valued matrix algebras \cite{Kornell18}, and can therefore be represented by a sum $\bigoplus_{\alpha\in A}B(H_\alpha)$ for some set-indexed family $(H_\alpha)_{\alpha\in A}$ of finite-dimensional Hilbert spaces. We call such a set-indexed family of finite-dimensional Hilbert spaces a \emph{quantum set}. Also quantum relations between hereditarily atomic von Neumann algebras correspond to morphisms between quantum sets, which we simply call \emph{binary relations} between quantum sets, as they are noncommutative generalizations of  binary relations between ordinary sets. Even though quantum sets and hereditarily atomic von Neumann algebras contain the same information, we will see that categories of ordinary sets embed covariantly into categories of quantum sets, whereas they embed contravariantly into categories of hereditarily atomic von Neumann algebras. For this reason, as a generalization of sets, quantum sets might feel more natural than hereditarily atomic von Neumann algebras. 

Quantum sets and binary relations between quantum sets form a category $\mathbf{qRel}$, which is dagger compact and enriched over the category $\mathbf{Sup}$ of complete lattices and suprema-preserving morphisms. Any category with such an enrichment is called a \emph{quantaloid.} Because of these properties, $\mathbf{qRel}$ admits a well-behaved calculus of relations, including notions of symmetry, antisymmetry, transitivity, and reflexivity. This makes it possible to internalize classical  structures within $\mathbf{qRel}$, leading to noncommutative analogues of familiar objects like graphs and posets. In this sense, the process of discrete quantization can be understood as the internalization of classical structures in $\mathbf{qRel}.$

Although a quantization process based on on arbitrary von Neumann algebras, discrete quantization  suffices for most applications in quantum information theory and quantum computing. Moreover, the compact structure of $\mathbf{qRel}$ enables constructions, such as the quantization of the power set monad, that appear impossible at the level of general von Neumann algebras.

The strength of discrete quantization lies further in the fact that it allows one to quantize theories instead of just categories. For instance, in \cite{klm-quantumposets}, the category of quantum posets was investigated, and many theorems in order theory carry over to the quantum case. Similarly, in \cite{klm-quantumcpos}, $\omega$-\emph{complete partial orders (cpos)} were quantized, and the category of the resulting quantum cpos was investigated. In practice, in order to prove noncommutative versions of theorems in a theory one tries to quantize via discrete quantization, one sometimes relies on arguments based on the structure of quantum sets. However, more often, one can prove the theorems purely via categorical arguments based on the categorical structure of $\mathbf{qRel}$ or $\mathbf{WRel}$. 
This leads to the question whether we can reduce the proofs completely to categorical arguments. 

We note that $\mathbf{Rel}$ is the prime example of an \emph{allegory}, a kind of category generalizing $\mathbf{Rel}$ introduced in \cite{allegories}, just like topoi generalize $\mathbf{Set}$. Allegories are strongly related to topoi, since the latter are precisely the categories of internal maps in power allegories, i.e., allegories with so-called \emph{power objects} that generalize power sets. As a consequence, allegories have a rich structure that allow for the systematic internalization of most mathematical structures. However, $\mathbf{qRel}$ fails to be an allegory (cf. Example \ref{ex:qRel}). \emph{Bicategories of relations} form another categorical generalization of $\mathbf{Rel}$  introduced in \cite{BiCatRel}, but since every bicategory of relations is an allegory, $\mathbf{qRel}$ cannot be a bicategory of relations either. Fundamentally, the biggest issue seems to be that the category $\mathbf{qSet}$ of internal maps in $\mathbf{qRel}$ inherits a monoidal product from $\mathbf{qRel}$ that is not cartesian. 

Hence, we cannot rely on existing categorical frameworks that generalize $\mathbf{Rel}$. Instead we draw inspiration from recent axiomatizations of dagger categories such as the category $\mathbf{Hilb}$ of Hilbert spaces and bounded linear maps \cite{heunenkornell} or the category $\mathbf{Rel}$ \cite{Kornell23}. Hence, we try to identify the essential categorical properties of $\mathbf{qRel}$ that allow for a systematic quantization of most mathematical structures. We also hope that the identification of these properties will be a step in the direction of an eventual axiomatization of $\mathbf{qRel}$. 

We also draw inspiration from \emph{fuzzification}, the process of introducing degrees of truth or membership to concepts that are traditionally considered either true or false. Just like quantization, this process can also be regarded as an internalization process in a category that resembles $\mathbf{Rel}$, namely the category $V$-$\mathbf{Rel}$ of sets and binary relations with values in a commutative quantale $V$, which represents the degrees of truth. One retrieves $\mathbf{Rel}$ as a special case of $V$-$\mathbf{Rel}$ by choosing $V$ to be the two-point lattice. There are ample examples of choices of $V$ for which $V$-$\mathbf{Rel}$ is not an allegory, for instance, when $V$ is affine, but not a frame (cf. Example \ref{ex:V-Rel-dagger}). We further note that the category $V$-$\mathbf{Rel}$ also plays a role in the field of \emph{monoidal topology} \cite{monoidaltopology}. Here, one unifies ordered, metric and topological structures in a single framework of lax algebras of lax monads on $V$-$\mathbf{Rel}$ for some suitable quantale $V$.

Thus, the starting point for our work is the categorical structure that is shared by $\mathbf{Rel}$, $\mathbf{qRel}$ and $V$-$\mathbf{Rel}$, which are both dagger compact categories that are simultaneously \emph{quantaloids}, i.e., categories enriched over the category $\mathbf{Sup}$ of complete lattice and suprema-preserving maps. In the preliminaries, i.e., Section \ref{sec:preliminaries}, we explore (dagger) symmetric monoidal categories and quantaloids, and biproducts in these categories. In Section \ref{sec:monoidal-quantaloids}, we discuss how to combine monoidal structures with a quantaloid structure in a compatible way, leading to the main notions of this paper, which we call a \emph{ (dagger) symmetric monoidal quantaloids} and \emph{(dagger) compact quantaloids}. We show that the former generalize infinitely distributive (dagger) symmetric monoidal categories with a quantaloid structure. In Section \ref{sec:orthomodular}, we investigate how to integrate the notion of dagger kernel categories with dagger quantaloids, and how the existence of dagger kernels imply that homsets are orthomodular. In the remaining sections, we internalize various structures. Some of the obtained results were already proven for $\mathbf{qRel}$ in \cite{KLM20}, but here we reprove those results in the more general framework of dagger symmetric monoidal quantaloids. In Section \ref{sec:internal maps} we describe internal maps in symmetric monoidal quantaloids. In $\mathbf{Rel}$, these correspond to functions, in $\mathbf{qRel}$ to unital $*$-homomorphisms (as already known from the work of Kornell \cite{Kornell18}). In Section \ref{sec:internal preorders}, we study internal preorders, monotone maps, and monotone relations. In Section \ref{sec:embedding-sets}, we investigate under what conditions we can embed categories of set-theoretic structures into categories of their internalizations in a dagger symmetric monoidal quantaloid. Finally, in Section \ref{sec:power objects}, we use these structures and some extra assumptions to derive the existence of power objects. The most important of these assumptions is that the category of internal maps is symmetric monoidal closed. We conclude by investigating when the existence of power objects imply a monoidal closed structure of the internal maps.

\section{Preliminaries}\label{sec:preliminaries}

In the following, we will give the definitions of compact categories, biproducts and quantaloids. All these concepts can be combined with the notion of a dagger on a category. Moreover, we present two quantaloids, the category $V$-$\mathbf{Rel}$ of sets and binary relations with values in a quantale $V$, and the category $\mathbf{qRel}$ of quantum sets and binary relations between quantum sets, which will provide examples of the material developed in this article.

\begin{definition}
    A category $\mathbf C$ is called a \emph{dagger} category if it is equipped with a contravariant involutive functor $(-)^\dag$ that is the identity on objects. We refer to this functor as the \emph{dagger} on $\mathbf C$. Furthermore, a morphism $f:X\to Y$ in $\mathbf C$ is called
    \begin{itemize}
    \item a \emph{dagger mono} if $f^\dag\circ f=\id_X$;
    \item a \emph{dagger epi} if $f\circ f^\dag=\id_Y$;
    \item a \emph{dagger isomorphism} or a \emph{unitary} if it is both a dagger mono and a dagger epi;
        \item \emph{selfadjoint} if $X=Y$ and $f^\dag=f$;
    \item a \emph{projection} if $X=Y$ and $f\circ f=f=f^\dag$.    \end{itemize}
    \end{definition}

\begin{definition}
 A \emph{dagger} functor $ F:\mathbf C\to\mathbf D$ between dagger categories is a functor such that for each morphism $f:X\to Y$ in $\mathbf C$ we have $F(f^\dag)=(Ff)^\dag$. If, in addition, $F$ forms an equivalence of categories with a  dagger functor $G:\mathbf D\to\mathbf C$ such that the natural isomorphisms $GFX\cong X$ and $FGY\cong Y$ for $X\in\mathbf C$ and $Y\in\mathbf D$ are unitaries, then we call $F$ and $G$ an \emph{equivalence of dagger categories}, and $\mathbf C$ and $\mathbf D$ \emph{equivalent dagger categories}.
\end{definition}

\subsection{Monoidal categories}

We assume the reader is familiar with symmetric monoidal categories. To fix notation, we will write symmetric monoidal categories as $(\mathbf{C}, \otimes, I)$, 
suppressing the associator $\alpha$, the left and right unitors $\lambda, \rho$, 
and the symmetry $\sigma$. We call a symmetric monoidal category $(\mathbf C,\otimes, I)$
\emph{closed} if for each object $X\in\mathbf C$ the functor $\mathbf C\to\mathbf C$, $Y\mapsto X\otimes Y$ has a right adjoint, in which case we denote the right adjoint by $[X,-]$. The counit of the adjunction is denoted by $\Eval_X$. We denote the $Y$-component of $\Eval_X$ by $\Eval_{X,Y}$, which is a morphism $\Eval_{X,Y}:[X, Y]\otimes X\to Y$ that satisfies the universal property that for any morphism $f:Z\otimes X\to Y$ there is a unique morphism $\hat f:Z\to[X,Y]$ such that $\Eval_{X,Y}\circ (\hat f\otimes\id_{X})=f$. Often, we will simply write $\Eval$ instead of $\Eval_{X,Y}$.

In a symmetric monoidal category, the morphisms with the monoidal unit as codomain play a special role.

\begin{definition}
    Let $(\mathbf C,\otimes,I)$ be a symmetric monoidal category and let $X\in\mathbf C$ be an object. Then a morphism $e:X\to I$ is called an \emph{effect} on $X$. 
\end{definition}

The dual concept of an effect, i.e., a morphism with the monoidal unit as domain, is usually called a \emph{state}, but will be of lesser importance in this contribution. Another special role is played by morphisms that are simultaneously states and effects:

\begin{definition}
    Let $(\mathbf C,\otimes, I)$ be a symmetric monoidal category. We call the morphisms $s:I\to I$ \emph{scalars}. For any two objects $X$ and $Y$, we define \emph{(left) scalar multiplication} as the operation $\mathbf C(I,I)\times\mathbf C(X,Y)\to\mathbf C(X,Y)$, $(s,f)\mapsto s\cdot f$,  where $s\cdot f:=\lambda_Y\circ (s\otimes f)\circ \lambda_X^{-1}$.
\end{definition}

The proof of the following lemma is straightforward. 
\begin{lemma}\label{lem:IisZero}
    In a symmetric monoidal category $(\mathbf C,I,\otimes)$ with a zero object $0$ the following statements are equivalent:
    \begin{itemize}
        \item[(1)] $I\cong 0$;
        \item[(2)] $\id_I=0_I$;
        \item[(3)] there is precisely one scalar.
    \end{itemize}
\end{lemma}

If $(\mathbf C,\otimes,I)$ is a symmetric monoidal closed category with a zero object $0$ isomorphic to $I$, it follows for any two objects $X$ and $Y$ of $\mathbf C$ that $\mathbf C(X,Y)\cong\mathbf C(I,[X,Y])\cong\mathbf C(0,[X,Y])=1$, hence, there is exactly one morphism $X\to Y$. It follows that all objects of $\mathbf C$ are isomorphic to each other, hence $\mathbf C$ is equivalent to the trivial category.

\subsection{Compact categories}

       A symmetric monoidal category $(\mathbf C,\otimes,I)$ is called \emph{compact} or \emph{compact closed} if each object $X$ in $\mathbf C$ has a \emph{dual} $X^*$, i.e., an object for which there are morphisms  $\eta_X:I\to X^*\otimes X$ and $\epsilon_X:X\otimes X^*\to I$, called the \emph{unit} and the \emph{counit}, respectively, such that
    \begin{align}
\label{eq:compact1}        \lambda _{X}\circ (\epsilon _{X}\otimes \id_X)\circ \alpha _{{X,X^{*},X}}^{{-1}}\circ (\id_X\otimes \eta _{X})\circ \rho _{X}^{{-1}} & =\id_{X},\\
  \label{eq:compact2}      \rho_{X*}\circ(\id_{X^*}\otimes\epsilon_{X})\circ\alpha_{X^*,X,X^*}\circ(\eta_X\otimes\id_{X^*})\circ\lambda_{X^*}^{-1} & = \id_{X^*}
    \end{align}

In particular, any compact category $(\mathbf C,\otimes,I)$ is symmetric monoidally closed with internal hom $[X,Y]=X^*\otimes Y$ \cite{Kelly-Laplaza}.

    Let $f:X\to Y$ be a morphism in a compact closed category $(\mathbf C,\otimes,I)$. Then we define its \emph{name} $\name{f}:I\to X^*\otimes Y$ and \emph{coname} $\coname{f}:X\otimes Y^*\to I$ as the morphisms
    \begin{align*}
        \name{f} & :=(\id_{X^*}\otimes f)\circ\eta_X;\\
        \coname f &: = \epsilon_Y\circ (f\otimes\id_{Y^*}).
    \end{align*}

\begin{lemma}\label{lem:epsilon}
   Let $X$ and $Y$ be objects in a compact closed category $(\mathbf C,\otimes,I)$.
Then we have bijections
\begin{align*}
\mathbf C(X,Y)\xrightarrow{\cong}\mathbf C(I,X^*\otimes Y) & , \qquad g\mapsto \name g\\
 \mathbf C(X,Y)\xrightarrow{\cong}\mathbf C(X\otimes Y^*,I) & , \qquad f\mapsto\coname f,
\end{align*}
with respective inverses
\begin{align*}\mathbf C(I,X^*\otimes Y)\to\mathbf C(X,Y), & \qquad h\mapsto\lambda_Y\circ (\coname{\id_X}\otimes\id_Y)\circ\alpha_{X,Y^*,X}^ {-1}\circ(\id_X\otimes h)\circ\rho_X\\
    \mathbf C(X\otimes Y^*,I)\to\mathbf C(X,Y),& \qquad k\mapsto \lambda_Y\circ (k\otimes\id_{Y})\circ\alpha_{X,Y^*,Y}^{-1}\circ(\id_X\otimes\name{\id_Y})\circ\rho_X^{-1}.
\end{align*}
\end{lemma}
\begin{proof}
The existence of the bijections between the homsets is a basic result in the theory of compact-closed categories, and is claimed in for instance \cite{Kelly-Laplaza}. 
\end{proof}
    Let $(\mathbf C,\otimes,I)$ be a compact closed category. For each morphism $f:X\to Y$, define $f^*:Y^*\to X^*$ to be the morphism
    \[ f^*=\rho_X\circ (\id_{X^*}\otimes\epsilon_Y)\circ(\id_{X^*}\otimes
    (f\otimes\id_{Y^*}))\circ\alpha_{X^*,X,Y^*}\circ(\eta_X\otimes\id_{Y^*})\circ\lambda_{Y^*}^{-1}.\]
      Then the assignment $X\mapsto X^*$ on objects becomes a functor $\mathbf C\to\mathbf C^\op$ by defining its action on morphisms $f:X\to Y$ by $f\mapsto f^*$. Moreover, the functors $\id_\mathbf C:\mathbf C\to\mathbf C$ and $(-)^{**}:\mathbf C\to\mathbf C$ are natural isomorphic.

\subsection{Dagger compact categories}

    A \emph{dagger symmetric monoidal category} is a symmetric monoidal category $(\mathbf C,\otimes, I)$ that is also a dagger category in such a way that $(f\otimes g)^\dag=f^\dag\otimes g^\dag$ for all morphisms $f$ and $g$, and such that the associator, unitors and symmetry are unitaries. If, in addition, $(\mathbf C,\otimes,I)$ is compact and satisfies $\sigma_{A,A^* }\circ \epsilon_A^\dag=\eta_{A}$, then we call $(\mathbf C,\otimes,I)$ a \emph{dagger compact category}.

\begin{definition}\label{def:dagger strong monoidal functor}
    Let $(\mathbf C,\otimes,I)$ and $(\mathbf D,\odot,J)$ be  symmetric monoidal categories. Then a \emph{symmetric monoidal functor} $F:\mathbf D\to\mathbf C$ is a functor $F:\mathbf D\to\mathbf C$ equipped with coherence morphisms $\varphi:I\to FJ$ and $\varphi_{X,Y}:FX\otimes FY\to F(X\odot Y)$ for each $X,Y\in\mathbf D$ satisfying the following axioms:
    
\begin{align}\label{eq:coherence1}
F\alpha_{X,Y,Z}\circ\varphi_{X\odot Y,Z}\circ(\varphi_{X,Y}\otimes\id_{FZ}) & =   \varphi_{X,Y\odot Z}\circ(\id_{FY}\otimes\varphi_{Y,Z})\circ\alpha_{FX,FY,FZ}\\
    \label{eq:coherence2}
    F\lambda_Y\circ\varphi_{J,Y}\circ (\varphi\otimes\id_{FY}) & = \lambda_{FY}\\
    \label{eq:coherence3}
     F\rho_X\circ\varphi_{X,J}\circ(\id_{FX}\otimes\varphi) & = \rho_{FX}\\
     \label{eq:coherence4}
     F\sigma_{X,Y}\circ\varphi_{X,Y}& =\varphi_{Y,X}\circ\sigma_{FX,FY}.
    \end{align}
for objects $X,Y,Z$ of $\mathbf D$. If all coherence morphisms are isomorphisms, we call $F$ \emph{strong}. If, in addition, $(\mathbf C,\otimes,I)$ and $(\mathbf D,\odot,J)$ are dagger symmetric monoidal categories, and the coherence morphisms are unitaries, we call $F$ a \emph{dagger strong} symmetric monoidal functor. If all coherence morphisms are identities, then $F$ is called \emph{strict}.
    \end{definition}

\begin{lemma}\label{lem:dagger_and_dual}
    Let $f:X\to Y$ be a morphism in a dagger compact category $(\mathbf C,\otimes, I)$. 
    Then 
    \begin{align*}
        (f^*)^\dag & =(f^\dag)^*;\\
        \epsilon_Y\circ(f\otimes \id_{Y^*}) & = \epsilon_X\circ(\id_X\otimes f^*);\\
        (\id_{X^*}\otimes f)\circ \eta_X & =(f^*\otimes \id_Y)\circ \eta_Y.
          \end{align*}
\end{lemma}
\begin{proof}
For the first equality, see \cite[Lemma 3.55]{heunenvicary}. For the remaining equalities, see Equation (3.10) in Lemma 3.12 of \cite{heunenvicary}.
\end{proof}

Finally, dagger compact categories enjoy the property of having a trace.

\begin{definition}
    Let $(\mathbf C,\otimes,I)$ be a dagger compact category. For each object $X\in\mathbf C$, we denote the map $\mathbf C(X,X)\to\mathbf C(I,I)$, $f\mapsto\epsilon_X\circ (f\otimes\id_{X^*})\circ\epsilon_X^\dag$ by $\Tr_X(f)$, or simply by $\Tr(f)$.
\end{definition}

We record the following properties of the trace:
\begin{proposition}\cite[Lemmas 3.61 \& 3.63]{heunenvicary}\label{prop:trace properties}
  Let $(\mathbf C,\otimes,I)$ be a dagger compact category. Then:
  \begin{itemize}
      \item[(a)] $\Tr_I(s)=s$ for any scalar $s:I\to I$;
      \item[(b)] $\Tr_X(0_X)=0_I$ for any object $X$ of $\mathbf C$ if $\mathbf C$ has a zero object;
      \item[(c)] $\Tr_{X\otimes Y}(f\otimes g)=\Tr(f)_X\circ\Tr_Y(g)$ for any morphisms $f:X\to X$ and $g:Y\to Y$;
      \item[(d)] $\Tr_X(g\circ f)=\Tr_Y(f\circ g)$ for any morphisms $f:X\to Y$ and $g:Y\to X$.
  \end{itemize}
\end{proposition}

\begin{definition}
    Let $(\mathbf C,\otimes,I)$ be a dagger compact category. Then we define the \emph{dimension} $\dim(X)$ of an object $X$ of $\mathbf C$ to be the scalar $\Tr(\id_X)$.
\end{definition}

\subsection{Biproducts}

Given a zero object $0$ in a category $\mathbf C$, we denote by $0_{X,Y}$ the unique morphism $X\to Y$ that factors via $0$. If $X=Y$, we write often $0_{X}$ instead of $0_{X,X}$. Moreover, we define $\delta_{X,Y}:X\to Y$ to be the morphism in $\mathbf C$ given by
\[\delta_{X,Y}=\begin{cases}
    \id_X, & X=Y;\\
    0_{X,Y}, & X\neq Y.
\end{cases}\]
If $(X_\alpha)_{\alpha\in A}$ is a set-indexed family of objects in $\mathbf C$, we often will write $\delta_{\alpha,\beta}$ instead of $\delta_{X_\alpha,X_\beta}$ for $\alpha,\beta\in A$.

\begin{definition}\label{def:dagger biproducts}Let $\mathbf C$ be a category with a zero object $0$. We say that a set-indexed family $(X_\alpha)_{\alpha\in A}$ of objects in $\mathbf C$ has a \emph{biproduct} if there exists an object $\bigoplus_{\alpha\in A}X_\alpha$ and morphisms $p_{X_\beta}:\bigoplus_{\alpha\in A}X_\alpha\to X_\beta$ and $i_{X_\beta}:X_\beta\to \bigoplus_{\alpha\in A}X_\alpha$ such that:
\begin{itemize}
    \item $\bigoplus_{\alpha\in A}X_\alpha$ is the product of $(X_\alpha)_{\alpha\in A}$ with canonical projections $p_{X_\alpha}$;
    \item $\bigoplus_{\alpha\in A}X_\alpha$ is the coproduct of $(X_\alpha)_{\alpha\in A}$ with canonical injections $i_{X_\alpha}$;
    \item $p_{X_\beta}\circ i_{X_\alpha}=\delta_{X_\alpha,X_\beta}$ for each $\alpha,\beta\in A$.
\end{itemize}
If, in addition, $\mathbf C$ is a dagger category, we call $\bigoplus_{\alpha\in A}X_\alpha$ the \emph{dagger} biproduct of $(X_\alpha)_{\alpha\in A}$ if the following condition is satisfied:
\begin{itemize}
      \item $p_{X_\alpha}^\dag=i_{X_\alpha}$ for each $\alpha\in A$.
  \end{itemize}
If we only consider the biproduct of a single set-indexed family $(X_\alpha)_{\alpha\in A}$ of objects instead of biproducts of several families, we sometimes write $p_\alpha$, $i_\alpha$ and $\delta_{\alpha,\beta}$ instead of $p_{X_\alpha}$, $i_{X_\alpha}$ and $\delta_{X_\alpha,X_\beta}$, respectively, for $\alpha,\beta\in A$.

Given set-indexed families  $(X_\alpha)_{\alpha\in A}$ and $(Y_\alpha)_{\alpha\in A}$ of objects in a category $\mathbf R$ with small biproducts, and morphisms $f_\alpha:X_\alpha\to Y_\alpha$ for each $\alpha\in A$, we have $\prod_{\alpha\in A}f_\alpha=\coprod_{\alpha\in A}f_\alpha$, which we will denote by $\bigoplus_{\alpha\in A}f_\alpha$.
\end{definition}

\begin{definition}
We say that a (dagger) category $\mathbf C$ \emph{has small (dagger) biproducts} if it has a zero object and the (dagger) biproduct of any set-indexed family of objects in $\mathbf C$ exists. 
\end{definition}

\begin{definition}
Given an object $X$ of a category $\mathbf C$ with small biproducts, and given an index set $A$, we denote the morphisms $\langle \id_X\rangle_{\alpha\in A}:X\to\bigoplus_{\alpha\in A}X$ and $[\id_X]_{\alpha\in A}:\bigoplus_{\alpha\in A}X\to X$ by $\Delta_X^A$ and $\nabla_X^A$, respectively. If no confusion is possible, we drop subscripts and/or superscripts.
\end{definition}

The proofs of the following lemmas are straightforward.
\begin{lemma}\label{lem:projection and injection}
    Let $\mathbf R$ be a pointed category and let $(X_\alpha)_{\alpha\in A}$ be a set-indexed family of objects in $\mathbf R$ whose biproduct exists. Then for each $\beta\in A$, we have $p_\beta=[\delta_{\alpha,\beta}]_{\alpha\in A}$ and $i_\alpha=\langle \delta_{\alpha,\beta}\rangle_{\beta\in A}$
\end{lemma}

\begin{lemma}\label{lem:nabladeltabiproduct}
    Let $\mathbf C$ be a category with biproducts, let $(X_\alpha)_{\alpha\in A}$ be a set-indexed family of objects in $\mathbf C$, and let $X=\bigoplus_{\alpha\in A}X_\alpha$. Let $Y\in\mathbf C$, and for each $\alpha\in A$, let $f_\alpha:Y\to X_\alpha$ and $g_\alpha:X_\alpha\to Y$ be morphisms. Then
   $\langle f_\alpha\rangle_{\alpha\in A}=\left(\bigoplus_{\alpha\in A}f_\alpha\right)\circ\Delta$ and $[g_\alpha]_{\alpha\in A}=\nabla\circ\left(\bigoplus_{\alpha\in A}g_\alpha\right)$.
\end{lemma}

\subsection{Superposition rule}

    Let $\mathbf R$ be a category with all small biproducts. Given objects $X$ and $Y$ in $\mathbf R$ and a set-indexed family  $(f_\alpha)_{\alpha\in A}$ of morphisms $X\to Y$, we define the morphism $\sum_{\alpha\in A}f_\alpha:X\to Y$ by \[ \sum_{\alpha\in A}f_\alpha :=\nabla\circ\left(\bigoplus_{\alpha\in A}f_\alpha\right)\circ\Delta.\] 
    Furthermore, given $f_1,f_2\in\mathbf R(X,Y)$, we define $f_1+f_2:X\to Y$ by \[f_1+f_2:=\sum_{\alpha\in\{1,2\}}f_\alpha.\]

The first two properties in the next proposition express that homsets in a category $\mathbf R$ with all small biproducts form complete monoids in the sense of Laird \cite{Laird}, which is a generalization of the notion of $\Sigma$-monoids introduced by Haghverdi \cite{haghverdi_2000} to the uncountable case. Combined with the sixth property, these properties express that $\mathbf R$ is enriched over the category of complete monoids and sum-preserving maps, which is proven in \cite[Proposition 2.3]{Laird}. The proof of the remaining properties is straightforward; the essential steps are the same in the more familiar case of finitely-indexed families of morphisms. We note that complete monoids are also studied by Andr\'es-Mart\'inez and Heunen \cite{AndresMartinez-Heunen}.

 \begin{proposition}\label{prop:homsets are commutative infinitary monoids}
     Let $\mathbf R$ be a category with small biproducts, and let $X$ and $Y$ be objects of $\mathbf R$. Then for any set-indexed family $(f_\alpha)_{\alpha\in A}$ of morphisms $X\to Y$, we have
      \begin{itemize}
                 \item[(1)] $\sum_{\alpha\in A}f_\alpha=f_\beta$ if $A$ is the singleton $\{\beta\}$;
        \item[(2)] $\sum_{\alpha\in A}f_\alpha=\sum_{\beta\in B}\sum_{\alpha\in k^{-1}[\{\beta\}]}f_\alpha$ for each function $k:A\to B$;
        \item[(3)] $\sum_{\alpha\in A}f_\alpha=\sum_{\beta\in A}f_{k(\beta)}$
for each bijection $k:A\to A$;
\item[(4)] $\sum_{\alpha\in\emptyset}f_\alpha=0_{X,Y}$;
\item[(5)] $\sum_{\alpha\in A}f_\alpha=\sum_{\alpha\in A\setminus B}f_\alpha$ for each $B\subseteq A$ such that $f_\beta=0_{X,Y}$ for each $\beta\in B$;
\item[(6)] For each object $Z$ and morphism $g:Y\to Z$ and $h:Z\to X$, we have
     \begin{align*}g\circ \left(\sum_{\alpha\in A}f_\alpha \right) = \sum_{\alpha\in A}(g\circ f_\alpha), \qquad
     \left(\sum_{\alpha\in A}f_\alpha\right)\circ h = \sum_{\alpha\in A}(f_\alpha\circ h).
     \end{align*}
      \end{itemize}
 \end{proposition}

\begin{corollary}\label{cor:sum of ip is id}
    Let $\mathbf C$ be a category with all small biproducts, and let $(X_\alpha)_{\alpha\in A}$ be a collection of objects in $\mathbf C$. Then $\id_{\bigoplus_{\alpha\in A}X_\alpha}=\sum_{\alpha\in A}i_\alpha\circ p_\alpha$.
\end{corollary}
\begin{proof}
    For each $\beta\in A$ we have $p_\beta\circ \sum_{\alpha\in A}i_\alpha\circ p_\alpha=\sum_{\alpha\in A}p_\beta\circ i_\alpha\circ p_\alpha=\sum_{\alpha\in A}\delta_{\alpha,\beta}\circ p_\alpha=p_\beta$, whence we must have $\sum_{\alpha\in A}i_\alpha\circ p_\alpha=\id_{\bigoplus_{\alpha\in A}X_\alpha}$.
\end{proof}

\subsection{Matrices}

    Let $\mathbf R$ be a category with small biproducts. Let $(X_\alpha)_{\alpha\in A}$ and $(Y_\beta)_{\beta\in B}$ be collections of objects in $\mathbf R$, and for each $\alpha\in A$ and $\beta\in B$ let $f_{\alpha,\beta}$ be a morphism $X_\alpha\to Y_\beta$. Then we define the morphism  $(f_{\alpha,\beta})_{\alpha\in A,\beta\in B}:\bigoplus_{\alpha \in A}X_\alpha\to\bigoplus_{\beta\in B}Y_\beta$ by 
    \[ (f_{\alpha,\beta})_{\alpha\in A,\beta\in B}:=\sum_{\alpha\in A,\beta\in B}i_{Y_\beta}\circ f_{\alpha,\beta}\circ p_{X_\alpha}.\]
     For simplicity, we will sometimes write $(f_{\alpha,\beta})_{\alpha,\beta}$ instead of $(f_{\alpha,\beta})_{\alpha\in A,\beta\in B}$.
If $f=(f_{\alpha,\beta})_{\alpha\in A,\beta\in B}$, we will refer to $(f_{\alpha,\beta})_{\alpha\in A,\beta\in B}$ as the \emph{matrix} corresponding to $f$; the morphisms $f_{\alpha,\beta}$ are called \emph{matrix elements} of $f$.

The following lemma is an infinite version of Lemma 2.26 and Corollary 2.27 of \cite{heunenvicary}. Except for working with a possibly infinite index-set instead of a finite one, the proof is the same.
\begin{lemma}\label{lem:matrix elements}
Let $\mathbf R$ be a category with small biproducts, let $(X_\alpha)_{\alpha\in A}$ and $(Y_\beta)_{\beta\in B}$ be families of objects in $\mathbf R$. Then any morphism $f:\bigoplus_{\alpha\in A}X_\alpha\to \bigoplus_{\beta\in B}Y_\beta$ has a corresponding matrix, i.e., $f=(f_{\alpha,\beta})_{\alpha\in A,\beta\in B}$ with matrix elements \[f_{\alpha,\beta}:=p_{Y_\beta}\circ f\circ i_{X_\alpha}.\] Moreover, $f$ is uniquely determined by its matrix elements.
\end{lemma}

\begin{lemma}\label{lem:matrices,projections,injections}
    Let $\mathbf R$ be a category with small biproducts, let $(X_\alpha)_{\alpha\in A}$ and $(Y_\beta)_{\beta\in B}$ be families of objects in $\mathbf R$, and let $f:\bigoplus_{\alpha\in A}X_\alpha\to \bigoplus_{\beta\in B}Y_\beta$ be a morphism. Then
    \begin{itemize}
        \item[(a)] $p_{Y_\beta}\circ f=\sum_{\alpha\in A}f_{\alpha,\beta}\circ p_{X_\alpha}$;
        \item[(b)] $f\circ i_{X_\alpha}=\sum_{\beta\in B}i_{Y_\beta}\circ f_{\alpha,\beta}$.
    \end{itemize}
\end{lemma}
\begin{proof}
   By Lemma \ref{lem:matrix elements}, we have $f=\sum_{\alpha\in A,\beta\in B}i_{Y_\beta}\circ f_{\alpha,\beta}\circ p_{X_\alpha}$. The statements now follow directly from (6) of Proposition \ref{prop:homsets are commutative infinitary monoids} and from the definition of biproducts.
\end{proof}

The following lemma is an infinite version of \cite[Proposition 2.28]{heunenvicary}. Its practically the same.
\begin{lemma}\label{lem:matrix multiplication}
Let $\mathbf R$ be a category with small biproducts, and let $(X_{\alpha})_{\alpha\in A}$, $(Y_\beta)_{\beta\in B}$ and $(Z_\gamma)_{\gamma\in C}$ be collections of objects in $\mathbf R$. 
    Let $f:\bigoplus_{\alpha\in A}X_\alpha\to \bigoplus_{\beta\in B}Y_{\beta}$ and $g:\bigoplus_{\beta\in }Y_\beta\to \bigoplus_{\gamma\in C}Z_\gamma$ be morphisms with matrices $(f_{\alpha,\beta})_{\alpha\in A,\beta\in B}$ and $(g_{\beta,\gamma})_{\beta\in B,\gamma\in C}$, respectively. Then 
    \[ g\circ f=\left(\sum_{\beta\in B}g_{\beta,\gamma}\circ f_{\alpha,\beta}\right)_{\alpha\in A,\gamma\in C}.\]
    \end{lemma}

\begin{lemma}\label{lem:matrix identity}
    Let $\mathbf R$ be a category with all small biproducts, and let $(X_\alpha)_{\alpha\in A}$ be a collection of objects in $\mathbf R$. Then the $(\alpha,\beta)$-matrix entry of $\id_{\bigoplus_{\alpha\in A}X_\alpha}$ is given by 
   $(\id_{\bigoplus_{\alpha\in A}X_\alpha})_{\alpha,\beta}=\delta_{\alpha,\beta}$.
\end{lemma}
\begin{proof}
    By Lemma \ref{lem:matrix elements}, we have $(\id_{\bigoplus_{\gamma\in A}X_\gamma})_{\alpha,\beta}=p_{\beta}\circ\id_{\bigoplus_{\gamma\in A}X_\gamma}\circ i_{\alpha}=p_{\beta}\circ i_{\alpha}=\delta_{\alpha,\beta}$.
\end{proof}

\subsection{Dagger biproducts}

If a dagger category has small dagger biproducts, we can calculate the adjoint of matrices as follows.

\begin{proposition}\label{prop:dagger of matrix}
    Let $\mathbf R$ be a dagger category with small dagger biproducts. Let $f=(f_{\alpha,\beta})_{\alpha\in A,\beta\in B}:\bigoplus_{\alpha\in A}X_\alpha\to\bigoplus_{\beta\in B}Y_\beta$ be a morphism in $\mathbf R$. Then for each $\alpha\in A$ and each $\beta\in B$, we have
    $(f^\dag)_{\beta,\alpha}=(f_{\alpha,\beta})^\dag.$
\end{proposition}

\begin{lemma}\label{lem:dagger preserves dagger biproducts}
    Let $\mathbf R$ be a dagger category with all dagger biproducts. For any two families $(X_\alpha)_{\alpha\in A}$ and $(Y_\alpha)_{\alpha\in A}$, and for any set-indexed family of morphisms $(r_\alpha:X_\alpha\to Y_\beta)_{\alpha\in A}$, we have $\left(\bigoplus_{\alpha\in A}r_\alpha\right)^\dag=\bigoplus_{\alpha\in A}r_\alpha^\dag$.
\end{lemma}

\begin{proposition}\label{prop:dagger nabla and delta}
Let $\mathbf R$ be a dagger category with all small dagger biproducts. Let $Y$ be an object of $\mathbf R$, and let $(X_\alpha)_{\alpha\in A}$ be a set-indexed family of objects in $\mathbf R$ with dagger biproduct $X$. For each $\alpha\in A$, let $r_\alpha:X_\alpha\to Y$ be a morphism in $\mathbf R$, and let $r:=[r_\alpha]_{\alpha\in A}:X\to Y$. Then 
\begin{itemize}
    \item[(a)] $[r_\alpha]_{\alpha\in A}^\dag=\langle r_\alpha^\dag\rangle_{\alpha\in A}$;
    \item[(b)] $r\circ r^\dag=\sum_{\alpha\in A}r_\alpha\circ r_\alpha^\dag$;
    \item[(c)] $(r^\dag\circ r)_{\alpha,\beta}=r_\beta^\dag\circ r_\alpha$ for each $\alpha,\beta\in A$;
    \item[(d)] $\Delta^A_Y=(\nabla_Y^A)^\dag$.
\end{itemize} 
\end{proposition}

\subsection{Distributivity}

  A symmetric monoidal category $(\mathbf C,\otimes,I)$ is called \emph{infinitely distributive symmetric monoidal} if it has all small coproducts and for each object $X\in\mathbf C$ and each set-indexed family $(Y_\alpha)_{\alpha\in A}$ of objects in $\mathbf C$ the canonical morphism
    \[  [\id_X\otimes i_{Y_\alpha}]_{\alpha\in A}: \coprod_{\alpha\in A}(X\otimes Y_\alpha)\to X\otimes \coprod_{\alpha\in A}Y_\alpha   \]
    is an isomorphism. 

The following proposition is a standard result in category theory. 
\begin{proposition}\label{prop:monoidalcloseddistributive}
    Any symmetric monoidal closed category $(\mathbf C,\otimes, I)$ with all small coproducts is an infinitely distributive symmetric monoidal category. 
\end{proposition}

\begin{corollary}\label{cor:compact category is distributive}
    Any compact closed category with all small coproducts is infinitely distributive symmetric monoidal.
\end{corollary}

The next proposition is a generalization of \cite[Lemma 3.22]{heunenvicary}, and its proof is essentially the same.
\begin{proposition}\label{prop:tensor is enriched over commutative infinitary monoids}
    Let $(\mathbf R,\otimes,I)$ be an infinitely distributive symmetric monoidal category with all small biproducts. For each $X,Y,Z,W\in\mathbf R$, each morphism $f:X\to W$, and each set-indexed family $(g_{\alpha})_{\alpha\in A}$ of morphisms $Y\to Z$, we have
    \[f\otimes \sum_{\alpha\in A}g_\alpha= \sum_{\alpha\in A}f\otimes g_\alpha, \qquad  \left(\sum_{\alpha\in A}g_\alpha\right)\otimes f=\sum_{\alpha\in A}(g_\alpha\otimes f).\]
        \end{proposition}

\begin{lemma}\label{lem:distributivity scalar multiplication and sums}
    Let $(\mathbf R,\otimes,I)$ be an infinitely distributive symmetric monoidal category with small biproducts. For any set-indexed family $(s_\alpha)_{\alpha\in A}$ of scalars, and for any any morphism $f:X\to Y$ in $\mathbf R$, we have $\sum_{\alpha\in A}(s_\alpha\cdot f)=\left(\sum_{\alpha\in A}s_\alpha\right)\cdot f$.
\end{lemma}
\begin{proof}
\begin{align*} \left(\sum_{\alpha\in A}s_\alpha\right)\cdot f &  = \lambda_Y\circ \left(\left(\sum_{\alpha\in A}s_\alpha\right)\otimes f\right)\circ \lambda _X^{-1}=\lambda_Y\circ \left(\sum_{\alpha\in A}s_\alpha\otimes f\right)\circ\lambda_X^{-1}\\
& =\sum_{\alpha\in A}(\lambda_Y\circ (s_\alpha\otimes f)\circ\lambda_X^{-1})=\sum_{\alpha\in A}(s_\alpha\cdot f),
\end{align*}
where we used Proposition \ref{prop:tensor is enriched over commutative infinitary monoids} in the second equality, and Proposition \ref{prop:homsets are commutative infinitary monoids} in the penultimate equality.
\end{proof}

\subsection{Quantales}
Before we discuss quantaloids, we define quantales, which are partially ordered structures that can be regarded as quantum generalizations of locales, and they provide the right setting for describing graded or weighted notions of relation, such as fuzzy membership.
 \begin{definition}\label{def:quantales}
     A \emph{quantale} $(V,\cdot,e)$ consists of a complete lattice $V$ equipped with an associative operation $(x,y)\mapsto x\cdot y$, called the \emph{multiplication}, and  a neutral element $e$ for the multiplication such that 
     \[ \left(\bigvee_{\alpha\in A}x_\alpha\right)\cdot  y=\bigvee_{\alpha\in A}x_\alpha\cdot y, \qquad y\cdot\bigvee_{\alpha\in A}x_\alpha=\bigvee_{\alpha\in A}y\cdot x_\alpha\]
     for each set-indexed family $(x_\alpha)_{\alpha\in A}$ of elements in $V$ and each $y\in V$. We denote the least and greatest element of $V$ by $\perp$ and $\top$, respectively. If, in addition,  $V$ is equipped with a \emph{dagger}, i.e., a map $(-)^\dag:V\to V$ 
     such that for each $x,y\in V$:
     \begin{itemize}
         \item $x^{\dag\dag}=x$;
         \item $x\leq y$ implies $x^\dag\leq y^\dag$;\item $(x\cdot y)^\dag=y^\dag\cdot x^\dag$,
     \end{itemize}
     then we call $V$ a \emph{dagger quantale}.
     \end{definition}
In this work, quantales are always assumed to be unital, whereas some authors use the term `quantale' for the non-unital notion and treat the unital case as a special case. We also note that in the literature, dagger quantales are called $*$-quantales or \emph{involutive} quantales, see for instance \cite{Heymans-Stubbe-2}. Because daggers play a crucial role in this work, we decided to deviate from the standard terminology.

\begin{definition}
        We call a quantale $(V,\cdot,e)$:
     \begin{itemize} 
     \item \emph{trivial} if $V=1$;
     \item \emph{nontrivial} if $\top\neq\perp$;
     \item \emph{affine} or \emph{integral} if $e=\top$;
    \item \emph{commutative} if $x\cdot y=y\cdot x$ for each $x,y\in V$;
    \item \emph{idempotent} if $x\cdot x=x$ for each $x\in V$.
    \end{itemize}
    \end{definition}
    It is straightforward to see that a quantale $V$ is commutative if and only if the identity is a dagger on $V$.

\begin{example}\label{ex:quantales}When equipped with the order and multiplication inherited from the real numbers, $2:=\{0,1\}$ is a commutative, affine idempotent quantale, and $[0,1]$ is a commutative affine quantale. 
\end{example}

    The proofs of the following two lemmas are straightforward.
\begin{lemma}\label{lem:multiplication with bottom in quantale}
    Let $(V,\cdot,e)$ be a quantale. Then ${\perp}\cdot x=\perp =x\cdot{\perp}
    $ for each $x\in V$.
\end{lemma}
    
\begin{lemma}
    Let $(V,\cdot,e)$ be a quantale. Then $V$ is nontrivial if and only if $e\neq\perp$.
\end{lemma}

Frames are special cases of quantales as follows from the following well-known result (see for instance \cite[Corollary 2]{BorceuxBossche1986}, but note that the authors already assume idempotency in the definition of a quantale):
\begin{lemma}\label{lem:affine-idempotent-quantale-is-frame}
    An affine and idempotent quantale $(V,\cdot,e)$ is a frame with $x\wedge y=x\cdot y$ for each $x,y\in V$. In particular, $(V,\cdot,e)$ is commutative.
\end{lemma}

\subsection{Quantaloids}

Next, we review the definition of quantaloids and some basic properties.

    A \emph{quantaloid} is a category $\mathbf Q$  in which every homset is a complete lattice such that composition or morphisms preserves suprema in both arguments separately. A \emph{homomorphism of quantaloids} is a functor $F:\mathbf Q\to\mathbf R$ between quantaloids that preserves the suprema of parallel morphisms, i.e., for each set-indexed family $(f_\alpha)_{\alpha\in A}$ of morphisms in a homset $\mathbf Q(X,Y)$, we have $F(\bigvee_{\alpha\in A}f_\alpha)=\bigvee_{\alpha\in A}F(f_\alpha)$. If, in addition, there is a functor $G:\mathbf R\to\mathbf Q$ such that $F$ and $G$ form an equivalence of categories, then we call $F$ an \emph{equivalence of quantaloids}, and we say that $\mathbf Q$ and $\mathbf R$ \emph{equivalent} quantaloids.

We note that the functor $G$ is automatically an homomorphism of quantaloids. Indeed, if we denote the natural isomorphism $FG\to \id_{\mathbf R}$ by $\epsilon$, then the inverse of the bijection $\mathbf R(X,Y)\to\mathbf Q(GX,GY)$, $g\mapsto Gg$ is the map $\mathbf Q(GX,GY)\to\mathbf R(X,Y)$, $f\mapsto \epsilon_Y\circ Ff\circ\epsilon_X^{-1}$, which preserves suprema because $f\mapsto Ff$ preserves suprema for $F$ is a homomorphism of quantaloids, and because composition in $\mathbf R$ preserves suprema. It is a standard result that the inverse of a supremum-preserving bijection between complete lattices also preserves suprema, which shows that $G$ is a homomorphism of quantaloids.

The following example shows that quantaloids are obtained as oidificiations of quantales.
\begin{example}\label{ex:quantale-induced quantaloid}
    Given a quantale $V$, we denote by $\mathbf V$ the quantaloid with a single object $1$ whose single homset is given by $\mathbf{V}(1,1)=V$. Composition is defined by $\mathbf V(1,1)\times\mathbf V(1,1)\to\mathbf V(1,1)$, $(x,y)\mapsto x\cdot y$. Furthermore, we have $\id_1=e$.

    As a special case, obtain $\mathbf 2$ from the quantale $2=\{0,1\}$.
\end{example}

\begin{example}\label{ex:Sup}\cite[Section 2.1]{eklund2018semigroups}
    Let $\mathbf{Sup}$ be the category of complete lattices and supremum-preserving maps. Given two complete lattices $X$ and $Y$, the (external) homset $\mathbf{Sup}(X,Y)$ is a complete lattice when ordered pointwise. It is straightforward to show that composition in $\mathbf{Sup}$ preserves suprema in each argument separately, hence $\mathbf{Sup}$ is a quantaloid.
Moreover, $(\mathbf{Sup},\otimes,2)$ is a symmetric monoidal category, where $2$ is the two-point lattice $\{0,1\}$, and  for complete lattices $X$ and $Y$, we define $X\otimes Y:= \mathbf{Sup}(X,Y^\mathrm{op})^\mathrm{op}$. The internal hom is given by the external hom. 
As a consequence, $\mathbf{Sup}$ is enriched over itself \cite[Proposition 6.2.6]{borceux:handbook2}. In fact, quantaloids are precisely the categories that are enriched over $\mathbf{Sup}$.
\end{example}

Our next example generalizes the category $\mathbf{Rel}$ of sets and binary relations in a natural way. It is motivated by fuzzy mathematics, which requires binary relations to take truth values in a many-valued set rather than the usual two-point set $2$ described in Example \ref{ex:quantales}. In the next example, the many-valued sets of truth values is represented by a quantale $V$. Taking $V=2$ in the next Example yields $\mathbf{Rel}$.

\begin{example}\label{ex:V-valued relations}\cite{monoidaltopology}
Let $V$ be a quantale. Let $A$ and $B$ be sets. Then a function $r:A\times B\to V$ is called a $V$-\emph{valued relation} or simply a $V$-\emph{relation} from $A$ to $B$, in which case we write $r:A\sto B$.  Sets and $V$-valued relations form a category $V$-$\mathbf{Rel}$ if we define the composition $s\bullet r$ of $V$-valued relations $r:A\sto B$ and $s:B\sto C$ by
\[(s\bullet r)(\alpha,\gamma):=\bigvee_{\beta\in B}r(\alpha,\beta)\cdot s(\beta,\gamma),\] and the identity morphism on a set $A$ as the $V$-relation $e_A:A\sto A$ defined by \[e_A(\alpha,\alpha'): =\begin{cases}e, & \alpha=\alpha',\\ \perp, & \text{otherwise}.
\end{cases}\]
$V$-$\mathbf{Rel}$ becomes a quantaloid if we order parallel $V$-valued relations $r,s:A\sto B$ by  $r\leq s$ if and only if $r(\alpha,\beta)\leq s(\alpha,\beta)$ for each $\alpha\in A$ and $\beta\in B$. The supremum $\bigvee_{\kappa\in K}r_\kappa$ of any set-indexed family $(r_\kappa)_{\kappa\in K}$ of parallel $V$-valued relations $A\sto B$ is calculated via \[\left(\bigvee_{\kappa\in K}r_\kappa\right)(\alpha,\beta)=\bigvee_{\kappa\in K}r_\kappa(\alpha,\beta)\]
    for each $\alpha\in A$ and each $\beta\in B$.

\end{example}

The following example will be fundamental for the definition of a quantum set.
\begin{example}\label{ex:FdOS}
   We define the category $\mathbf{FdOS}_0$ as the category whose objects are finite-dimensional Hilbert spaces; any morphism $V:X\to Y$ between objects $X$ and $Y$ of $\mathbf{FdOS}_0$ is a subspace $V\subseteq B(X,Y)$. Since $X$ and $Y$ are finite-dimensional, so is $B(X,Y)$, whence such $V$ is closed, hence an operator space. Given another object $Z\in\mathbf{FdOS}_0$ and an operator space $W:Y\to Z$, the composition $W\cdot V:X\to Z$ of $V$ with $W$ is defined as the operator space $\mathrm{span}\{wv:w\in W,v\in V\}$.
$\mathbf{FdOS}_0$ becomes a quantaloid if we order its homsets by inclusion. The identity operator space $\id_X$ on the object $X$ is given by $\mathbb C1_X$, where $1_X:X\to X$ is the identity operator on $X$, so the identity morphism on $X$ regarded as object of the category $\mathbf{FdHilb}$ of finite-dimensional Hilbert spaces and linear maps. The supremum $\bigvee_{\alpha\in A}V_\alpha$ of a set-indexed family $(V_\alpha)_{\alpha\in A}$ of parallel operator spaces $X\to Y$ is given by $\mathrm{span}\left(\bigcup_{\alpha\in A}V_\alpha\right)$. We denote the full subcategory of $\mathbf{FdOS}_0$ of all nonzero finite-dimensional Hilbert spaces by $\mathbf{FdOS}$, which is also a quantaloid because it is a full subcategory of $\mathbf{FdOS}_0$.
\end{example}

The proof of the next lemma is straightforward.
\begin{lemma}\label{lem:fully faithful homomorphism of quantaloids}
Let $F:\mathbf{Q}\to\mathbf R$ be a faithful homomorphism of quantaloids. Then, for each $X$ and $Y$ in $\mathbf Q$, the map $F_{X,Y}:\mathbf Q(X,Y)\to\mathbf R(FX,FY)$, $f\mapsto Ff$ is an order embedding. If, in addition, $F$ is full, then $F_{X,Y}$ is an order isomorphism.
\end{lemma}

Since homsets of quantaloids are complete lattices, the following definition makes sense:

\begin{definition}
    Let $\mathbf Q$ be a quantaloid. For any two objects $X$ and $Y$, we denote the largest and least element of $\mathbf Q(X,Y)$ by $\top_{X,Y}$ and $\perp_{X,Y}$, respectively. We write $\top_X$ instead of $\top_{X,X}$, and $\perp_X$ instead of $\perp_{X,X}$.
\end{definition}
The proofs of the next lemmas are all straightforward if one uses that $\perp_{X,Y}=\bigvee\emptyset_{X,Y}$, where $\emptyset_{X,Y}$ denotes the empty subset of $\mathbf{Q}(X,Y)$ in a quantaloid $\mathbf Q$.
\begin{lemma}\label{lem:composition preserves least elements in quantaloids}
    Let $X$, $Y$, and $Z$ be objects in a quantaloid $\mathbf Q$, and let $f:X\to Y$ and $g:Y\to Z$ be morphisms in $\mathbf Q$. Then $\perp_{Y,Z}\circ f=\perp_{X,Z}$ and $g\circ\perp_{X,Y}=\perp_{X,Z}$.
\end{lemma}

\begin{lemma}\label{lem:quantaloid with zero}
Let $\mathbf Q$ be a quantaloid with a zero object $0$. Then for any two objects $X$ and $Y$, we have $0_{X,Y}=\perp_{X,Y}$.
\end{lemma}

\begin{lemma}\label{lem:quantaloid homomorphism preserves zero}
  Let $F:\mathbf Q\to\mathbf R$ be a homomorphism of quantaloids, and let $X$ and $Y$ be objects in $\mathbf Q$. Then  $F(\perp_{X,Y})=\perp_{FX,FY}$. If, in addition, both $\mathbf Q$ and $\mathbf R$ have a zero object, then we have $F(0_{X,Y})=0_{FX,FY}$.  
\end{lemma}

\subsection{Dagger quantaloids}
When a mathematical object is endowed with multiple structures, these structures often required to be mutually compatible. For instance, a topological group is not just a group with a topology, but one also requires that the group operations are continuous. Another example is the definition of a dagger compact category above, where the unit and counit of the compact structure are required to be related to each other via the dagger operation. In the same way, we aim to describe how to combine the concepts of quantaloids and of dagger compact categories. We start with the combination of dagger categories and quantaloids:

\begin{definition}
     A \emph{dagger quantaloid} is a quantaloid $\mathbf Q$ that is at the same time a dagger category, such that for each two objects $X$ and $Y$ in $\mathbf Q$ the bijection
     \[ \mathbf Q(X,Y)\xrightarrow{\cong}\mathbf Q(Y,X),\qquad r\mapsto r^\dag\]is an order isomorphism.
\end{definition}
Note that since the dagger is involutive and a bijection on homsets, we could also have required that $(-)^\dag$ is monotone, or that it preserves arbitrary suprema.
In the literature, dagger quantaloids are often called $*$-\emph{quantaloids} or \emph{involutive quantaloids}, see for instance \cite{Heymans-Stubbe}.

\begin{example}\label{ex:commutative quantale induces dagger quantaloid}
Let $V$ be a dagger quantale. Then the quantaloid $\mathbf V$ from Example \ref{ex:quantale-induced quantaloid} is a dagger quantaloid, where the dagger on $\mathbf V(1,1)\to\mathbf V(1,1)$, $v\mapsto v^\dag$ is the dagger $V\to V$, $v\mapsto v^\dag$ on $V$.
\end{example}

\begin{example}\label{ex:V-Rel-dagger}
    Let $V$ be a dagger quantale. Then $V$-$\mathbf{Rel}$ (cf. Example \ref{ex:V-valued relations}) is a dagger quantaloid where the dagger $r^\dag:B\sto A$ of a $V$-valued relation $r:A\sto B$ is the $V$-valued relation $(\beta,\alpha)\mapsto r(\alpha,\beta)^\dag$. 
    
We are in particular interested in the case that $V$ is commutative. If, in addition, $V$ is affine, but not a frame, then $V$-$\mathbf{Rel}$ is not an allegory. Indeed, it follows from Lemma \ref{lem:affine-idempotent-quantale-is-frame} that $V$ cannot be idempotent, so there must be some $v\in V$ such that $v\cdot v\neq v$. Since $V$ is affine, we have $v\leq e$, hence $v\cdot v\leq v\cdot e=v$, hence, we must have $v\nleq v\cdot v$. Now, we cannot have $v\leq v\cdot v\cdot v$, because otherwise $v\leq v\cdot v\cdot v\leq v\cdot v\cdot e=v\cdot v$. As a consequence, taking $r:1\sto 1$ in $V$-$\mathbf{Rel}$ given by $r(*,*)=v$, we cannot have $r\leq r\bullet r^\dag\bullet r$, which should hold in an allegory \cite[ Lemma A.3.2.1]{johnstone:elephant1}.
\end{example}

\begin{example}\label{ex:FdOS-dagger}
    $\mathbf{FdOS}$ and $\mathbf{FdOS}_0$ (cf. Example \ref{ex:FdOS}) become dagger quantaloids if for each morphism $V:X\to Y$ we define $V^\dag:Y\to X$ by $V^\dag=\{v^\dag:v\in V\}$, where $v^\dag:Y\to X$ denotes the adjoint of the operator $v:X\to Y$. 
\end{example}

\begin{definition}
    A \emph{homomorphism of dagger quantaloids} is a homomorphism of quantaloids $F:\mathbf Q\to\mathbf R$ between two dagger quantaloids $\mathbf Q$ and $\mathbf R$ that is also a dagger functor. If, in addition, there is a dagger functor $G:\mathbf R\to \mathbf Q$ such that $F$ and $G$ form an equivalence of dagger categories, we call $\mathbf Q$ and $\mathbf R$ \emph{equivalent} dagger quantaloids.
\end{definition}

In the quantaloid literature such as \cite{Stubbe}, biproducts are also called \emph{direct sums}, and can be characterized in the following way. 

\begin{proposition}\cite[Proposition 8.3]{Stubbe}\label{prop:biproducts in quantaloids}
Let $(X_\alpha)_{\alpha\in A}$ be a set-indexed family of objects in a quantaloid $\mathbf Q$. Let $X$ be an object of $\mathbf Q$. Then the following statements are equivalent:
    \begin{itemize}
        \item[(a)] $X$ is the product of $(X_\alpha)_{\alpha\in A}$ with canonical projections $p_\alpha:X\to X_\alpha$ for each $\alpha\in A$;
        \item[(b)] $X$ is the coproduct of $(X_\alpha)_{\alpha\in A}$ with canonical injections $i_\alpha:X_\alpha\to X$ for each $\alpha\in A$;
        \item[(c)] $X$ is the biproduct of $(X_\alpha)_{\alpha\in A}$ with canonical projections $p_\alpha:X\to X_\alpha$ and canonical injections $i_\alpha:X_\alpha\to X$ for each $\alpha\in A$;
        \item[(d)] For each $\alpha\in A$ there are morphisms $p_\alpha:X\to X_\alpha$, $i_\alpha:X_\alpha\to X$ such that 
        $\bigvee_{\alpha\in A}i_\alpha\circ p_\alpha=\id_X$ and such that $p_\beta\circ i_\alpha=\delta_{{\alpha},\beta}$ for each $\alpha,\beta\in A$,
    \end{itemize}
    in which case the following identities hold:
    \begin{itemize}
        \item $p_\beta=[\delta_{\alpha,\beta}]_{\alpha\in A}$ for each $\beta\in A$;
        \item $i_\alpha=\langle \delta_{\alpha,\beta}\rangle_{\beta\in A}$ for each $\alpha\in A$;
        \item $\langle f_\alpha\rangle_{\alpha\in A}=\bigvee_{\alpha\in A}i_\alpha\circ f_\alpha$ for each object $Y$ of $\mathbf Q$ and each family $(f_\alpha:Y\to X_\alpha)_{\alpha\in A}$ of morphisms;
        \item $[g_\alpha]_{\alpha\in A}=\bigvee_{\alpha\in A}g_\alpha\circ p_\alpha$ for each object $Y$ of $\mathbf Q$ and each family $(g_\alpha:X_\alpha\to Y)_{\alpha\in A}$ of morphisms. 
    \end{itemize}
   \end{proposition}

The next proposition is an infinite version of \cite[Lemma 2.21]{heunenvicary}, but the proof is essentially the same.
\begin{proposition}\label{prop:sums are suprema in quantaloids with biproducts}
     Let $\mathbf Q$ be a quantaloid with small biproducts. For objects $X$ and $Y$ in $\mathbf Q$, let $(f_\alpha)_{\alpha\in A}$ be  a set-indexed family in $\mathbf Q(X,Y)$. Then $\sum_{\alpha\in A}f_\alpha=\bigvee_{\alpha\in A}f_\alpha$.
\end{proposition}
 
With the previous two proposition, the proof of the next proposition is straightforward.
\begin{proposition}\label{prop:biproducts in quantaloids are monotone}
    Let $\mathbf Q$ be a quantaloid with small biproducts. Let $X\in\mathbf Q$, and let $(Y_\alpha)_{\alpha\in A}$, $(Z_\alpha)_{\alpha\in A}$ and $(W_\beta)_{\beta\in B}$ be set-indexed families of objects in $\mathbf Q$.
    \begin{itemize}
              \item[(a)] Given parallel morphisms $r_\alpha,s_\alpha:X\to Y_\alpha$ for each $\alpha\in A$, we have $r_\alpha\leq s_\alpha$ for each $\alpha\in A$ if and only if $\langle r_\alpha\rangle_{\alpha\in A}\leq \langle s_\alpha\rangle_{\alpha\in A}$;
        \item[(b)] Given parallel morphisms $r_\alpha,s_\beta:Y_\alpha\to X$ for each $\alpha\in A$, we have $r_\alpha\leq s_\alpha$ for each $\alpha\in A$ if and only if $[r_\alpha]_{\alpha\in A}\leq[s_\alpha]_{\alpha\in A}$.
         \item[(c)] Given parallel morphism $r_\alpha,s_\alpha:Y_\alpha\to Z_\alpha$ for each $\alpha\in A$, we have $r_\alpha\leq s_\alpha$ for each $\alpha\in A$ if and only if $\bigoplus_{\alpha\in A}r_\alpha\leq\bigoplus_{\alpha\in A}s_\alpha$;
         \item[(d)] Given parallel morphisms $r,s:\bigoplus_{\alpha\in A}Y_\alpha\to\bigoplus_{\beta\in B}W_\beta$, we have $r\leq s$ if and only if $r_{\alpha,\beta}\leq s_{\alpha,\beta}$ for each $\alpha\in A$ and each $\beta\in B$.
         \end{itemize}
     \end{proposition}

\begin{proposition}\label{prop:quantaloid homomorphisms preserve biproducts}
    Let $\mathbf Q$ and $\mathbf R$ both be (dagger) quantaloids with small (dagger) biproducts. If $F:\mathbf Q\to\mathbf R$ is a homomorphism of (dagger) quantaloids, then $F$ preserves (dagger) biproducts. 
\end{proposition}

\begin{proof}
This follows directly from the alternative characterization of biproducts in Proposition \ref{prop:biproducts in quantaloids}.
\end{proof}

\begin{proposition}\cite[Example 3.7]{Heymans-Stubbe-2}
Let $\mathbf Q$ be a quantaloid. Then we define a new quantaloid $\mathrm{Matr}(\mathbf Q)$ whose objects are set-indexed families $(X_\alpha)_{\alpha\in A}$ where $X_\alpha$ is an object of $\mathbf Q$ for each $\alpha\in A$. A morphism $f:X\to Y$ where $X=(X_\alpha)_{\alpha\in A}$ and $Y=(Y_\beta)_{\beta\in B}$ are objects of $\mathrm{Matr}(\mathbf Q)$ is a `matrix' of morphisms in $\mathbf Q$. To be more precise, $f$ is a set-indexed family $(f_\alpha^\beta)_{(\alpha,\beta)\in A\times B}$ where $f_\alpha^\beta:X_\alpha\to Y_\beta$ is a morphism in $\mathbf Q$. The composition with a morphism $g:Y\to Z$ where $Z=(Z_\gamma)_{\gamma\in C}$ is an object in $\mathrm{Matr}(\mathbf Q)$ is defined via
\[ (g\circ f)_\alpha^\gamma=\bigvee_{\beta\in B}g_\beta^\gamma\circ f_\alpha^\beta\]
for each $(\alpha,\gamma)\in A\times C$. For $\alpha,\beta\in A$, the $(\alpha,\beta)$-entry $(\id_X)_\alpha^\beta$ of the identity morphism $\id_X$ on $X$ is given by
\[ (\id_X)_\alpha^{\beta}= \begin{cases}\id_{X_\alpha}, & \alpha=\beta,\\
\perp_{X_\alpha,X_{\beta}}, & \alpha\neq\beta.
\end{cases}\]
We order parallel morphisms $f,g:X\to Y$ for objects $X=(X_\alpha)_{\alpha\in A}$ and $Y=(Y_\beta)_{\beta\in B}$ of $\mathrm{Matr}(\mathbf Q)$ by $f\leq g$ if and only if $(f_\alpha)^\beta\leq g_\alpha^\beta$ for each $(\alpha,\beta)\in A\times B$. Clearly, the supremum $\bigvee_{\gamma\in C}f_\gamma$ of any set-indexed family $(f_\gamma)_{\gamma\in C}$ of morphisms $X\to Y$ is then determined by
\[ \left(\bigvee_{\gamma\in C}f_\gamma\right)_\alpha^\beta=\bigvee_{\gamma\in C}(f_\gamma)_\alpha^\beta\]
for each $(\alpha,\beta)\in A\times B$.
Finally, if $\mathbf Q$ is a dagger quantaloid, then so is $\mathrm{Matr}(\mathbf Q)$, if for each morphism $f:(X_\alpha)_{\alpha\in A}\to (Y_\beta)_{\beta\in B}$ in $\mathrm{Matr}(\mathbf Q)$ we define $f^\dag$ by $(f^\dag)_{\beta}^\alpha:=(f_\alpha^\beta)^\dag$ for each $(\beta,\alpha)\in B\times A$.
\end{proposition}

\begin{example}\label{ex:qRel}
Recall the dagger quantaloid $\mathbf{FdOS}$ (cf. Examples \ref{ex:FdOS} and \ref{ex:FdOS-dagger}). We denote the quantaloid $\mathrm{Matr}(\mathbf{FdOS})$ by $\mathbf{qRel}$. Its objects are called \emph{quantum sets}, typically denoted by calligraphic letters $\X$, $\Y$, $\Z$. The index set of a quantum set $\X$ is typically denoted by $\At(\X)$, whose elements are called the \emph{atoms} of $\X$, so $\X=(X_\alpha)_{\alpha\in\At(\X)}$ For any natural number $n$, we denote the quantum set $(\CC)_{i=1}^n$ by $\mathbf n$. The morphisms of $\qRel$ are called \emph{binary relations between quantum sets}. $\mathbf{qRel}$ is no allegory, because its full subcategory $\mathbf{FdOS}$ is not an allegory, which can be seen as follows. Let $H$ be a finite-dimensional Hilbert space and let $a:H\to H$ be an invertible linear map that is not a unitary map. Then $a^\dag\neq a^{-1}$. For instance, let $H=\mathbb C^2$, and let $a=\begin{pmatrix}
    1 & 1\\
    0 & 1
\end{pmatrix}$. The inverse of $a$ is $a^{-1}=\begin{pmatrix}
    1 & -1\\
    0 & 1
\end{pmatrix}$, which clearly is not equal to $a^\dag$.  Let $V,W\subseteq B(H)$ be given by $V=\CC a$ and $W=\CC a^{-1}$. Then $V$ and $W$ are automorphisms of $H$ in $\mathbf{FdOS}$, and $V\cdot W=\CC 1_H$ and $W\cdot V=\CC 1_H$, so $W=V^{-1}$ in $\mathbf{FdOS}$. However, it follows from  \cite[Lemma A.3.2.3]{johnstone:elephant1} that  $f^{-1}=f^\dag$ for any invertible morphism $f$ in an allegory, and since $a^{-1}\neq a^\dag$, we have $V^{-1}\neq V^\dag$, so $\mathbf{FdOS}$ cannot be an allegory.
\end{example}

The matrix-like composition of $V$-valued relations suggest that $V$-$\mathbf{Rel}$ is the biproduct completion of a quantaloid. 

\begin{proposition}\label{prop:VRelisMatrV}
    Let $V$ be a (dagger) quantale. Define $F_V:V$-$\mathbf{Rel}\to\mathrm{Matr}(\mathbf V)$ by $A\mapsto (1)_{\alpha\in A}$ on objects and by $r\mapsto (r(\alpha,\beta))_{(\alpha,\beta)\in A\times B}$ on morphisms $r:A\sto B$. Define $G_V:\mathrm{Matr}(\mathbf V)\to V$-$\mathbf{Rel}$ to be the functor that is defined on objects by $(1)_{\alpha\in A}\mapsto A$ and that maps  morphisms $(f_\alpha^\beta)_{(\alpha,\beta)\in A\times B}:(1)_{\alpha\in A}\to(1)_{\beta\in B}$ to the $V$-valued relation $A\sto B$, $(\alpha,\beta)\mapsto f_\alpha^\beta$. Then $F_V$ and $G_V$ are homomorphisms of (dagger) quantaloids and form an isomorphism of categories, hence $V$-$\mathbf{Rel}$ and $\mathrm{Matr}(\mathbf V)$ are equivalent (dagger) quantaloids.
\end{proposition}

Recall that in $\mathbf{Rel}$, biproducts are given by $\biguplus_{\alpha\in A}A_\kappa=\left\{(\kappa,\alpha):\kappa\in K,\alpha\in A_\kappa\right\}$. Also $\mathrm{Matr}(\mathbf Q)$ has all small biproducts. In fact, we have:

\begin{theorem}\label{thm:biproduct-completion}\cite[p.43]{Stubbe}
    Let $\mathbf Q$ be a (dagger) quantaloid. Then $\mathrm{Matr}(\mathbf Q)$ is the universal (dagger) biproduct completion of $\mathbf Q$:
    \begin{itemize}
    \item The (dagger) biproduct $X$ of a set-indexed family $(X_\kappa)_{\kappa\in K}$ in $\mathrm{Matr}(\mathbf Q)$ with $X_\kappa=(X_{\kappa,\alpha})_{\alpha\in A_\kappa}$ is given by $(X_{\kappa,\alpha})_{(\kappa,\alpha)\in A}$ with $A=\bigoplus_{\kappa\in K}A_\kappa$. For each $\lambda\in K$, the canonical projection $p_\lambda:X\to X_\lambda$ and canonical injection $i_\lambda$ are given by 
\begin{align*}
    (p_\lambda)_{(\kappa,\alpha)}^\beta & = \begin{cases} \id_{X_{\kappa,\alpha}}, & \lambda=\kappa,\alpha=\beta,\\
    \perp_{X_{\kappa,\alpha},X_{\lambda,\beta}}, & \text{otherwise},
    \end{cases}
    \\
    (i_\lambda)_\beta^{(\kappa,\alpha)} & = \begin{cases} \id_{X_{\kappa,\alpha}}, & \lambda=\kappa,\alpha=\beta,\\
    \perp_{X_{\lambda,\beta},X_{\kappa,\alpha}}, & \text{otherwise},
    \end{cases}
\end{align*}
    for each $(\kappa,\alpha)\in A$ and each $\beta\in A_\lambda$. 
    \item There is a fully faithful homomorphism of (dagger) quantaloids $E_\mathbf Q:\mathbf Q\to\mathrm{Matr}(\mathbf Q)$ sending each $X$ of $\mathbf Q$ to the family $(X_\alpha)_{\alpha\in 1}$ with $X_*=X$, and which sends each morphism $f:X\to Y$ to $(f_\alpha^\beta)_{(\alpha,\beta)\in 1\times 1}$ with $f_*^*=f$. 
    \item Given any other (dagger) quantaloid $\mathbf R$ with small (dagger) biproducts, and given any homomorphism of (dagger) quantaloids $F:\mathbf Q\to\mathbf R$, there is a  homomorphism of (dagger) quantaloids $\overline F:\mathrm{Matr}(\mathbf Q)\to\mathbf R$ that is unique up to natural isomorphism such that
    the following diagram commutes up to natural isomorphism:
 \[ \begin{tikzcd}
 \mathbf Q\ar{r}{E_\mathbf Q}\ar{dr}[swap]{F} & \mathrm{Matr}(\mathbf Q)\ar{d}{\overline F}
  \\
  & \mathbf R,
\end{tikzcd} \]
 namely the functor that maps any object $(X_\alpha)_{\alpha\in A}$ in $\mathrm{Matr}(\mathbf Q)$ to $\bigoplus_{\alpha\in A}X_\alpha$ in $\mathbf{R}$ and that maps morphisms $(f_\alpha^\beta)_{(\alpha,\beta)\in A\times B}:(X_\alpha)_{\alpha\in A}\to(Y_\beta)_{\beta\in B}$ to $\bigvee_{\alpha\in A,\beta\in B}i_{Y_\beta}\circ f_\alpha^\beta\circ p_{X_\alpha}$.
    \item If $\mathbf Q$ already has all small (dagger) biproducts, then $\mathbf Q$ and $\mathrm{Matr}(\mathbf Q)$ are equivalent (dagger) quantaloids via $E_\mathbf Q$ and $P_{\mathbf Q}:=\overline{\id_{\mathbf Q}}$.
    \item $\mathrm{Matr}$ is functorial: for each homomorphism of (dagger) quantaloids $F:\mathbf Q\to\mathbf R$, we define $\mathrm{Matr}(F):=\overline{E_\mathbf R\circ F}$. Explicitly, $\mathrm{Matr}(F)$ maps objects $(X_\alpha)_{\alpha\in A}$ in $\mathrm{Matr}(\mathbf Q)$ to $(F X_\alpha)_{\alpha\in A}$ in $\mathrm{Matr}(\mathbf R)$, and morphisms $(f_{\alpha}^\beta)_{(\alpha,\beta)\in A\times B}:(X_\alpha)_{\alpha\in A}\to(Y_\beta)_{\beta\in B}$ in $\mathrm{Matr}(\mathbf Q)$ to $(Ff_\alpha^\beta)_{(\alpha,\beta)\in A\times B}$ in $\mathrm{Matr}(\mathbf R)$. The homomorphism $\mathrm{Matr}(F)$ is faithful if $F$ is faithful and full if $F$ is full.
    \end{itemize}
\end{theorem}

\begin{example} \label{ex:biproducts in V-Rel}
Since for any (dagger) quantale $V$, the quantaloid $\mathrm{Matr}(\mathbf V)$ has all small (dagger) biproducts, it follows from Proposition~\ref{prop:VRelisMatrV} that the category $V$-$\mathbf{Rel}$ likewise has all small (dagger) biproducts. Specifically, 
the (dagger) biproduct of a set-indexed family $(A_\kappa)_{\kappa\in K}$ of sets is the disjoint union $A:=\biguplus_{\kappa\in K}A_\kappa$, so $A=\{(\kappa,\alpha):\kappa\in K, \alpha\in A\}$. For each $\kappa\in K$, the canonical injection $i_\kappa:A_\kappa\sto A$ and the canonical projection $p_\kappa:A\sto A_\kappa$ are the $V$-relations given by
    \begin{align*}
    i_\kappa(\alpha,\beta) & = \begin{cases} e, & \beta=(\kappa,\alpha),\\
    \perp, & \textrm{otherwise},
    \end{cases}\\
    p_\kappa(\beta,\alpha) &  =  \begin{cases} e, & \beta=(\kappa,\alpha),\\
    \perp, & \textrm{otherwise},
    \end{cases}
        \end{align*}
        for each $\alpha\in A_\kappa$ and each $\beta\in A$.    
        \end{example}

\subsection{Dagger kernels}\label{sec:dagger kernel}
 Let $\mathbf R$ be a dagger category. If the equalizer $e$ of two parallel morphisms $f,g:X\to Y$ exists and can be chosen to be a dagger mono, then we call it a \emph{dagger equalizer}. In addition, assume that $\mathbf R$ has a zero object. If the dagger equalizer $k:K_f\to X$ of a morphism  $f:X\to Y$ and $0_{X,Y}$ exists, then we call $k$ a \emph{dagger kernel} of $f$, in which case we write $\ker(f):=k$. Sometimes we just write $K$ instead of $K_f$. If every morphism in $\mathbf R$ has a dagger kernel, we call $\mathbf R$ a \emph{dagger kernel category} or we say that $\mathbf R$ has dagger kernels. 

We note that to show that a dagger mono $k:K\to X$ is the dagger equalizer of some morphisms $f,g:X\to Y$ in a dagger category, it suffices to show that for each morphism $m:Z\to K$ with $f\circ m=g\circ m$, there is some morphism $h:Z\to K$ with $k\circ h=m$. Uniqueness of $h$ follows automatically because $k$ is a dagger mono. 

The notion in the following  definition is originally due to Heunen and Jacobs \cite{heunenjacobs}.
\begin{definition}\label{def:delta}
    Let $\mathbf C$ be a category with a zero object. A morphism $m:Y\to Z$ in $\mathbf C$ is called a \emph{zero-mono} if $m\circ f=0_{X,Z}$ implies $f=0_{X,Y}$ for each object $X$ of $\mathbf C$ and each morphism $f:X\to Y$.  Dually, a morphism $e:X\to Y$ is called a \emph{zero-epi} if for each morphism $f:Y\to Z$ we have that $f\circ e=0_{X,Z}$ implies $f=0_{Y,Z}$.
\end{definition}

The next lemma is a slight variation of \cite[Lemma 4]{heunenjacobs}, providing an alternative description of zero-monos in dagger kernel categories.
\begin{lemma}\label{lem:alternative-description-zero-monos}
    Let $\mathbf R$ be a dagger category with a zero object $0$. Then a morphism $m:X\to Y$ is a zero-mono if and only if its dagger kernel $\ker(m)$ exists and is the morphism $0_{0,X}:0\to X$.
\end{lemma}

Let $X$ be an object in a dagger kernel category. Two monomorphisms $m_1:S_1\to X$ and $m_2:S_2\to X$ are called \emph{equivalent} if there is some isomorphism $f:S_1\to S_2$ such that $m_1=m_2\circ f$, in which case we write $m_1\sim m_2$. Then $\sim$ is an equivalence relation; an equivalence class of a monomorphism $m:S\to X$ under $\sim$ is called a subobject of $X$, and is denoted by $[m]$. Since we assume all our categories to be wellpowered, the class $\mathrm{Sub}(X)$ of subobjects of $X$ is a set, and is actually a poset if we ordered it via $[m_1]\leq[m_2]$ if there is some morphism $f:S_1\to S_2$ such that $m_1=m_2\circ f$ for monomorphisms $m_1:S_1\to X$ and $m_2:S_2\to X$.

By definition, any dagger kernel $k:K\to X$ is a monomorphism, so a representative of a subobject of $X$. Let $k_1:K_1\to X$ and $k_2:K_2\to X$ be dagger kernels such that $[k_1]\leq[k_2]$ in $\mathrm{Sub}(X)$, so there is some morphism $f:K_1\to K_2$ such that $k_1=k_2\circ f$. It is shown in \cite[Lemma 1]{heunenjacobs} that $f$ is a dagger kernel, so in particular it is a dagger mono. We write $k_1\sim k_2$ if $k_1$ and $k_2$ are equivalent monomorphisms with codomain $X$, so $k_1=k_2\circ m$ for some isomorphism $m:K_1\to K_2$, in which case it immediately follows that  $m^{-1}=m^\dag\circ m\circ m^{-1}=m^\dag$, so $m$ is a dagger isomorphism. The set $\mathrm{KSub}(X)$ of equivalence classes of dagger kernels with codomain $X$ under $\sim$ is contained in $\mathrm{Sub}(X)$, and becomes a poset when equipped with the order inherited from $\mathrm{Sub}(X)$.

We will use the following lemma several times.
\begin{lemma}\cite[Proposition 7]{heunenjacobs}\label{lem:daggerkernel-zeroepi-factorization}
    Let $f:X\to Y$ be a morphism in a dagger kernel category $\mathbf C$, and let $k:K\to Y$ be the dagger kernel of $f^\dag:Y\to X$. Then $f=\ker(k^\dag)\circ e$ for some zero-epi $e:X\to K$.
\end{lemma}

Dagger kernels have the following nondegeneracy property:

\begin{lemma}\cite[Lemma 2.49]{heunenvicary}\label{lem:nondegeneracy}
Let $\mathbf C$ be a dagger kernel category. Then for each morphism $f:X\to Y$, we have $f^\dag\circ f=0_X$ if and only if $f=0_{X,Y}$.
    \end{lemma}

\begin{proposition}\label{prop:KSub is OML}\cite[Lemma 1, Lemma 2, Proposition 1]{heunenjacobs}
    Let $\mathbf R$ be a dagger kernel category. Then $\mathrm{KSub}(X)$ is an orthomodular lattice if we define the orthocomplement $\neg[k]$ of $[k]$ for a dagger kernel $k:K\to X$ by $\neg [k]=[k_\perp]$, where $k_\perp:=\ker(k^\dag)$, whose domain is denoted by $K^\perp$. The pullback $K$ of any two dagger kernels $k_1:K_1\to X$ and $k_2:K_2\to X$ exists, and if $k:K\to X$ denotes the induced map by composing the  the pullback maps with $k_1$ and $k_2$, then $k$ is a dagger kernel such that $[k_1]\wedge[k_2]=[k]$. Moreover, $[k_1]\perp[k_2]$ if and only if $k_1^\dag\circ k_2=0_{K_2,K_1}$.
\end{proposition}

\section{Symmetric monoidal and compact quantaloids}\label{sec:monoidal-quantaloids}

As far as we know, quantaloids with a monoidal structure have never been investigated before. We propose the following definition of symmetric monoidal category that is also a quantaloid such that the quantaloid structure interacts with the monoidal structure:

\begin{definition}
   A \emph{symmetric monoidal quantaloid} is a symmetric monoidal category $(\mathbf Q,\otimes, I)$ for which $\mathbf Q$ is a quantaloid such that the map $\mathbf Q(X,W)\times\mathbf Q(Y,Z)\to\mathbf Q(X\otimes Y,W\otimes Z)$, $(f,g)\mapsto f\otimes g$ preserves suprema in both arguments separately. If, in addition, 
   \begin{itemize}
       \item $(\mathbf Q,\otimes,I)$ is a dagger symmetric monoidal category and $\mathbf Q$ is a dagger quantaloid, then we call $(\mathbf Q,\otimes,I)$ a \emph{dagger} symmetric monoidal quantaloid;
       \item $(\mathbf Q,\otimes,I)$ is compact, then we call it a compact quantaloid;
       \item $(\mathbf Q,\otimes,I)$ is dagger compact, then we call it a dagger compact quantaloid. 
       \end{itemize}
Given two (dagger) symmetric monoidal quantaloids $(\mathbf Q,\odot,J)$ and $(\mathbf R,\otimes,I)$, a \emph{homomorphism of (dagger) symmetric monoidal quantaloids} $F:(\mathbf Q,\otimes,I)\to(\mathbf R,\odot, J)$ is a homomorphism of (dagger) quantaloids $F:\mathbf Q\to\mathbf R$ that is also a (dagger) strong symmetric monoidal functor (cf. Definition \ref{def:dagger strong monoidal functor}). If, in addition, there is a (dagger) functor $G:\mathbf R\to\mathbf Q$ such that $F$ and $G$ form an equivalence of (dagger) categories, we call $\mathbf Q$ and $\mathbf R$ \emph{equivalent} (dagger) symmetric monoidal quantaloids.
\end{definition}
If $F$ is a homomorphism of (dagger) symmetric monoidal quantaloids, it is straightforward to see that $G$ is also a homomorphism of (dagger) symmetric monoidal quantaloids.

We thank the anonymous reviewer for the following observation: the category of quantaloids and homomorphisms of quantaloids has a cartesian monoidal structure, and becomes a $2$-category with natural transformations between homomorphisms of quantaloids as $2$-cells. We recall that a symmetric monoidal category is precisely a symmetric pseudomonoid in the $2$-category of categories equipped with the cartesian monoidal structure. In a similar way, a symmetric monoidal quantaloid is precisely a symmetric pseudomonoid in the $2$-category of quantaloids, and the preservation of suprema in both arguments separately emerges because of the cartesian monoidal structure of this $2$-category.

\begin{lemma}\label{lem:tensor preserves least elements in symmetric monoidal quantaloids}
    Let $(\mathbf Q,\otimes,I)$ be a symmetric monoidal quantaloid. Then for any objects $W,X,Y,Z$ of $\mathbf Q$ and any morphism $f:X\to Y$, we have $f\otimes\perp_{W,Z}=\perp_{X\otimes W,Y\otimes Z}$ and $\perp_{W,Z}\otimes f=\perp_{W\otimes X,Z\otimes Y}$.
\end{lemma}
\begin{proof}
Straightforward if one uses that the monoidal product preserves suprema of parallel morphisms in both of its arguments and by using that $\perp_{W,Z}=\bigvee\emptyset_{W,Z}$ with $\emptyset_{W,Z}$ the empty subset of $\mathbf Q(W,Z)$.
\end{proof}

\subsection{Scalars and types of symmetric monoidal quantaloids}\label{subsec:scalars of quantaloids}
Given a symmetric monoidal quantaloid $(\mathbf Q,\otimes,I)$, it follows that its scalars $(\mathbf Q(I,I),\circ,\id_I)$ form a quantale. Moreover, by \cite[Lemma 2.3]{heunenvicary}, this quantale is commutative.  
Inspired by Definition \ref{def:quantales}, we say that $(\mathbf Q,\otimes, I)$ is 
\begin{itemize}
\item \emph{trivial} if $(\mathbf Q(I,I),\circ,\id_I)$ is a trivial quantale;
    \item \emph{nontrivial} if $(\mathbf Q(I,I),\circ,\id_I)$ is a nontrivial quantale;
    \item \emph{affine} if $(\mathbf Q(I,I),\circ,\id_I)$ is an affine quantale;
    \item \emph{binary} if $(\mathbf Q(I,I),\circ,\id_I)$ is isomorphic to the quantale $(2,\cdot,1)$ of Example \ref{ex:quantales}.
\end{itemize}

Recall that an affine quantale is also called \emph{integral}. There is already a notion of \emph{integral quantaloids}, but this differs from our definition of affine symmetric monoidal quantaloids, because identity morphisms of integral quantaloids are assumed to be the largest endomorphisms of any object. For affine quantaloids, this condition is only assumed for the monoidal unit. For our purposes, the condition of integral quantaloids is too strong. For instance, we will see in Example \ref{ex:FdOS-dagger-compact} that $\mathbf{FdOS}$ is affine, even binary, but it is not integral: the identity morphism $\id_X$ of a finite-dimensional Hilbert space of dimension $H$ higher than $1$ in $\mathbf{FdOS}$ is $\CC 1_X$, which is a one-dimensional operator space, whereas $\top_X=B(X)$, which has a dimension higher than 1, hence $\id_X\neq\top_X$. 

The next result justifies the terminology `trivial symmetric monoidal quantaloid.'

\begin{lemma}
    Let $(\mathbf Q,\otimes,I)$ be a trivial symmetric monoidal quantaloid. Then $\mathbf Q$ is equivalent to the trivial category, i.e., the category with one object and one morphism. 
\end{lemma}
\begin{proof}
Let $X$ and $Y$ be objects in $\mathbf Q$, and $f:X\to Y$ a morphism. Then it follows from Lemma \ref{lem:tensor preserves least elements in symmetric monoidal quantaloids} that $f\otimes\perp_{I}=\perp_{X\otimes I,Y\otimes I}$. Using the  naturality of the right unitor $\rho$, and Lemma \ref{lem:composition preserves least elements in quantaloids}, we have $f=\rho_Y\circ (f\otimes\id_I)\circ\rho_X^{-1}=\rho_Y\circ (f\otimes \perp_I)\circ \rho_X^{-1}=\rho_Y\circ \perp_{X\otimes I,Y\otimes I}\circ \rho_X^{-1}=\perp_{X,Y}$. 
We conclude that $\mathbf Q(X,Y)=1$ for any two objects, hence all objects of $\mathbf Q$ are mutually isomorphic, which implies that $\mathbf Q$ is equivalent to the trivial category.
\end{proof}

\begin{example}\label{ex:dagger compact quantaloid induced by quantale}
Let $(V,\cdot,e)$ be a commutative quantale. If we define $\otimes:\mathbf V\times\mathbf V\to \mathbf V$ by $1\otimes 1=1$ and $v\otimes w=v\cdot w$ for each $v,w\in V$, then $(\mathbf V,\otimes,1)$ is a strict symmetric monoidal category. Clearly, the symmetry of $\mathbf V$ follows from the commutativity of $V$. Since $\otimes$ clearly preserves suprema in both arguments, $(\mathbf V,\otimes,1)$ is a dagger symmetric monoidal quantaloid. It is straightforward to see that $(\mathbf V,\otimes,1)$ is in fact a dagger compact quantaloid if we define the unit $\eta_1:1\to 1\otimes 1$ and the counit $\epsilon_1:1\otimes 1\to 1$ to be the identity on $1$, i.e., the element $e\in V$. Clearly, $\mathbf V$ is nontrivial if and only if $V$ is nontrivial, affine if and only if $V$ is affine, and binary if and only if $V=2$.  
\end{example}

So the monoid structure of $V$ both induces the composition and the monoidal product of morphisms in $\mathbf V$. The Eckmann-Hilton argument prevents us from considering two different monoid structures on $V$, one inducing the composition, and the other the monoidal product of morphisms in $\mathbf V$. So the above example seems to be the most general way a quantale $V$ can be turned into a monoidal category that is a one-object quantaloid.

\begin{example}\label{ex:FdOS-dagger-compact}
    The quantaloids $\mathbf{FdOS}$ and $\mathbf{FdOS}_0$ from Example \ref{ex:FdOS} are dagger compact quantaloids. The monoidal product of $X\otimes Y$ of objects in both quantaloids is the usual tensor product of Hilbert spaces. The monoidal unit is $\mathbb C$. Given morphisms $V:X_1\to Y_1$ and $W:X_2\to Y_2$, we define $V\otimes W:X_1\otimes X_2\to Y_1\otimes Y_2$ as $\mathrm{span}\{v\otimes w:v\in V,w\in W\}$. 

    The associator $(X\otimes Y)\otimes Z\to X\otimes (Y\otimes Z)$ in both $\mathbf{FdOS}$ and $\mathbf{FdOS}_0$ is given by $\CC\alpha_{X,Y,Z}$, where $\alpha_{X,Y,Z}:(X\otimes Y)\otimes Z\to X\otimes (Y\otimes Z)$ is the associator in the category $\mathbf{FdHilb}$ of finite-dimensional Hilbert spaces and linear operators. The unitors and the symmetry of $\mathbf{FdOS}$ and $\mathbf{FdOS}_0$ are defined similarly, as are the unit and the counit of the compact structure of $\mathbf{FdOS}$ and $\mathbf{FdOS}_0$. Here, the dual $X^*$ of an object $X$ in both quantaloids is the usual Banach space dual of $X$.

Given finite-dimensional Hilbert spaces $X_1,X_2,Y_1,Y_2$ and operator spaces $V:X_1\to Y_1$ and $W:X_2\to Y_2$, it follows from $(v\otimes w)^\dag=v^\dag\otimes w^\dag$ for each $v\in V$ and $w\in W$ that $(V\otimes W)^\dag=V^\dag\otimes W^\dag$.
  Let $(V_\kappa)_{\kappa\in K}$ be a set-indexed family of operator spaces $X_1\to Y_1$.
Clearly, for each $\lambda\in K$, we have $V_\lambda\otimes W\subseteq\left(\bigvee_{\kappa\in K}V_\kappa\right)\otimes W$, whence $\bigvee_{\kappa\in K}(V_\kappa\otimes W)\subseteq\left(\bigvee_{\kappa\in K}V_\kappa\right)\otimes W$. 
Let $x\in \left(\bigvee_{\kappa\in K}V_\kappa\right)\otimes W$. Then there are $v_1,\ldots,v_n\in \bigvee_{\kappa\in K}V_\kappa$ and $w_1,\ldots,w_n\in W$ such that $x=v_1\otimes w_1+\ldots+v_n\otimes w_n$. For each $i\in\{1,\ldots,n\}$, there are $\kappa_{(i,1)},\ldots,\kappa_{(i,m_i)}\in K$ and $v_i^1\in V_{\kappa_{(i,1)}},\ldots, v_i^{m_i}\in V_{\kappa_{(i,m_i)}}$ such that $v_i=v_i^1+\ldots+v_i^{m_i}$. For $i\in\{1,\ldots,n\}$ and $j\in\{1,\ldots,m_i\}$, it follows that $x_{(i,j)}:=v_i^{j}\otimes w_1+\ldots+ v_i^{j}\otimes w_n\in V_{\kappa_{(i,j)}}\otimes W$. 
Since $x=\sum_{i=1}^n\sum_{j=1}^{m_i}x_{(i,j)}$ by linearity of the Hilbert space tensor product in each argument, it follows that $x\in \bigvee_{\kappa\in K}(V_\kappa\otimes W)$. We conclude that $\left(\bigvee_{\kappa\in K}V_\kappa\right)\otimes W=\bigvee_{\kappa\in K}(V_\kappa\otimes W)$, 
and similarly, one shows that $\otimes$ preserves suprema in the second argument. So $\mathbf{FdOS}$ and $\mathbf{FdOS}_0$ are indeed dagger compact quantaloids.

Finally, $\mathbf{FdOS}(\CC,\CC)$ is the space of all subspaces of $B(\CC)$. Since the latter is one-dimensional, there can only be two subspaces, namely $B(\CC)$ itself and the zero space $0$. Hence, $\mathbf{FdOS}$ is a binary dagger compact quantaloid. Since $\mathbf{FdOS}_0(\CC,\CC)=\mathbf{FdOS}(\CC,\CC)$, also $\mathbf{FdOS}_0$ is binary. 
\end{example}

\begin{theorem}\label{thm:distributive symmetric monoidal quantaloid}
    Let $(\mathbf Q,\otimes,I)$ be a symmetric monoidal category with small biproducts such that $\mathbf Q$ is a quantaloid. Then $(\mathbf Q,\otimes,I)$ is a symmetric monoidal quantaloid if and only if $(\mathbf Q,\otimes,I)$ is an infinitely distributive symmetric monoidal category. 
\end{theorem}
\begin{proof}
Assume that $(\mathbf Q,\otimes,I)$ is infinitely distributive. Then it follows directly from combining Propositions \ref{prop:tensor is enriched over commutative infinitary monoids} and \ref{prop:sums are suprema in quantaloids with biproducts} that it is a symmetric monoidal quantaloid. Conversely, assume that $(\mathbf Q,\otimes,I)$ is a symmetric monoidal quantaloid.
Let $X\in\mathbf Q$ be an object and let $(Y_\alpha)_{\alpha\in A}$ a family of objects in $\mathbf Q$. In order to show that $(\mathbf Q,\otimes,I)$ is infinitely distributive, we need to show that the canonical morphism
\[ \psi:=[\id_X\otimes i_{Y_\alpha}]_{\alpha\in A}:\bigoplus_{\alpha\in A}(X\otimes Y_\alpha)\to X\otimes\bigoplus_{\alpha\in A}Y_\alpha\]
is an isomorphism. Our candidate inverse is $\varphi:=\langle \id_X\otimes p_{Y_\alpha}\rangle_{\alpha\in A}$.
Using the identities in Proposition \ref{prop:biproducts in quantaloids}, we have
\begin{align*}
    \psi = \bigvee_{\alpha\in A}(\id_X\otimes i_{Y_\alpha})\circ p_{X\otimes Y_\alpha},\qquad \varphi  = \bigvee_{\alpha\in A}i_{X\otimes Y_\alpha}\circ (\id\otimes p_{Y_\alpha}).
\end{align*}
Then, using the identities for canonical projections and canonical injections of biproducts, and using that $\id_X\otimes(-)$ preserves suprema, which follows since $(\mathbf Q,\otimes,I)$ is a symmetric monoidal quantaloid, direct calculations yield $
    \psi\circ\varphi =\id_{X\otimes\bigoplus_{\alpha\in A}Y_\alpha}$ and $
    \varphi\circ\psi = \id_{\bigoplus_{\alpha\in A}X\otimes Y_\alpha}$, so  $\psi$ is an isomorphism.
\end{proof}

\begin{example}\cite[Section 2.1]{eklund2018semigroups}\label{ex:Sup-monoidal}
The quantaloid $(\mathbf{Sup},\otimes,2)$ of Example \ref{ex:Sup} is a symmetric monoidal closed quantaloid with small biproducts. This can be seen as follows. 
Given a collection  $(X_\alpha)_{\alpha\in A}$ of complete lattices, their set-theoretic product $\bigoplus_{\alpha\in A}X_\alpha$ is a complete lattice when ordered coordinate-wise. The canonical projections $p_\beta:\bigoplus_{\alpha\in A}X_\alpha\to X_\beta$ preserve all suprema, hence $\bigoplus_{\alpha\in A}X_\alpha$ is the product of $(X_\alpha)_{\alpha\in A}$.  Since $\mathbf{Sup}$ is a quantaloid, it follows from Proposition \ref{prop:biproducts in quantaloids} it has all small biproducts. Explicitly, the canonical injection $i_\beta: X_\beta\to\bigoplus_{\alpha\in A}$ is given by $x\mapsto (x_\alpha)_{\alpha\in A}$, where
\[ x_\alpha =\begin{cases}x, & \alpha=\beta,\\
\perp, & \alpha\neq\beta. \end{cases}\]
Since $(\mathbf{Sup},\otimes,2)$ is symmetric monoidal closed and has small biproducts, it follows from Proposition \ref{prop:monoidalcloseddistributive} that $\mathbf{Sup}$ is an infinitely distributive symmetric monoidal category. Hence by Theorem \ref{thm:distributive symmetric monoidal quantaloid}, it is a symmetric monoidal closed quantaloid. Finally, it is straightforward to see that $\mathbf{Sup}(2,2)\cong 2$, hence $\mathbf{Sup}$ is binary.
\end{example}

\begin{corollary}\label{cor:compact quantaloid with biproducts is monoidal quantaloid}
Let $(\mathbf Q,\otimes,I)$ be a compact-closed category with small biproducts such that $\mathbf Q$ is a quantaloid. Then $(\mathbf Q,\otimes,I)$ is a compact quantaloid.        
\end{corollary}
\begin{proof}
    This follows directly from combining Corollary \ref{cor:compact category is distributive} and Theorem \ref{thm:distributive symmetric monoidal quantaloid}.
    \end{proof}

\begin{proposition}\label{prop:dagger-compact quantaloid with dagger biproducts is dagger symmetric monoidal quantaloid}
    Let $(\mathbf Q,\otimes,I)$ be a dagger compact category with small dagger biproducts such that $\mathbf Q$ is a quantaloid. Then $(\mathbf Q,\otimes, I)$ is a dagger compact quantaloid. 
\end{proposition}
\begin{proof}
Let $X$ and $Y$ be objects in $\mathbf Q$ and let $r,s:X\to Y$ be morphisms. By Proposition \ref{prop:sums are suprema in quantaloids with biproducts} and Lemma \ref{lem:dagger preserves dagger biproducts} we have $r\leq s$ if and only if $r\vee s=s$ if and only if $r+s=s$ if and only if $r^\dag+s^\dag=s^\dag$ if and only if $r^\dag\vee s^\dag=s^\dag$ if and only if $r^\dag\leq s^\dag$. So the involution is an order embedding, which is also a bijection, hence it must be an order isomorphism. Thus,   $\mathbf Q$ is a dagger quantaloid. It remains to be proven that $(\mathbf Q,\otimes,I)$ is a dagger symmetric monoidal quantaloid, but this follows from Corollary \ref{cor:compact quantaloid with biproducts is monoidal quantaloid}.
\end{proof}

\begin{lemma}\label{lem:counit in compact quantaloid is order iso}
    Let $(\mathbf Q,\otimes,I)$ be a compact quantaloid. Then for any two objects $X$ and $Y$ in $\mathbf Q$, the following bijections are order isomorphisms: 
    \begin{align*}
        \mathbf Q(X,Y)\xrightarrow{\cong}\mathbf Q(I,X^*\otimes Y),& \qquad r\mapsto \name r,\\
           \mathbf Q(X,Y)\xrightarrow{\cong}\mathbf Q(X\otimes Y^*,I),& \qquad r\mapsto \coname r,\\
              \mathbf Q(X,Y)\xrightarrow{\cong}\mathbf Q(Y^*,X^*),& \qquad r\mapsto r^*.             
    \end{align*}
\end{lemma}
\begin{proof}
Since $\mathbf Q$ is a compact quantaloid, the operations $r\mapsto r\otimes s$ and $r\mapsto s\otimes r$ are monotone for any morphism $s$. Moreover, pre- and postcomposition with a fixed morphism are also monotone operations by definition of a quantaloid. Hence, by definition of $\name r$, $\coname r$ and $r^*$, all bijections in the statement are monotone maps. In the same way, it follows that the inverses of the first two bijections (cf. Lemma \ref{lem:epsilon}) are also monotone. Hence, the first two bijections are order isomorphisms. We show that the last bijection is an order isomorphism by showing that it is an order embedding, since a bijection order embedding is an order isomorphism.  Let $f,g\in\mathbf Q(X,Y)$. We already showed that the last bijection is monotone, $f^*\leq g^*$. Conversely, assume that $f^*\leq g^*$. Since from the last (monotone) bijection we can deduce that also $\mathbf Q(Y^*,X^*)\to\mathbf Q(X^{**},Y^{**})$, $h\mapsto h^*$ is a monotone bijection, it follows that $f^{**}\leq g^{**}$. 
Notate the natural isomorphism $\id_\mathbf Q\to (-)^{**}$ by $\delta$. Then it follows from naturality that $f=\delta_Y^{-1}\circ f^{**}\circ\delta_X$ and $g=\delta_Y^{-1}\circ g^{**}\circ\delta_X$. Hence, using that pre- and postcomposition in a quantaloid is monotone, we obtain $f=\delta_Y^{-1}\circ f^{**}\circ\delta_X\leq \delta_Y^{-1}\circ g^{**}\circ\delta_X=g$.
Thus also the last bijection is an order isomorphism.
\end{proof}

\begin{lemma}\label{lem:trace preserves suprema}
    Let $(\mathbf R,\otimes,I)$ be a dagger compact quantaloid. Then for any object $X$ of $\mathbf R$, the map $\Tr:\mathbf R(X,X)\to\mathbf R(I,I)$, $r\mapsto \Tr(r)$ preserves arbitrary suprema. 
\end{lemma}
\begin{proof}
Since $(\mathbf R,\otimes,I)$ is a dagger compact quantaloid, it is a symmetric monoidal quantaloid, hence the map $\mathbf R(X,X)\to\mathbf R(X\otimes X^*,X\otimes X^*)$, $r\mapsto r\otimes\id_{X^*}$ preserves suprema. Since $\Tr(r)=\epsilon_X\circ(r\otimes\id_{X^*})\circ \epsilon_X^\dag$, and both pre- and postcomposition in quantaloids preserve suprema, the statement follows.
\end{proof}

\subsection{Biproduct-induced quantaloid structure}

Next, we give a proof that any infinitely distributive symmetric monoidal category with small biproducts and precisely two scalars is a symmetric monoidal quantaloid. This proof is essentially the proof of \cite[Proposition 4.3]{Kornell23}. Our assumptions are slightly weaker, which causes the proof to be slightly different. We first need a lemma, which is adapted from \cite[Lemma 4.1]{Kornell23}.

\begin{lemma}\label{lem:invertible endomorphisms}
    Let $\mathbf R$ be a category with small biproducts. If $X$ is an object of $\mathbf R$ for which every nonzero endomorphism is invertible, then $(\mathbf R(X,X),+,0_X)$ is an idempotent commutative monoid. 
\end{lemma}
\begin{proof}
Let $R=\mathbf R(X,X)$. It easily follows from Proposition \ref{prop:homsets are commutative infinitary monoids} that $(R,+,0_X)$ is a commutative monoid. By assumption each nonzero morphism $f:X\to X$ has an inverse $f^{-1}$. Let $\omega=\sum_{i=1}^\infty\id_X$. Clearly, we have $\omega+\omega=\omega$. Assume first that $\omega=0_X$. Then $0_X=\omega=\omega+\id_X=0_X+\id_X=\id_X$, hence for each $f\in R$, we have $f=f\circ\id_X=f\circ 0_X=0_X$, i.e., $R=\{0_X\}$, which is trivially an idempotent commutative monoid. Next, assume that $\omega\neq 0_X$. Then $\omega$ is invertible, whence $\id_X+\id_X=\omega^{-1}\circ\omega+\omega^{-1}\circ\omega=\omega^{-1}\circ(\omega+\omega)=\omega^{-1}\circ\omega=\id_X$. It now follows for each $f\in R$ that $f+f=f\circ(\id_X+\id_X)=f\circ\id_X=f$, hence  $(R,+,0_X)$ is a commutative idempotent monoid. \qedhere
\end{proof}

\begin{theorem}\label{thm:biproducts imply quantaloid}
    Let $(\mathbf R,\otimes, I)$ be an infinitely distributive symmetric monoidal  category  with all small biproducts and with precisely two scalars $\id_I$ and $0_I$. Then $\mathbf R$ is a symmetric monoidal quantaloid where the supremum $\bigvee_{\alpha\in A}r_\alpha$ of a set-indexed family $(r_\alpha)_{\alpha\in A}$ of morphisms in a homset $\mathbf R(X,Y)$ is given by $\sum_{\alpha\in A}r_\alpha$. 
\end{theorem}
\begin{proof}
    Let $X$ and $Y$ be objects in $\mathbf R$, and let $r\in \mathbf R(X,Y)$. We have $\id_I\cdot r=\lambda_Y\circ (\id_I\otimes r)\circ \lambda_X^{-1}=r\circ\lambda_X\circ\lambda_X^{-1}=r$ by naturality of $\lambda$. Hence, for any nonempty set $A$, we have 
    \begin{equation}\label{eq:r-identity}
    \sum_{\alpha\in A}r=\sum_{\alpha\in A}(\id_I\cdot r)=\left(\sum_{\alpha\in A}\id_I\right)\cdot r, 
    \end{equation}
    where we used Lemma \ref{lem:distributivity scalar multiplication and sums} in the last equality.
In particular, for $A=\{1,2\}$, we obtain $r+r=(\id_I+\id_I)\cdot r$. Now, since $\id_I$ is the only nonzero scalar, which is clearly invertible, it follows from Lemma \ref{lem:invertible endomorphisms} that $(\mathbf R(I,I),+,0_I)$ is an idempotent commutative monoid, so $\id_I+\id_I=\id_I$. Thus $r+r=\id_I\cdot r=r$, hence $(\mathbf R(X,Y),+,0_{X,Y})$ is an idempotent commutative monoid. It is well known that such a monoid is a join-semilattice with $r\vee s=r+s$ for each $r,s:X\to Y$. Hence, $r\leq s$ if and only if $r\vee s=s$ if and only if $r+s=s$.
Let $(r_\alpha)_{\alpha\in A}$ be a set-indexed family of morphisms in $\mathbf R(X,Y)$. It immediately follows that $\sum_{\alpha\in A}r_\alpha$ is an upper bound for the family. As a consequence, we also obtain $\id_I\leq\sum_{\alpha\in A}\id_I$, and since $\id_I$ is clearly the largest element in $\mathbf R(I,I)=\{0_I,\id_I\}$, we must have $\sum_{\alpha\in A}\id_I=\id_I$. 

Assume $s$ is another upper bound of $(r_\alpha)_{\alpha\in A}$. By (\ref{eq:r-identity}), we obtain $\sum_{\alpha\in A}s=\left(\sum_{\alpha\in A}\id_I\right)\cdot s=\id_I\cdot s=s$.
Hence, for each $\alpha\in A$, we have $r_\alpha\leq s$, so $r_\alpha+s=s$, whence, $s=\sum_{\alpha\in A}s=\sum_{\alpha\in A}(r_\alpha+s)=\sum_{\alpha\in A}r_\alpha+\sum_{\alpha\in A}s=\sum_{\alpha\in A}r_\alpha+s$. Thus $\sum_{\alpha\in A}r_\alpha\leq s$, showing that $\bigvee_{\alpha\in A}r_\alpha=\sum_{\alpha\in A}r_\alpha$. It now follows from Proposition \ref{prop:homsets are commutative infinitary monoids} that $\mathbf R$ is enriched over $\mathbf{Sup}$, so it is a quantaloid. By Proposition \ref{prop:tensor is enriched over commutative infinitary monoids} also the monoidal product $\otimes$ on $\mathbf R$ is enriched over $\mathbf{Sup}$, so $\mathbf R$ is a symmetric monoidal quantaloid.
\end{proof}

\begin{theorem}\label{thm:dagger biproducts imply dagger compact quantaloid}
    Let $(\mathbf R,\otimes,I)$ be an infinitely distributive dagger symmetric monoidal category with small dagger biproducts and precisely two scalars. Then $\mathbf R$ is a dagger symmetric monoidal quantaloid, where the supremum $\bigvee_{\alpha\in A}f_\alpha$ of any set-indexed family $(f_\alpha)_{\alpha\in A}$ in any homset $\mathbf R(X,Y)$ is given by $\sum_{\alpha\in A}f_\alpha$. 
\end{theorem}
\begin{proof}
By Lemma \ref{lem:IisZero}, it follows that $\id_I\neq 0_I$, so the only non-zero scalar in $\mathbf R$ is invertible. By Theorem \ref{thm:biproducts imply quantaloid} it follows that $\mathbf R$ is a quantaloid and the supremum of morphisms in a homset is provided by taking their sums. 

In order to show that $\mathbf R$ is a dagger quantaloid, we have to show that  for each $X,Y\in\mathbf R$, the map $\mathbf R(X,Y)\to\mathbf R (Y,X)$, $r\mapsto r^\dag$ is an order isomorphism.
So let $r,s:X\to Y$. Using Proposition \ref{prop:dagger nabla and delta} and Lemma \ref{lem:dagger preserves dagger biproducts}, we find 
    $r^\dag\vee s^\dag=r^\dag+s^\dag=\nabla\circ(r^\dag\oplus s^\dag)\circ\Delta=\Delta^\dag\circ(r\oplus s)^\dag\circ\nabla^\dag=(\nabla\circ (r\oplus s)\circ\Delta)^\dag=(r+s)^\dag=(r\vee s)^\dag$.
    Hence, $r\leq s$ if and only if $s=r\vee s$ if and only if $s^\dag=(r\vee s)^\dag$ if and only if $s^\dag=r^\dag\vee s^\dag$ if and only if $r^\dag\leq s^\dag$. Hence, $\mathbf R(X,Y)\to\mathbf R(Y,X)$, $r\mapsto r^\dag$ is an order embedding. Since it is also a bijection, it is an order isomorphism. 

    Finally, we need to show that $(\mathbf R,\otimes,I)$ is a symmetric monoidal quantaloid, but this follows directly from Proposition \ref{prop:tensor is enriched over commutative infinitary monoids}.
    \end{proof}    

\subsection{Biproduct completion of monoidal and compact quantaloids}\ 

\noindent Let $(\mathbf Q,\otimes,I,\alpha,\lambda,\rho,\sigma)$ be a symmetric monoidal quantaloid. We define $\otimes:\mathrm{Matr}(\mathbf Q)\times\mathrm{Matr}(\mathbf Q)\to\mathrm{Matr}(\mathbf Q)$ by $X\otimes Y=(X_\alpha\otimes Y_\beta)_{(\alpha,\beta)\in A\times B}$ for objects $X=(X_\alpha)_{\alpha\in A}$ and $Y=(Y_\beta)_{\beta\in B}$ in $\mathrm{Matr}(\mathbf Q)$. If $W=(W_\gamma)_{\gamma\in C}$ and $Z=(Z_\delta)_{\delta\in D}$ are two other objects in $\mathrm{Matr}(\mathbf Q)$ and $f=(f_\alpha^\gamma)_{(\alpha,\gamma)\in A\times C}:X\to W$ and $g=(g_{\beta}^\delta)_{(\beta,\delta)\in B\times D}:Y\to Z$ morphisms in $\mathrm{Matr}(\mathbf Q)$, then we define 
\[(f\otimes g)_{(\alpha,\beta)}^{(\gamma,\delta)}:=f_{\alpha}^{\gamma}\otimes g_{\beta}^{\delta}\] for each $\alpha\in A$, $\beta\in B$, $\gamma\in C$ and $\delta\in D$.
We define $J\in\mathrm{Matr}(\mathbf Q)$ to be the object $(J_\alpha)_{\alpha\in 1}$ with $J_*=I$. 

For objects $X=(X_\beta)_{\beta\in B}$, $Y=(Y_\gamma)_{\gamma\in C}$, and $Z=(Z_\delta)_{\delta\in D}$, we define
\begin{align*}
\alpha_{X,Y,Z}&:(X\otimes Y)\otimes Z\to X\otimes (Y\otimes Z)\\
\lambda_X&:J\otimes X\to X\\
\rho_X&:X\otimes J\to X\\
\sigma_{X,Y}&:X\otimes Y\to Y\otimes X
\end{align*}
 by 
 \begin{align*}
 (\alpha_{X,Y,Z})_{((\beta,\gamma),\delta)}^{(\beta',(\gamma',\delta'))}& =\begin{cases}
\alpha_{X_\beta,Y_\gamma,Z_\delta}, & \beta=\beta',\gamma=\gamma',\delta=\delta',\\
\perp_{(X_\beta\otimes Y_\gamma)\otimes Z_\delta, X_{\beta'}\otimes(Y_{\gamma'}\otimes Z_{\delta'})}, & \text{otherwise},
\end{cases}     \\
(\lambda_X)_{(*,\beta)}^{\beta'}&=\begin{cases} \lambda_{X_\beta}, & \beta=\beta',\\
\perp_{I\otimes X_{\beta},X_{\beta'}}, & \text{otherwise},
\end{cases}\\
(\rho_X)_{(\beta,*)}^{\beta'}&=\begin{cases} \rho_{X_\beta}, & \beta=\beta',\\
\perp_{ X_{\beta}\otimes I,X_{\beta'}}, & \text{otherwise},
\end{cases}\\
(\sigma_{X,Y})_{(\beta,\gamma)}^{(\gamma',\beta')} & = \begin{cases}
    \sigma_{X_\beta,Y_\gamma}, & \beta=\beta',\gamma=\gamma',\\
    \perp_{X_\beta\otimes Y_\gamma,Y_{\gamma'}\otimes X_{\beta'}}, & \text{otherwise}.
\end{cases}
 \end{align*}

 \begin{proposition}\label{prop:biproduct completion monoidal quantaloid}
     Let $(\mathbf Q,\otimes,I)$ be a (dagger) symmetric monoidal quantaloid. Then $(\mathrm{Matr}(\mathbf Q),\otimes,J)$ as defined above is a (dagger) symmetric monoidal quantaloid that is nontrivial if and only if $\mathbf Q$ is nontrivial, affine if and only if $\mathbf Q$ is affine, and binary if and only if $\mathbf Q$ is binary. Moreover, if $(\mathbf Q,\otimes, I)$ is a dagger symmetric monoidal quantaloid with small dagger biproducts, then the functors $E_\mathbf Q:\mathbf Q\to\mathrm{Matr}(\mathbf Q)$ and $P_\mathbf Q:\mathrm{Matr}(\mathbf Q)\to\mathbf Q$ in Theorem \ref{thm:biproduct-completion} form an equivalence of dagger symmetric monoidal quantaloids.
 \end{proposition}
\begin{proof}
Clearly, $\otimes$ is a bifunctor on $\mathrm{Matr}(\mathbf Q)$. We verify the triangle identity. Let $X=(X_\beta)_{\beta\in B}$ and $Y=(Y_\gamma)_{\gamma\in C}$ objects in $\mathrm{Matr}(\mathbf Q)$. Then for each $\beta,\beta'\in B$ and $\gamma,\gamma'\in C$, we have
\begin{align*}
    \big((\id_X\otimes\lambda_Y)\circ\alpha_{X,J,Y}\big)_{((\beta,*),\gamma)}^{(\beta',\gamma')} & = \bigvee_{(\beta'',(*,\gamma''))\in B\times(1'\times C)}(\id_X\otimes\lambda_Y)_{(\beta'',(*,\gamma''))}^{(\beta',\gamma')}\circ (\alpha_{X,J,Y})_{((\beta,*),\gamma)}^{(\beta'',(*,\gamma''))}\\
    & = \bigvee_{\beta''\in B, \gamma''\in C}\big((\id_X)_{\beta''}^{\beta'}\otimes(\lambda_Y)_{(*,\gamma'')}^{\gamma'}\big)\circ (\alpha_{X,J,Y})_{((\beta,*),\gamma)}^{(\beta'',(*,\gamma''))}
\end{align*}
Note that for two morphisms $f:U\to V$  and $g:V\to W$ in $\mathbf Q$, we have $g\circ f=\perp_{U,W}$ if either $f=\perp_{U,V}$ or $g=\perp_{V,W}$ by Lemma \ref{lem:composition preserves least elements in quantaloids}. By assumption, $\mathbf Q$ is a symmetric monoidal quantaloid, hence given morphisms $h:U\to W$ and $k:V\otimes Z$ in $\mathbf Q$, it follows from Lemma \ref{lem:tensor preserves least elements in symmetric monoidal quantaloids} that  $h\otimes k=\perp_{U\otimes V,W\otimes Z}$ if either $h=\perp_{U,W}$ or $k=\perp_{V,Z}$. Hence,
\begin{align*}
    \big((\id_X\otimes\lambda_Y)\circ\alpha_{X,J,Y}\big)_{((\beta,*),\gamma)}^{(\beta',\gamma')} 
    & = \big((\id_X)_{\beta}^{\beta'}\otimes(\lambda_Y)_{(*,\gamma)}^{\gamma'}\big)\circ (\alpha_{X,J,Y})_{((\beta,*),\gamma)}^{(\beta,(*,\gamma))}\\
    & = \begin{cases}
    (\id_{X_\beta}\otimes\lambda_{Y_\gamma})\otimes\circ \alpha_{X_\beta,I,Y_\gamma}, & \beta=\beta',\gamma=\gamma',\\
    \perp_{(X_\beta\otimes I)\otimes Y_\gamma,X_{\beta'}\otimes(I\otimes Y_{\gamma'})}, & \text{otherwise}
        \end{cases}\\
         & = \begin{cases}
    \rho_{X_\beta}\otimes\id_{Y_\gamma}, & \beta=\beta',\gamma=\gamma',\\
    \perp_{(X_\beta\otimes I)\otimes Y_\gamma,X_{\beta'}\otimes(I\otimes
    Y_{\gamma'})}, & \text{otherwise}
        \end{cases}\\
        & = (\rho_X\otimes \id_Y)^{(\beta',\gamma')}_{(\beta,\gamma)},
\end{align*}
where we used the triangle identity for $\mathbf Q$ in the penultimate equality. Hence, the triangle identity holds for $\mathrm{Matr}(\mathbf Q)$.
The pentagon identity for $\mathrm{Matr}(\mathbf Q)$ follows in a similar way from the pentagon identity for $\mathbf Q$. 

Let $(X_\alpha)_{\alpha\in A}$, $(Y_\beta)_{\beta\in B}$, $W=(W_\gamma)_{\gamma\in C}$ and $Z=(Z_\delta)_{\delta\in D}$ be objects in $\mathrm{Matr}(\mathbf Q)$. It remains to show that $\mathrm{Matr}(\mathbf Q)(X,W)\times\mathrm{Matr}(\mathbf Q)(Y,Z)\to\mathrm{Matr}(\mathbf Q)(X\otimes Y,W\otimes Z)$, $(f,g)\mapsto f\otimes g$ preserves suprema in each argument separately. So let $(f_\kappa)_{\kappa\in K}$ be a family of morphisms in $\mathrm{Matr}(\mathbf Q)(X,W)$. Then for each $\alpha\in A$, $\beta\in B$, $\gamma\in C$ and $\delta $, we have
\begin{align*}
    \left(\left(\bigvee_{\kappa\in K}f_\kappa\right)\otimes g\right)_{(\alpha,\beta)}^{(\gamma,\delta)} & = \left(\bigvee_{\kappa\in K}f_\kappa\right)_\alpha^\gamma\otimes g_\beta^\delta=\left(\bigvee_{\kappa\in K}(f_\kappa)_\alpha^\gamma\right)\otimes g_\beta^\delta\\
    & = \bigvee_{\kappa\in K}\left((f_\kappa)_{\alpha}^\gamma\otimes g_{\beta}^\delta\right) = \bigvee_{\kappa\in K}(f_\kappa\otimes g)_{(\alpha,\beta)}^{(\gamma,\delta)}=\left(\bigvee_{\kappa\in K}(f_\kappa\otimes g)\right)_{(\alpha,\beta)}^{(\gamma,\delta)},
\end{align*}
where the first and penultimate equalities are by definition of $\otimes$ on $\mathrm{Matr}(\mathbf Q)$, the second and last equalities by definition of suprema in $\mathrm{Matr}(\mathbf Q)$, and the third equality because $\mathbf Q$ is a symmetric monoidal quantaloid. Thus $\left(\bigvee_{\kappa\in K}f_\kappa\right)\otimes g=\bigvee_{\kappa\in K}(f_\kappa\otimes g)$, and in a similar way, we show that $\otimes$ preserves suprema in the second argument. It is already asserted in Theorem \ref{thm:biproduct-completion} that $\mathrm{Matr}(\mathbf Q)$ is a dagger quantaloid when $\mathbf Q$ is a dagger quantaloid. Assume that $(\mathbf Q,\otimes,I)$ is a dagger symmetric monoidal quantaloid, and let $f=(f_{\alpha}^\gamma)_{(\alpha,\gamma)\in A\times C}:(X_\alpha)_{\alpha\in A}\to (W_\gamma)_{\gamma\in C}$ and $g=(g_\beta^\delta)_{(\beta,\delta)\in B\times D}:(Y_\beta)_{\beta\in B}\to(Z_{\delta})_{\delta\in D}$ are morphisms in $\mathrm{Matr}(\mathbf Q)$. Then for each $\alpha\in A$, $\beta\in B$, $\gamma\in C$ and $\delta\in D$, we have $((f\otimes g)^\dag)_{(\alpha,\beta)}^{(\gamma,\delta)}=((f\otimes g)_{(\alpha,\beta)}^{(\gamma,\delta)})^\dag=(f_\alpha^\gamma\otimes g_\beta^\delta)^\dag=( f_\alpha^\gamma)^\dag\otimes (g_\beta^\delta)^\dag=(f^\dag)_\alpha^\gamma\otimes (g^\dag)_\beta^\delta=(f^\dag\otimes g^\dag)_{(\alpha,\beta)}^{(\gamma,\delta)}$, where we used the $\mathbf Q$ is a dagger symmetric monoidal category in the fourth equality. So $(f\otimes g)^\dag=f^\dag\otimes g^\dag$. 
Hence, if $\mathbf Q$ is a dagger symmetric monoidal quantaloid, then so is $\mathrm{Matr}(\mathbf Q)$. Since the functor $E_\mathbf Q:\mathbf Q\to\mathrm{Matr}(\mathbf Q)$ is fully faithful (cf. Theorem \ref{thm:biproduct-completion}), we have $\mathrm{Matr}(\mathbf Q)(J,J)=\mathbf Q(I,I)$, hence $\mathrm{Matr}(\mathbf Q)$ is nontrivial if and only if $\mathbf Q$ is nontrivial, affine if and only if $\mathbf Q$ is affine, and binary if and only if $\mathbf Q$ is binary.

Assume that $\mathbf Q$ has all dagger biproducts. 
By Theorem \ref{thm:biproduct-completion}, the functors $E_\mathbf Q$ and $P_\mathbf Q$ are homomorphisms of dagger quantaloids, and form an equivalence of categories, hence it suffices to show that only one of $E_\mathbf Q$ and $P_\mathbf Q$ is dagger strong monoidal. For $X,Y\in\mathbf Q$, note that $J=(I)=E_\mathbf Q(I)$, and $E_\mathbf Q(X)\otimes E_{\mathbf Q}(Y)=(X)\otimes (Y)=(X\otimes Y)$, hence the coherence isomorphisms for $E_\mathbf Q$ can be chosen to be identities, hence $E_\mathbf Q$ is a strict symmetric monoidal functor, and since identities are in particular are unitaries, $E_\mathbf Q$ is dagger strong symmetric monoidal. We conclude that $E_\mathbf Q$ and $P_\mathbf Q$ form an equivalence of dagger symmetric monoidal quantaloids. 
\end{proof}

\begin{proposition}\label{prop:Matr-homomorphism-of-dagger-symmetric-monoidal-quantaloids}
    Let $F:(\mathbf Q,\odot,I_{\mathbf Q})\to(\mathbf R,\otimes,I_{\mathbf R})$ be a homomorphism of dagger symmetric monoidal quantaloids with coherence dagger isomorphisms $\varphi:I_\mathbf R\to FI_{\mathbf Q}$ and $\varphi_{X,Y}:FX\otimes FY\to F(X\odot Y)$ for  $X,Y\in\mathbf Q$. Let $J_\mathbf Q$ and $J_{\mathbf R}$ be the monoidal units of $\mathrm{Matr}(\mathbf Q)$ and $\mathrm{Matr}(\mathbf R)$, respectively. Then also $\mathrm{Matr}(F):\mathrm{Matr}(\mathbf Q)\to\mathrm{Matr}(\mathbf R)$ (cf. Theorem  \ref{thm:biproduct-completion}) is a homomorphism of dagger symmetric monoidal quantaloids with coherence dagger isomorphisms 
    $\Psi:J_{\mathbf R}\to \mathrm{Matr}(F)(J_{\mathbf Q})$ to be the `matrix' $(\varphi)$. For objects $X=(X_\alpha)_{\alpha\in A}$ and $Y=(Y_\beta)_{\beta\in B}$ in $\mathrm{Matr}(\mathbf Q)$, we define $\Phi_{X,Y}:\mathrm{Matr}(F)(X)\otimes\mathrm{Matr}(F)(Y)\to\mathrm{Matr}(F)(X\odot Y)$ for each $\alpha,\gamma\in A$ and $\beta,\delta\in B$ by
    \[ (\Phi_{X,Y})_{(\alpha,\beta)}^{(\gamma,\delta)}=\begin{cases} \varphi_{X_\alpha,Y_\beta}, & (\alpha,\beta)=(\gamma,\delta),\\
    \perp_{FX_\alpha\otimes FY_{\beta},F(X_\gamma\odot Y_\delta)}, & \text{otherwise}.\end{cases}\]
     Moreover, $\mathrm{Matr}(F)$ is (fully) faithful if $F$ is (fully) faithful, and preserves small dagger biproducts. 
\end{proposition}
\begin{proof}
It follows from Theorem \ref{thm:biproduct-completion} that in order to show that $\mathrm{Matr}(F)$ is a homomorphism of dagger symmetric monoidal quantaloids, we only need to show that it is a dagger strong symmetric monoidal functor, which is straightforward but tedious, so as an illustration of the proof, we only verify condition (\ref{eq:coherence2}) from the definition of a dagger strong symmetric monoidal functor (cf.  Definition \ref{def:dagger strong monoidal functor}). Since $F$ is  dagger strong symmetric monoidal, we have $F\lambda_Y\circ\varphi_{J,Y}\circ (\varphi\otimes\id_{FY}) = \lambda_{FY}$ for each $X,Y\in\mathbf Q$. Now, let  $Y=(Y_\beta)_{\beta\in B}$ be an object of $\mathrm{Matr}(\mathbf Q)$. Then for $(*,\beta)\in 1\times B$ and $\delta\in B$, we need to show that 
\[\big(\mathrm{Matr}(F)(\lambda_Y) \circ\Phi_{J_\mathbf Q,Y}\circ (\Phi\otimes\id_{\mathrm{Matr}(F)(Y)})\big)_{(*,\beta)}^\delta=(\lambda_{\mathrm{Matr}(F)(Y)})_{(*,\beta)}^\delta.\]
By definition, the right-hand side evaluates to $\lambda_{FY_\beta}$ if $\delta=\beta$, and to $\perp_{I_\mathbf R\otimes FY_\beta,FY_\delta}$ otherwise.
Because $F\perp=\perp$, and both the monoidal product and composition preserve $\perp$ in each argument separately by definition of a symmetric monoidal quantaloid, we have also have that the left-hand side eavluates to $\perp_{I_\mathbf R\otimes FY_\beta,FY_\delta}$ if $\beta\neq\delta$. In case $\beta=\delta$, the left-hand side evaluates to 
\begin{align*} \mathrm{LHS}& =\bigvee_{\alpha,\gamma\in B}F((\lambda_{Y})_{(*,\alpha)}^\beta)\circ (\Phi_{J_\mathbf Q,Y})_{(*,\gamma)}^{(*,\alpha)}\circ(\Phi_*^*\otimes(\id_{\mathrm{Matr}(F)(Y)})_{\beta}^\gamma)
\\
& = F((\lambda_{Y})_{(*,\beta)}^\beta)\circ (\Phi_{J_\mathbf Q,Y})_{(*,\beta)}^{(*,\beta)}\circ(\Phi_*^*\otimes(\id_{\mathrm{Matr}(F)(Y)})_{\beta}^\beta)\\
& = F(\lambda_{Y_\beta})\circ \varphi_{I_\mathbf Q,Y_\beta}\circ(\varphi\otimes
\id_{FY_\beta})=\lambda_{FY_\beta}=(\lambda_{\mathrm{Matr}(F)(Y)})_{(*,\beta)}^\beta
\end{align*}
The other conditions in Definition \ref{def:dagger strong monoidal functor} are verified in a similar way, and it is also  straightforward that the coherence morphisms of $\mathrm{Matr}(F)$ are unitaries, and that $\mathrm{Matr}(F)$ is (fully) faithful when $F$ is (fully) faithful. Finally, it follows from Proposition \ref{prop:quantaloid homomorphisms preserve biproducts} that $\mathrm{Matr}(F)$ preserves small dagger biproducts. 
\end{proof}

\begin{theorem}\label{thm:sum completion dagger compact quantaloid}
    Let $(\mathbf Q,\otimes,I)$ be a (dagger) compact quantaloid with unit morphisms $\eta_Y:I\to Y^*\otimes Y$ and counit morphisms $\epsilon_Y:Y\otimes Y^*\to I$ for each object $Y$ of $\mathbf Q$. Then $(\mathrm{Matr}(\mathbf Q),\otimes,J)$ becomes a (dagger) compact quantaloid if for each object $X=(X_\alpha)_{\alpha\in A}$ in $\mathrm{Matr}(\mathbf Q)$ we define $X^*:=(X_\alpha^*)_{\alpha\in A}$, and $\eta_X:J\to X^*\otimes X$ and $\epsilon_X:X\otimes X^*\to J$ by
   \begin{align*}
   (\eta_X)^{(\alpha,\beta)}_*  :=\begin{cases}
    \eta_{X_\alpha}, & \alpha=\beta,\\
    \perp_{I,X_\alpha^*\otimes X_\beta}, & \mathrm{otherwise},
    \end{cases},\qquad 
     (\epsilon_X)_{(\alpha,\beta)}^* : = \begin{cases}
       \epsilon_{X_\alpha}, & \alpha=\beta\\
       \perp_{X_\alpha\otimes X_\beta^*,I}, & \mathrm{otherwise}
    \end{cases}
    \end{align*}
    for each $\alpha,\beta\in A$.
    
\end{theorem}
\begin{proof}
Similar as in the proof of Proposition \ref{prop:biproduct completion monoidal quantaloid}, we find for each $\alpha,\beta\in A$:
    \begin{align*}
        (\lambda_X\circ & (\epsilon_X\otimes\id_X)\circ\alpha_{X,X^*,X}^{-1}\circ(\id_X\otimes\eta_X)\circ\rho_X^{-1})_\alpha^\beta \\
        & = \begin{cases}
        \lambda_{X_\alpha}\circ(\epsilon_{X_\alpha}\otimes\id_{X_\alpha})\circ\alpha_{{X_\alpha},X_\alpha^*,{X_\alpha}}^{-1}\circ(\id_{X_\alpha}\otimes\eta_{X_\alpha})\circ\rho_{X_\alpha}^{-1}, & \alpha=\beta,\\
        \perp_{X_\alpha,X_\beta}, & \alpha\neq\beta
        \end{cases}\\
        & = \begin{cases}
            \id_{X_\alpha}, & \alpha=\beta,
        \\
        \perp_{X_\alpha,X_\beta}, & \alpha\neq\beta,
        \end{cases}\\
        & = (\id_X)_\alpha^\beta,
    \end{align*}
    where we used that $\mathbf Q$ is compact in the second equality. Thus,
    $\lambda_X\circ
    (\epsilon_X\otimes\id_X)\circ\alpha_{X,X^*,X}^{-1}\circ(\id_X\otimes\eta_X)\circ\rho_X^{-1}=\id_X$,
    and in a similar way, we find
    $\rho_{X^*}\circ(\id_{X^*}\otimes\epsilon_X)\circ\alpha_{X^*,X,X^*}\circ(\eta_X\otimes\id_{X^*})\circ\lambda_{X^*}^{-1}=\id_{X^*}$,
    so $\mathrm{Matr}(\mathbf Q,\otimes,J)$ is a compact quantaloid. If $\mathbf Q$
    is a dagger compact quantaloid, we have for each $\alpha,\beta\in A$:

\begin{align*}
(\sigma_{X,X^*}\circ\epsilon_X^\dag)_*^{(\alpha,\beta)} & = \bigvee_{(\gamma,\delta)\in A\times A}(\sigma_{X,X^*})^{(\alpha,\beta)}_{(\gamma,\delta)}\circ (\epsilon_X^\dag)_*^{(\gamma,\delta)}
     =  \bigvee_{(\gamma,\delta)\in A\times A}(\sigma_{X,X^*})^{(\alpha,\beta)}_{(\gamma,\delta)}\circ ((\epsilon_X)^*_{(\gamma,\delta)})^\dag\\
    & = (\sigma_{X,X^*})^{(\alpha,\beta)}_{(\beta,\alpha)}\circ ((\epsilon_X)^*_{(\beta,\alpha)})^\dag= \begin{cases}
        \sigma_{X_\alpha,X_\alpha^*}\circ \epsilon_{X_\alpha}^\dag, & \alpha=\beta,\\
        \perp_{I,X_\alpha^*\otimes X_\beta}, & \alpha\neq\beta
    \end{cases}\\
    & = \begin{cases}
        \eta_{X_\alpha}, & \alpha=\beta,\\
    \perp_{I,X_\alpha^*\otimes X_\beta}, & \alpha\neq\beta
        \end{cases}\\
        & = (\eta_X)_*^{(\alpha,\beta)},
\end{align*}
from which we conclude that $\eta_X=\sigma_{X,X^*}\circ\epsilon_X^\dag$, which shows that $(\mathrm{Matr}(\mathbf Q),\otimes,J)$ is dagger compact. 
\end{proof}

\begin{example}\label{ex:V-Rel is dagger compact quantaloid}
Let $V$ be a commutative quantale. Then $\times$ becomes a bifunctor on $V$-$\mathbf{Rel}$ if for $V$-relations $r:X_1\sto Y_1$ and $s:X_2\sto Y_2$ we define $r\times s:X_1\times X_2\sto Y_1\times Y_2$ as the function $(X_1\times X_2)\times (Y_1\times Y_2)\to V$ given by \[(r\times s)\big( (x_1,x_2),(y_1,y_2)\big):=r(x_1,y_1)\cdot s(x_2,y_2).\] Then $(V$-$\mathbf{Rel},\times, 1)$ is a dagger compact quantaloid, which can be seen as follows. By Example \ref{ex:dagger compact quantaloid induced by quantale}, $\mathbf V$ is a dagger compact quantaloid, hence by Theorem \ref{thm:sum completion dagger compact quantaloid}, so is $\mathrm{Matr}(\mathbf V)$. It is straightforward to see that the equivalence of dagger quantaloids $\mathrm{Matr}(\mathbf V)\cong V$-$\mathbf{Rel}$ in Proposition \ref{prop:VRelisMatrV} is an equivalence of dagger symmetric monoidal quantaloids, hence also $V$-$\mathbf{Rel}$ is a dagger compact quantaloid. 
\end{example}

\begin{example}\label{ex:qRel-dagger-compact}
    Since $(\mathbf{FdOS},\otimes,\CC)$ is a binary dagger compact quantaloid (cf.
    Example \ref{ex:FdOS-dagger-compact}), it follows that
    $\mathbf{qRel}=\mathrm{Matr}(\mathbf{FdOS})$ is also a binary dagger compact
    quantaloid. The monoidal product of $\qRel$ is typically denoted by $\times$. The
    monoidal unit of $\mathbf{qRel}$ is $\mathbf 1=(\CC)$. Since the embedding
    $E_{\mathbf{FdOS}}:\mathbf{FdOS}\to\mathrm{Matr}(\mathbf{FdOS})=\mathbf{qRel}$ is
    fully faithful (cf. Theorem \ref{thm:biproduct-completion}), it follows that
    $\mathbf{qRel}$ is also binary.
\end{example}
 It might be odd to denote the monoidal product of $\qRel$ by $\times$, since in this article, given a symmetric monoidal quantaloid $\mathbf Q$, we denote the monoidal product on $\mathrm{Matr}(\mathbf Q)$ by the same symbol as the monoidal product of $\mathbf Q$, and the monoidal product on $\mathbf{FdOS}$ is the Hilbert space tensor product $\otimes$. Moreover, the monoidal product $\times$ on $\qRel$ is not the categorical product. However, the monoidal product $\times$ on $\qRel$ is the noncommutative generalization of the set-theoretic product on $\Rel$, which is also not a categorical product. In noncommutative mathematics, it is customary to use the same symbols for noncommutative generalizations of structures, which we choose to follow, even though it clashes with other conventions.

\section{Dagger kernel quantaloids and orthomodularity}\label{sec:orthomodular}

Dagger kernel categories require the existence of zero objects, which are not always present in a dagger quantaloids. However, we will see that one can freely add a zero object to any quantaloid. Since also the least elements of quantaloids with a zero object coincide with the zero maps (cf. Lemma \ref{lem:quantaloid with zero}), we can exploit this freely-added zero object to generalize the notion of dagger kernel categories to quantaloids without a zero object. 

\subsection{Freely adding a zero object to a quantaloid}
Let $\mathbf Q$ be a quantaloid. By $\mathbf Q_0$, we denote the category that has the same objects as $\mathbf Q$, together with one addition object, denoted $0$. For objects $X,Y\in\mathbf Q_0$, we define
\[\mathbf Q_0(X,Y) = \begin{cases}
\mathbf Q(X,Y), & X\neq 0\neq Y,\\
1, & \text{otherwise}.
\end{cases}\]
In case $X=0$ or $Y=0$, we denote the single morphism in $\mathbf Q(X,Y)$ by $\perp_{X,Y}$. We define composition $g\circ_0 f$ on $\mathbf Q_0$ of morphisms  $f:X\to Y$ and $g:Y\to Z$ in $\mathbf{Q}_0$ by
\[ g\circ_0 f : =\begin{cases} g\circ f, & X,Y,Z\in\mathbf Q,\\
\perp_{X,Z}, & \text{otherwise}.
\end{cases}\]
In practice, we will just write $g\circ f$ instead of $g\circ_0 f$.
It is straightforward that if $\mathbf Q_0$ is a quantaloid with a zero object $0$. If $\mathbf Q$ is a dagger quantaloid, then clearly $\mathbf Q_0$ also is a dagger quantaloid. If $\mathbf Q$ already had a zero object, then the $(-)_0$-construction adds an additional zero object.

If $(\mathbf Q,\otimes,I)$ is a (dagger) symmetric monoidal quantaloid, then we can define a symmetric monoidal product $\otimes_0$ on $\mathbf Q_0$ such that also $(\mathbf Q_0,\otimes_0,I)$ is a (dagger) symmetric monoidal quantaloid. Namely, for objects $X,Y\in\mathbf Q_0$ we define
\[ X\otimes_0 Y = \begin{cases} X\otimes Y, & X,Y\in\mathbf Q,\\
 0, & \text{otherwise}.
 \end{cases}\]
 For morphisms $f:X\to W$, $g:Y\to Z$ in $\mathbf Q_0$ we define
 \[f\otimes_0 g =\begin{cases} f\otimes g, &  f\in\mathbf Q(X,W), g\in\mathbf Q(Y,Z),\\
 \perp_{X\otimes_0 Y,W\otimes_0 Z}, & \text{otherwise}.
 \end{cases}\]
Since $X\otimes 0=0$ and $0\otimes Y=0$ for each $X,Y\in\mathbf Q$, it follows that the components of the associator, unitors, and symmetry involving $0$ must be $\perp$. 
In a similar way, it easily follows that $(\mathbf Q_0,\otimes_0,I)$ is a (dagger) compact quantaloid if  $(\mathbf Q,\otimes,I)$ is a (dagger) compact quantaloid by taking $0^*=0$.  
In practice, we just write $\otimes$ instead of $\otimes_0$.

Let $\mathbf Q$ be a dagger quantaloid. Then the inclusion $Z_\mathbf Q:\mathbf Q\to\mathbf Q_0$, which is clearly a homomorphism of dagger quantaloids, has the universal property that for each dagger quantaloid with zero object $\mathbf R$ and each homomorphism of dagger quantaloids $F:\mathbf Q\to\mathbf R$, there is a  homomorphism of dagger quantaloids $\hat F:\mathbf Q_0\to \mathbf R$ that is unique up to natural isomorphism such that the following diagram commutes up to natural isomorphism:

\[ \begin{tikzcd}
 \mathbf Q\ar{r}{Z_\mathbf Q}\ar{dr}[swap]{F} & \mathbf Q_0\ar{d}{\hat F}
  \\
  & \mathbf R.
\end{tikzcd} \]

Indeed, for each object $X\in\mathbf Q_0$, we define
\[ \hat FX  = \begin{cases}
FX, & X\in\mathbf Q;\\
0, & \text{otherwise},
\end{cases}
\]
and for each morphism $f:X\to Y$ in $\mathbf Q_0$ we define
\[\hat Ff =\begin{cases} Ff, & X,Y\in\mathbf Q,\\
\perp_{\hat FX,\hat FY}, & \text{otherwise}.
\end{cases}\]

\begin{example}
 The dagger compact quantaloid $\mathbf{FdOS}_0$ (cf. Example \ref{ex:FdOS-dagger-compact}) can be obtained from the dagger compact quantaloid $\mathbf{FdOS}$ by freely adjoining a zero object.
\end{example}

Our running examples $V$-$\mathbf{Rel}$ and $\mathbf{qRel}$ are obtained by applying the dagger biproduct completion to a dagger quantaloid $\mathbf Q$ without a zero object. The following proposition states that adding a zero object to $\mathbf Q$ yields the same quantaloid. 
\begin{proposition}\label{prop:MatrQ0isMatrQ}
    Let $\mathbf Q$ be a dagger quantaloid. Then $\mathrm{Matr}(Z_\mathbf Q):\mathrm{Matr}(\mathbf Q)$ and $\mathrm{Matr}(\mathbf Q_0)$ is an equivalence of dagger quantaloids.
\end{proposition}
\begin{proof}
By Theorem \ref{thm:biproduct-completion}, the inclusion $Z_\mathbf Q:\mathbf Q\to\mathbf Q_0$ induces a homomorphism of dagger quantaloids $\mathrm{Matr}(Z_\mathbf Q):\mathrm{Matr}(\mathbf Q)\to\mathrm{Matr}(\mathbf Q_0)$, which satisfies $\mathrm{Matr}(Z_\mathbf Q)\circ E_\mathbf Q=E_{\mathbf Q_0}\circ Z_\mathbf Q$.

Since $\mathrm{Matr}(\mathbf Q)$ has a zero object, it follows from the universal property of $Z_\mathbf Q:\mathbf Q\to\mathbf Q_0$ that there is some homomorphism of dagger quantaloids $G:\mathbf Q_0\to\mathrm{Matr}(\mathbf Q)$ such that $G\circ Z_\mathbf Q=E_\mathbf Q$, namely $G=\widehat{E_\mathbf Q}$. By the universal property of $E_{\mathbf Q_0}:\mathbf Q_0\to\mathrm{Matr}(\mathbf Q_0)$ as the dagger biproduct completion of $\mathbf Q_0$ (cf. Theorem \ref{thm:biproduct-completion}), it follows that there is some homomorphism of dagger quantaloids $H:\mathrm{Matr}(\mathbf Q_0)\to\mathrm{Matr}(\mathbf Q)$ such that $H\circ E_{\mathbf Q_0}=G$, namely $H=\overline{G}$. Then $H\circ\mathrm{Matr}(Z_\mathbf Q)\circ E_{\mathbf Q}=H\circ E_{\mathbf Q_0}\circ Z_\mathbf Q=G\circ Z_\mathbf Q=E_\mathbf Q$, and by the universal property of $E_\mathbf Q$, we find $H\circ\mathrm{Matr}(Z_\mathbf Q)\cong \id_{\mathrm{Matr}(\mathbf Q)}$.

Furthermore, we have $\mathrm{Matr}(Z_\mathbf Q)\circ G\circ Z_\mathbf Q=\mathrm{Matr}(Z_\mathbf Q)\circ E_\mathbf Q=E_{\mathbf Q_0}\circ Z_\mathbf Q$, and from the universal property of $Z_\mathbf Q$, we obtain $\mathrm{Matr}(Z_\mathbf Q)\circ G\cong E_{\mathbf Q_0}$. Then $\mathrm{Matr}(Z_\mathbf Q)\circ H\circ E_{\mathbf Q_0}=\mathrm{Matr}(Z_\mathbf Q)\circ G\cong E_{\mathbf Q_0}$, and by the universal property of $E_{\mathbf Q_0}$, we obtain $\mathrm{Matr}(Z_\mathbf Q)\circ H\cong\id_{\mathrm{Matr}(\mathbf Q_0)}$. 
\end{proof}

Now, in order to find the generalization of dagger quantaloids with dagger kernels, we  need the quantaloid analog of a zero-mono.

\begin{definition}
 We call a morphism $m:Y\to Z$ in a dagger quantaloid $\mathbf Q$ a $\bot$-mono if $m\circ f=\perp_{X,Z}$ implies $f=\perp_{X,Y}$ for each object $X$ of $\mathbf Q$ and each morphism $f:X\to Y$. 
\end{definition}

\begin{example}\label{ex:perp-monos-in-frames}
    Let $F$ be a frame. Then $F$ is an affine commutative quantale with $x\cdot y=x\wedge y$. Given $x\in F$, we define the \emph{pseudocomplement} $x^*$ of $x$ by $x^*:=\bigvee\{y\in F:x\wedge y=\perp\}$. We say that $x$ is \emph{dense} in $F$ if $x^*=\perp$. Then the dense elements of $F$ are precisely the $\bot$-monos of $\mathbf F(1,1)$. Indeed, if $x$ is dense, and $y$ is an element of $F$ such that $x\cdot y=\perp$, then $x\wedge y=\perp$, hence $y\leq x^*=\perp$, so $y=\perp$, which shows that $x$ is a $\bot$-mono. Conversely, if $x$ is a $\bot$-mono, then $a\cdot a^*=a\wedge a^*=\perp$, which implies that $a^*=\perp$, so $a$ is dense. Since Boolean algebras are precisely the frames for which $\top$ is the only dense element, it follows that $F$ is Boolean if and only if $\top$ is the only $\bot$-mono in $\mathbf F$.
\end{example}

The following two lemmas are straightforward.

\begin{lemma}\label{lem:perp-mono-in-Qiszero-mono-inQ0}
    Let $\mathbf Q$ be a dagger quantaloid. Then a morphism $m:X\to Y$ in $\mathbf Q$ is a $\bot$-mono if and only if it is a zero-mono in $\mathbf Q_0$.  If $\mathbf Q$ has a zero object, then its $\bot$-monos coincide with its zero-monos.
\end{lemma}

\begin{lemma}\label{lem:zeromonoeffect implies topzeromono}
Let $\mathbf Q$ be a dagger quantaloid. For each $X,Y\in\mathbf Q$, if $\mathbf Q(X,Y)$ has a $\bot$-mono $m$, then any $n\in \mathbf Q(X,Y)$ such that $m\leq n$ is a $\bot$-mono.
\end{lemma}

In particular, $\bot$-monos that are effects will be of importance.
\begin{lemma}\label{lem:unique-zero-monic-effects-in-MatrQ}
    Let $(\mathbf{Q},\otimes, I)$ be a dagger symmetric monoidal quantaloid such that for each object $X\in\mathbf Q$, there is precisely one $\bot$-monic effect $X\to I$. Then:
    \begin{itemize}
    \item[(a)] $\top_{X,I}$ is the only $\bot$-monic effect $X\to I$ for each object of $\mathbf Q$;
    \item[(b)] $(\mathbf Q,\otimes, I)$ is affine;
    \item[(c)] $\top_{Y,J}$ is the only $\bot$-monic effect for each object $Y$ of $\mathrm{Matr}(\mathbf Q)$;
    \item[(d)] $(\mathrm{Matr}(\mathbf Q),\otimes,J)$ is affine.
    \end{itemize}
\end{lemma}
\begin{proof}
  By Lemma \ref{lem:zeromonoeffect implies topzeromono}, it follows that $\top_{X,I}$ is the only $\bot$-monic effect of any object $X\in\mathbf Q$, which proves (a).  Since $\id_I$ is clearly a $\bot$-mono, this forces $\id_I=\top_I$, so $\mathbf Q$ is affine, proving (b). 

  For (c), let $Y=(Y_\alpha)_{\alpha\in A}$ be an object in $\mathrm{Matr}(\mathbf Q)$. If $A=\emptyset$, then $Y$ is the empty family, which is the zero object in $\mathrm{Matr}(\mathbf Q)$, from which it immediately follows that $\top_{Y,J}=\perp_{Y,J}$, and that $\top_{Y,J}$ is a $\bot$-mono. So we may assume that $A\neq\emptyset$.
Clearly, $\top_{Y,J}=(\top_{Y_\beta,I})_{(\beta,*)\in B\times 1}$. Let $f=(f_\alpha^\beta)_{(\alpha,\beta)\in A\times B}:X\to Y$ be a morphism such that $\top_{Y,J}\circ f=\perp_{X,J}$. Fix $\alpha\in A$. Then  we have
\[ \perp_{X_\alpha,I}=(\top_{Y,J}\circ f)_\alpha^*=\bigvee_{\beta\in B}(\top_{Y,J})_\beta^*\circ f_\alpha^\beta=\bigvee_{\beta\in B}\top_{Y_\beta,I}\circ f_\alpha^\beta,\]
which forces $\top_{Y_\beta,I}\circ f_\alpha^\beta=\perp_{X_\alpha,I}$ for each $\beta\in B$. Hence, for fixed $\beta\in B$, since $\top_{Y_\beta,I}$ is a $\bot$-mono, we must have $f_\alpha^\beta=\perp_{X_\alpha,Y_\beta}$. We conclude that $f=\perp_{X,Y}$, so $\top_{Y,J}$ is a $\bot$-mono.

Let $g:Y\to J$ be another effect, not equal to $\top_{Y,J}$. So there is some $\beta_0\in B$ such that $g_{\beta_0}^*<\top_{Y_{\beta_0},I}$. Then $g_{\beta_0}^*:Y_{\beta_0}\to I$ cannot be a $\bot$-mono in $\mathbf Q$, hence there is some nonzero $h:Z\to Y_{\beta_0}$ such that $g_{\beta_0}^*\circ h= \perp_{Z,I}$. Let $W=(W_\delta)_{\delta\in 1}$ with $W_*=Z$, so $W\in\mathrm{Matr}(\mathbf Q)$, and let $r:W\to Y$ be given by \[r_*^{\beta}=\begin{cases}
h, & \beta=\beta_0,\\
\perp_{Z,Y_{\beta}}, & \beta\neq\beta_0. \end{cases}\]
Then $r\neq \perp_{W,Y}$, but \[(g\circ r)_*^*=\bigvee_{\beta\in B}g_\beta^*\circ r_*^\beta=\bigvee_{\beta\neq\beta_0}g_\beta^*\circ \perp_{Z,Y_\beta}\vee g_{\beta_0}^*\circ h=0_{Z,I},\]
so $g\circ r=\perp_{W,J}$. We conclude that $g$ is not a $\bot$-mono, so $\top_{Y,J}$ is the only $\bot$-monic effect on $Y$. Now (d) follows from (a). 
\end{proof}

\begin{lemma}\label{lem:zero mono effects in FdOS0}
In $(\mathbf{FdOS}_0,\otimes,\CC)$, an effect $R:Y\to \mathbb C$ is a $\bot$-mono if and only if $R=\top_{Y,\mathbb C}$.
        \end{lemma}
\begin{proof}
        We first show that $\top_{Y,\mathbb C}$ is a $\bot$-mono, so let $S:X\to Y$ be a nonzero morphism in $\mathbf{FdOS}_0$. Then there is some $s\in S$ such that $s(x)\neq 0$ for some $x\in X$. Write $y=s(x)$. Let $\hat y:=\langle y,-\rangle:Y\to\CC$, which is an element of $B(Y,\CC)=\top_{Y,\CC}$. Then $\hat t(s(x))=\langle y,y\rangle\neq 0$, since $y\neq 0$. Thus, $\hat yr\neq 0$, hence $\top_{Y,\CC}\cdot S\neq 0$, so $\top_{Y,\CC}$ is a $\bot$-mono.  For the converse, assume that $R:Y\to\CC$ is not equal to $\top_{Y,\CC}=B(Y,\CC)$. So $R$ is a proper subspace of $B(Y,\CC)$, It follows that there is a nonzero functional $\varphi:Y\to\mathbb C$ that is orthogonal to all functionals in $S$. The Riesz representation theorem states that the map $y\mapsto\langle y,-\rangle$ is an antilinear bijective isometry $Y\to B(Y,\CC)$. Hence, for any functional $\psi:Y\to\CC$, the Riesz representation theorem assures the existence of a unique $y_\psi\in Y$ such that $\psi=\langle y_\psi,-\rangle$.
Write $y=y_\varphi$. Note that $y\neq 0$, for $\varphi$ is nonzero. Let $\psi\in S$.
Since $\psi\perp\varphi$, it follows from the Riesz representation theorem that
$0=\langle\varphi,\psi\rangle=\langle y_\psi,y\rangle=\psi(y)$. As a consequence, we
have $\psi(y)=\langle y_\psi,y_\varphi\rangle=\langle \psi,\varphi\rangle=0$.  Let
$\check{y}:\CC\to Y$ be function $\lambda\mapsto \lambda y$. Note that $\CC\hat
y:\CC\to Y$ is nonzero a morphism in $\mathbf{FdOS}_0$. Moreover, for each $\psi\in
S$ and each $\lambda\in\CC$, we have $\psi(\check y(\lambda))=\psi(\lambda
y)=\lambda\psi(y)=0$, hence $\psi\hat y=0$ for each $\psi\in S$. It follows that
$S\cdot\CC\hat y=0$, hence $S$ cannot be a $\bot$-mono in $\mathbf{FdOS}_0$. 
\end{proof}

In light of Lemma \ref{lem:alternative-description-zero-monos}, a quantaloid without zero object cannot have kernels of $\bot$-monos, leading to:

\begin{definition}
    Let $\mathbf Q$ be a dagger quantaloid. If the dagger equalizer $k$ of a morphism $f:X\to Y$ and $\perp_{X,Y}$ exists, we call it the \emph{dagger kernel} of $f$. We say that $\mathbf Q$ \emph{has dagger kernels} or that $\mathbf Q$ is a \emph{dagger kernel quantaloid} if the dagger kernel of any morphism that is not a $\bot$-mono exists.
\end{definition}
It follows immediately from Lemma \ref{lem:quantaloid with zero} that the definition of a dagger kernel above coincides with the definition of a dagger kernel category if $\mathbf Q$ has a zero object.

The next proposition shows that dagger kernel quantaloids generalize dagger quantaloids that are simultaneously dagger kernel categories, and justifies the terminology in the previous definition.

\begin{proposition}\label{prop:dagger kernel quantaloid}
Let $\mathbf Q$ be a dagger quantaloid. 
\begin{itemize}
    \item[(a)] If $\mathbf Q$ has a zero object, then it is a dagger kernel quantaloid if and only if it is a dagger kernel category, i.e., if $\mathbf Q$ is a dagger kernel quantaloid, then the dagger kernel of any morphism in $\mathbf Q$ exists.
\item[(b)] If $\mathbf Q$ does not have a zero object, then it is a dagger kernel quantaloid if and only if $\mathbf Q_0$ is a dagger kernel category.
\end{itemize}
\end{proposition}
\begin{proof}
    For (a), assume $\mathbf Q$ has a zero object. Then $0_{X,Y}=\perp_{X,Y}$ for each object $X,Y\in\mathbf Q$ (cf. Lemma \ref{lem:quantaloid with zero}), and zero-monos and $\bot$-monos coincide (cf. Lemma \ref{lem:perp-mono-in-Qiszero-mono-inQ0}). As a consequence, if $\mathbf Q$ is a dagger kernel category, then clearly it is a dagger kernel quantaloid. Conversely, if $\mathbf Q$ is a dagger kernel quantaloid, showing that $\mathbf Q$ is a dagger kernel category only requires showing that the dagger kernel of any zero-mono $f:X\to Y$ in $\mathbf Q$ exists, but this follows from Lemma \ref{lem:alternative-description-zero-monos}. 

    For (b), assume that $\mathbf Q$ does not have a zero object. By (a), it suffices to show that $\mathbf Q$ is a dagger kernel quantaloid if and only if $\mathbf Q_0$ is a dagger kernel quantaloid. 

    Assume that $\mathbf Q$ is a dagger kernel quantaloid, and let $f:X\to Y$ be a morphism in $\mathbf Q_0$ that is not a $\bot$-mono in $\mathbf Q_0$.  Assume first that $X\neq 0\neq Y$, so $f\in\mathbf Q(X,Y)$. Since $f$ is not a $\bot$-mono in $\mathbf Q_0$, there is some $Z\in\mathbf Q_0$ and $g\in\mathbf Q_0(Z,X)\setminus\{\perp_{Z,X}\}$ such that $f\circ g=\perp_{Z,Y}$. The condition $g\neq\perp_{Z,X}$ forces $Z\neq 0$, and hence $g\in \mathbf Q(Z,X)$. As a consequence, $f$ is also a $\bot$-mono in $\mathbf Q$, hence the dagger equalizer $k:K\to X$ of $f$ and $\perp_{X,Y}$ exists in $\mathbf Q$. Let $Z\in\mathbf Q_0$ and $g:Z\to X$ another morphism such that $f\circ g=\perp_{X,Y}\circ g$, i.e., such that $f\circ g=\perp_{Z,Y}$. Since $k$ is the dagger equalizer of $f$ and $\perp_{X,Y}$ in $\mathbf Q$, it follows that there is some $h\in\mathbf Q(Z,K)=\mathbf Q_0(Z,K)$ such that $k\circ g=h$. If $Z=0$, then $g=\perp_{0,X}$, hence there is also some $h\in \mathbf Q_0(Z,K)$ with $k\circ g=h$, namely $h=\perp_{Z,K}$. 

Now assume that $X=0$ or $Y=0$. Then we must have $f=\perp_{X,Y}$, hence the equalizer of $f$ and $\perp_{X,Y}$ is $\id_X\in\mathbf Q_0(X,X)$, which is clearly a dagger mono. We conclude that $\mathbf Q_0$ is a dagger kernel quantaloid.

Now assume $\mathbf Q_0$ is a dagger kernel quantaloid. Let $f:X\to Y$ be a morphism in $\mathbf Q$ that is not a $\bot$-mono in $\mathbf Q$. Hence, there is some $Z\in\mathbf Q$ and some $g\in\mathbf Q(Z,X)\setminus\{\perp_{Z,X}\}$ such that $f\circ g=\perp_{Z,Y}$.
Clearly, this implies that $f$ is also a $\bot$-mono in $\mathbf Q_0$, so the dagger equalizer $k:K\to X$ of $f$ and $\perp_{X,Y}$ in $\mathbf Q_0$ exists. Now, we cannot have $K=0$, because $f\circ g=\perp_{Z,Y}=\perp_{X,Y}\circ g$ implies the existence of some $h:Z\to 0$ such that $k\circ h=g$, which forces $g=0_{Z,X}=\perp_{Z,X}$. As a consequence, $K\in\mathbf Q$, so the dagger equalizer of $f$ and $\perp_{X,Y}$ in $\mathbf Q$ exists.
\end{proof}

\begin{lemma}\label{lem:dagger-kernel-quantale}
   The associated quantaloid $\mathbf V$ of a commutative quantale $(V,\cdot,e)$ (cf. Example \ref{ex:quantale-induced quantaloid}) is a dagger kernel quantaloid if and only if for each $v,w\neq\perp$ in $V$ we have $v\cdot w\neq\perp$.  
\end{lemma}
\begin{proof}
Recall that $v^\dag=v$ for each $v\in\mathbf V$. We first show that the dagger kernel of $v=\perp$ in $\mathbf V$ always exists, and equals $e$. Since $e^\dag\circ e=e$, it follows that $e$ is a dagger mono. We have $v\cdot e=\perp\cdot e=\perp$. Furthermore, for each $w\in V$, we have $v\cdot w=\perp\cdot w=\perp$, but $w$ factorizes via $e$: $w=e\cdot w$, so $e=\ker(\perp)$. Thus, $\mathbf V$ is a dagger kernel quantaloid if every $v\neq\perp$ that is not a $\bot$-mono has a dagger kernel. Now, the condition $v\cdot w\neq\perp$ for each $v,w\neq\perp$ in $V$ implies that any $v\neq\perp$ is a $\bot$-mono, so in this case $\mathbf V$ is a dagger kernel quantaloid. We prove the converse implication by contraposition. Assume that there are $v,w\neq \perp$ in $V$ such that $v\cdot w=\perp$. Thus $v$ is not a $\bot$-mono. Assume that there is a dagger mono $k\in V$ such that $v\cdot k=\perp$. The condition that $k$ is a dagger mono translates to $k\cdot k=e$. Then $\perp=v\cdot k\cdot k=v\cdot e=v$, contradicting that $v\neq \perp$, hence $v$ does not have a dagger kernel. So $\mathbf V$ is not a dagger kernel quantaloid.
\end{proof}

\begin{example}
Frames $F$ in which $v\wedge w\neq \perp$ for each $v,w\neq \perp$ in $F$ are pointfree generalizations of \emph{hyperconnected} topological spaces. The only Boolean algebra that has this property is the two-element Boolean algebra $2$.
\end{example}

\begin{example}
    Let $(M,\cdot,e)$ be a commutative monoid. For each $A,B\subseteq M$ define $A\cdot B=\{a\cdot b:a\in A,b\in B\}$. Then $(P(M),\cdot ,\{e\})$ is a commutative quantale such that $A\cdot B\neq \emptyset$ for each $A,B\neq\emptyset$ in the power set $P(M)$. 
\end{example}

\begin{lemma}\label{lem:FdOS-dagger-kernels}
$\mathbf{FdOS}$ and $\mathbf{FdOS}_0$ (cf. Example \ref{ex:FdOS}) are dagger kernel quantaloids. 
\end{lemma}
\begin{proof}
By Proposition \ref{prop:dagger kernel quantaloid}, it suffices to show that $\mathbf{FdOS}_0$ is a dagger kernel category. 

    Let $R:X\to Y$ be a morphism in $\mathbf{FdOS}_0$, so a subspace of $B(X,Y)$. Let $K=\bigcap_{r\in R}\ker r$. Then $K$ is a subspace of $X$, so we can denote $k:K\to X$ the inclusion in $\mathbf{FdHilb}$. Then $E=\CC k$ is a morphism in $\mathbf{FdOS}_0(K,X)$, which we claim to be the dagger kernel of $R$. Firstly, $E$ is a dagger mono, because for each $x,y\in K$, we have $\langle k^\dag kx,y\rangle=\langle kx,ky\rangle=\langle x,y\rangle$, so $k^\dag k=1_K$. Hence, $E^\dag\cdot E=\CC k^\dag k=\CC 1_K=\id_K$. It follows from $k^\dag k=1_K$ that $p:=kk^\dag:X\to X$ is a projection in $\mathbf{FdHilb}$, i.e., $p^2=p=p^\dag$. If $x\in K$, then $px=pkx=kk^\dag k x=kx=x$. Let $x\in X$ such that $px=x$. Let $r\in R$. Then for each $z\in K\subseteq\ker r$, we have $0=rz=rkz$, hence $rk=0$.
    As a consequence, we have $rx=rpx=rkk^\dag x=0$, so $x\in\ker r$. Since $r\in R$ is arbitrary, we obtain $x\in\bigcap_{r\in R}\ker r=K$. Thus, for each $x\in X$, we have $x\in K$ if and only if $px=x$. 
    
    We found that $rk=0$ for each $r\in R$. Hence,  $R\cdot E=\{rk:r\in R\}=0$. Now, let $S:Z\to X$ be another morphism in $\mathbf{FdOS}$ with $R\cdot S=0$. Let $T\in\mathbf{FdOS}_0(Z,K)$ be defined by $T=\{k^\dag s:s\in S\}$. Fix $s\in S$ and $z\in Z$. Then, for each $r\in R$, we have $rs=0$, hence $r(s(z))=0$, so $s(z)\in\ker r$. Thus $s(z)\in\bigcap_{r\in R}\ker r=K$. Thus $ps(z)=s(z)$ for each $z\in Z$, whence $ps=s$. As a consequence, we have $E\cdot T=\{k t:t\in T\}=\{kk^\dag s:s\in S\}=\{ps:s\in S\}=\{s\in S\}=S$. We conclude that $E:K\to X$ is the dagger kernel of $R$ in $\mathbf{FdOS}_0$. 
\end{proof}

\begin{lemma}\label{lem:r-zero-mono-iff-rdagr-zero-mono}
    A morphism $r:X\to Y$ in a dagger kernel quantaloid $\mathbf Q$ is a $\bot$-mono if and only if $r^\dag\circ r$ is a $\bot$-mono. 
\end{lemma}
\begin{proof}
   Assume $r$ is a $\bot$-mono in $\mathbf Q$, and let $s:Z\to X$ be a morphism in $\mathbf Q$ such that $r^\dag\circ r\circ s=\perp_{Z,X}$. By Lemma \ref{lem:perp-mono-in-Qiszero-mono-inQ0}, $r$ is a zero-mono in $\mathbf Q_0$, which is a dagger kernel category by Proposition \ref{prop:dagger kernel quantaloid}. By Lemma \ref{lem:quantaloid with zero}, we have $r^\dag\circ r\circ s=0_{Z,X}$ in $\mathbf Q_0$, so $(r\circ s)^\dag\circ (r\circ s)=s^\dag\circ r^\dag\circ r\circ s=0_Z$. By Lemma \ref{lem:nondegeneracy} we have $r\circ s=0_{Z,Y}$. Since $r$ is a zero-mono, we obtain $s=0_{Z,X}$, so $r^\dag\circ r$ is a zero-mono in $\mathbf Q_0$. By Lemma \ref{lem:perp-mono-in-Qiszero-mono-inQ0}, it follows that $r^\dag\circ r$ is a $\bot$-mono in $\mathbf Q$.

    Conversely, assume that $r^\dag\circ r$ is a $\bot$-mono, and let $s:Z\to X$ be a morphism such that $r\circ s=\perp_{Z,Y}$. Then $r^\dag\circ r\circ s=\perp_{Z,X}$, and since $r^\dag\circ r$ is a $\bot$-mono, it follows that $s=\perp_{Z,X}$, so $r$ is a $\bot$-mono. 
\end{proof}

\subsection{Orthomodularity}
An \emph{ortholattice} is a bounded lattice $L$ equipped with an involutive order-reversing map $x\mapsto \neg x$ such that $x\wedge \neg x=0$ and $x\vee \neg x=1$, where $0$ and $1$ denote the respective least and greatest element of $L$. If, in addition, for each $x,y\in L$ we have that $x\leq y$ implies $y=x\vee(\neg x\wedge y)$, we call $L$ an \emph{orthomodular lattice}. Orthomodular lattices generalize \emph{modular} ortholattices, i.e., ortholattices $L$ that satisfy the modular law: $x\leq z$ implies $x\vee(y\wedge z)=(x\vee y)\wedge z$ for each $x,y,z\in L$. 

The homsets of $\mathbf{Rel}$ and $\mathbf{qRel}$ are complete orthomodular lattices. In fact, the homsets of the former category are even Boolean algebras, whereas the homsets of the latter are complete modular ortholattices.

Affine dagger symmetric monoidal quantaloids with orthomodular homsets have the following property:
\begin{lemma}\label{lem:orthomodularity+unique-perp-monic-effects-implies-affine-and-boolean-scalars}
Let $(\mathbf Q,\otimes, I)$ be an affine dagger symmetric monoidal quantaloid such that every homset of $\mathbf Q$ is an orthomodular lattice. Then $(\mathbf Q(I,I),\circ,\id_I)$ is a complete Boolean algebra with $r\wedge s=r\circ s$ for each $r,s\in\mathbf Q(I,I)$.   
\end{lemma}
\begin{proof}
  Let $V=\mathbf Q(I,I)$.
We already saw in \ref{subsec:scalars of quantaloids} that $(V,\circ,\id_I)$ is a commutative quantale. By definition of affine symmetric monoidal quantaloids, we have $\id_I=\top_I$. Let $r,s\in V$. Then $r\circ s\leq r\circ \top_I= r\circ \id_I=r$ and similarly, $r\circ s\leq s$, so $r\circ  s\leq r\wedge s$. Hence, $\neg r\circ r\leq \neg r\wedge r=\perp_I$, forcing $\neg r\circ r=\perp_I$. Then $r=\id_I\circ r=\top_I\circ r=(r\vee \neg r)\circ r=(r\circ r)\vee (\neg r\circ r)=(r\circ r)\vee \perp_I=r\circ r$, so each $r\in V$ is idempotent. By Lemma \ref{lem:affine-idempotent-quantale-is-frame}, it follows that $V$ is a frame with $r\wedge s=r\circ s$ for each $r,s\in V$, hence it satisfies the distributivity property $r\wedge (s\vee t)=(r\wedge s)\vee(r\wedge t)$ for each $r,s,t\in V$. Now, since $V$ is an ortholattice satisfying this distributivity property, it is a Boolean algebra. Completeness follows by definition of a quantale. 
\end{proof}

If $\mathbf Q$ is a dagger compact quantaloid such that every object of $\mathbf Q$ has precisely one $\bot$-monic effect, then the existence of dagger kernels is sufficient to show that homset of $\mathbf Q$ are orthomodular. We will see later that $\mathbf{qRel}$ is an example of a dagger kernel quantaloid (cf. Proposition \ref{prop:qRel is a dagger kernel category}). We further note that in an orthomodular lattice, there is an orthogonality relation $\perp$ defined by $x\perp y$ if and only if $x\leq \neg y$.

We first start with the case of effects, for which we do not need compactness.

\begin{definition}\label{def:orthogonality effects}
    Let $(\mathbf Q,\otimes,I)$ be a dagger symmetric monoidal category. Then for each object $X$ of $\mathbf Q$, we define a binary relation $\perp$ on the set of effects $\mathbf Q(X,I)$ by $r\perp s$ if and only if $r\circ s^\dag=\perp_I$. 
\end{definition}

\begin{lemma}\label{lem:characterizations of negation}
    Let $(\mathbf Q,\otimes,I)$ be a dagger symmetric monoidal quantaloid with dagger kernels and a zero object. Let $r:X\to I$ be an effect with dagger kernel $k:K\to X$ (existence assured by Proposition \ref{prop:dagger kernel quantaloid}). Define
    $\neg r:X\to I$ by 
    \[\neg r:=\bigvee\{s\in\mathbf Q(X,I):r\perp s\}.\]
      Then:
    \begin{itemize}
    \item[(a)]    For any two effects $r,s:X\to I$, we have $s\leq \neg r$ if and only if $r\perp s$;
    \item[(b)] For any effect $s:X\to I$ we have $r\perp s$ if and only if $s=t\circ k^\dag$ for some effect $t:K\to I$.
    \item[(c)] $\neg r=\top_{K,I}\circ k^\dag$.
    \end{itemize}
\end{lemma}
\begin{proof}
By Lemma \ref{lem:quantaloid with zero}, $0_{X,Y}=\perp_{X,Y}$ for any two objects $X$ and $Y$ of $\mathbf Q$. If $r\perp s$ then $s\leq\neg r$ by definition of $\neg r$. Assume that $s\leq \neg r$ Then, since $\mathbf Q$ is a quantaloid, we have 
\begin{align*}
s\circ r^\dag & \leq \neg r\circ r^\dag = \left(\bigvee\{t\in\mathbf Q(X,I):r\perp t\}\right)\circ r^\dag  = \bigvee\{t\circ r^\dag:t\in\mathbf Q(X,I):r\perp t\} \\
& = \bigvee\{t\circ r^\dag:t\in\mathbf Q(X,I):t\circ r^\dag=\perp_I\}=\perp_I,
\end{align*}
which forces $s\circ r^\dag=\perp_I$. Thus $r\perp s$.

    For (b), if $s=t\circ k^\dag$, then $r\circ s^\dag=r\circ k\circ t^\dag=\perp_{K,I}\circ t^\dag=\perp_{I,I}=\perp_I$. Conversely, if $r\circ s^\dag=\perp_I$, then by the universal property of dagger kernels, there is a morphism $v:I\to K$ such that $k\circ v=s^\dag$. Choosing $t=v^\dag$ now yields $s=t\circ k^\dag$.

    Finally, for (c), it follows from (b) that $\neg r=\bigvee\{s\in\mathbf Q(X,I):r\perp s\}=\bigvee\{t\circ k^\dag:t\in\mathbf Q(K, I)\}=\left(\bigvee\mathbf Q(K,I)\right)\circ k^\dag=\top_{K,I}\circ k^\dag$.\end{proof}

\begin{lemma}\label{lem:effects in terms of kernels}
    Let $(\mathbf Q,\otimes,I)$ be a  dagger symmetric monoidal quantaloid with dagger kernels and a zero object such that every object of $\mathbf Q$ has precisely one $\bot$-monic effect. Then any $r:Y\to I$ equals $\top_{K^\perp,I}\circ k_\perp^\dag$, where $k:K\to Y$ is the dagger kernel of $r$ (existence assured by Proposition \ref{prop:dagger kernel quantaloid}).
\end{lemma}
\begin{proof}
By Lemma \ref{lem:daggerkernel-zeroepi-factorization}, in which we take $X=I$ and $f=r^\dag$, we have $r^\dag=\ker(k^\dag)\circ e$ for some zero-epi $e:I\to K^\perp$. Since $k_\perp=\ker(k^\dag)$, we obtain $r=e^\dag\circ k_\perp^\dag$. Since $e$ is a zero-epi, it follows that $e^\dag$ is a zero-mono, hence by Lemma \ref{lem:perp-mono-in-Qiszero-mono-inQ0}, $e^\dag$ is a $\bot$-mono. By Lemma \ref{lem:zeromonoeffect implies topzeromono}, also $\top_{K^\perp,I}$ is a $\bot$-mono. By assumption, there is precisely one $\bot$-monic $K^\perp\to I$, whence $e^\dag$ must equal $\top_{K^\perp,I}$.
\end{proof}

\begin{proposition}\label{prop:effects on one object form OML}
    Let $(\mathbf Q,\otimes, I)$ be a  dagger symmetric monoidal quantaloid such that $\mathbf Q$ is a dagger kernel category and such that every object of $\mathbf Q$ has precisely one $\bot$-monic effect. Let $X$ be an object of $\mathbf Q$. Then $\mathrm{KSub}(X)$ and $\mathbf Q(X,I)$ are ortho-isomorphic complete orthomodular lattices, where the orthocomplementation of the latter is the map $r\mapsto \neg r$ of Lemma \ref{lem:characterizations of negation}. The ortho-isomorphism
    $\mathbf Q(X,I)\to\mathrm{KSub}(X)$ is given by $r\mapsto [\ker(r)_\perp]$.
\end{proposition}
\begin{proof}
By Proposition \ref{prop:KSub is OML} it follows that  $\mathrm{KSub}(X)$ is an orthomodular lattice. Since $\mathbf Q$ is a dagger kernel category, it has a zero object, hence $\perp_{X,Y}=0_{X,Y}$ for each $X,Y\in\mathbf Q$ by Lemma \ref{lem:quantaloid with zero}, and zero-monos are $\bot$-monos by Lemma \ref{lem:perp-mono-in-Qiszero-mono-inQ0}.
    We claim that the map $\varphi:\mathbf Q(X,I)\to\mathrm{KSub}(X)$, $r\mapsto[\ker(r)_\perp]$ is an order isomorphism such that $\varphi(\neg r)=\neg\varphi(r)$ for each $r\in\mathbf Q(X,I)$. Since $\mathrm{KSub}(X)$ is an orthomodular lattice, it then follows that $r\mapsto\neg r$ defines an orthocomplementation on $\mathbf Q(X,I)$ such that $\mathbf Q(X,I)$ is an orthomodular lattice. Completeness of $\mathbf Q(X,I)$ follows since $\mathbf Q$ is a quantaloid. Note that once the ortho-isomorphism between $\mathbf Q(X,I)$ and $\mathrm{KSub}(X)$ is established, completeness of the former implies completeness of the latter.
    
    In order to show that $\varphi$ is an ortho-isomorphism, we first check that $\varphi$ is monotone. So let $r,s:X\to I$, and let $K_r$ and $K_s$ be the domains of $\ker(r)$ and $\ker(s)$, respectively. Assume that  $r\leq s$, then $r\circ\ker (s)\leq s\circ\ker (s)=0_{K_s,I}=\perp_{K_s,I}$, which forces $r\circ \ker(s)=\perp_{K_s,I}$. It follows from the universal property of dagger kernels that there must be some $a:K_s\to K_r$ such that $\ker(s)=\ker(r)\circ a$, hence $[\ker (s)]\leq[\ker(r)]$ in $\mathrm{KSub}(X)$, implying $\varphi(r)=[\ker(r)_\perp]\leq[\ker(s)_\perp]=\varphi(s)$. Next, we show that $\varphi$ is an order embedding. So assume that $\varphi(r)\leq\varphi(s)$, i.e., $[\ker(r)_\perp]\leq[\ker (s)_\perp]$. In other words, $\ker(r)_\perp=\ker(s)_\perp\circ a$ for some morphism $a:K_r^\perp\to K_s^\perp$, which is necessarily a dagger mono, see \ref{sec:dagger kernel}. Using Lemma \ref{lem:effects in terms of kernels}, we obtain $r=\top_{K_r^\dag,I}\circ\ker(r)_\perp^\dag=\top_{K_r^\dag,I}\circ a^\dag\circ\ker(s)_\perp^\dag\leq\top_{K_s^\perp,I}\circ\ker(s)_\perp^\dag=s$, so $\varphi$ is indeed an order embedding. In order to show that it is an order isomorphism, we only have to show it is surjective. So let $k:K\to I$ in $\mathrm{KSub}(X)$, and let $r=T_{K,I}\circ k^\dag$. Now,  $\top_{K,I}$ is a $\bot$-mono by Lemma \ref{lem:zeromonoeffect implies topzeromono}, so a zero-mono. Moreover, we have $\ker(m\circ f)=\ker (f)$ for each morphism $f$ and each zero-mono $m$ by \cite[Lemma 4.2]{heunenjacobs}. Hence, $\ker (r)=\ker(T_{K,I}\circ k^\dag)=\ker(k^\dag)=k_{\perp}$, which implies $\varphi(r)=[\ker(r)_\perp]=[k_{\perp\perp}]=[k]$. 

Finally, for arbitrary $r:X\to I$, we have $\neg r=\top_{K,I}\circ(\ker r)^\dag$ by Lemma \ref{lem:characterizations of negation}. Again using that $\top_{K,I}$ is a zero-mono, we obtain $\varphi(\neg r)=[\ker(\neg r)_\perp]=\neg[\ker(\neg r)]=\neg[\ker(\top_{K,I}\circ (\ker r)^\dag)]=\neg[\ker(\ker r)^\dag]=\neg[(\ker r)_\perp]=\neg\varphi(r)$. Finally, it follows from (a) of Lemma \ref{lem:characterizations of negation} that $\perp$ is the associated orthogonality relation of the orthocomplementation $\neg$ on $\mathbf Q(X,I)$. 
\end{proof}
We note that we never assumed our categories to be well powered, so a priori, there is no guarantee that $\mathrm{KSub}(X)$ in $\mathbf Q$ is a set. However, since $\mathbf Q$ is a quantaloid, it is locally small, and the theorem above establishes a bijection between $\mathbf Q(X,I)$ and $\mathrm{KSub}(X)$, which assures that the latter is indeed a set.

\begin{corollary}
    Let $(\mathbf Q,\otimes,I)$ be a  dagger symmetric monoidal quantaloid with dagger kernels and a zero object such that every object of $\mathbf Q$ has precisely one $\bot$-monic effect. Then the set $\mathbf Q(I,X)$ of states on any object $X$ is a complete orthomodular lattice.
\end{corollary}

\begin{corollary}\label{cor:QdaggerkernelsimplyMatrQdaggerkernels}
    Let $(\mathbf Q,\otimes,I)$ be a  dagger symmetric monoidal quantaloid with dagger kernels  such that every object in $\mathbf Q$ has precisely one $\bot$-monic effect. Then $\mathrm{Matr}(\mathbf Q)$ is a  dagger symmetric monoidal quantaloid with dagger kernels such that for each object $X$ of $\mathrm{Matr}(\mathbf Q)$, there is precisely one $\bot$-monic effect, namely $\top_{X,J}$. 
\end{corollary}

\begin{proof}
By Proposition \ref{prop:MatrQ0isMatrQ}, we may assume that $\mathbf Q$ has a zero object, hence by Proposition \ref{prop:dagger kernel quantaloid} it follows that $\mathbf Q$ is a dagger kernel category. The latter proposition also assures that it suffices to show that $\mathrm{Matr}(\mathbf Q)$ is a dagger kernel category. By Lemma \ref{lem:unique-zero-monic-effects-in-MatrQ}, each object $X$ of $\mathrm{Matr}(\mathbf Q)$ has precisely one $\bot$-monic effect, namely $\top_{X,J}$, hence it follows from Proposition \ref{prop:effects on one object form OML} that $\mathrm{KSub}(X)$ is a complete orthomodular lattice for each object $X$ of $\mathbf Q$.

Let $X=(X_\alpha)_{\alpha\in A}$, $Y=(Y_\beta)_{\beta\in B}$ be objects of $\mathrm{Matr}(\mathbf Q)$, and let $f=(f_\alpha^\beta)_{(\alpha,\beta)\in A\times B}:X\to Y$ be a morphism. We show that the dagger kernel of $f$ exists.
For each $\alpha\in A$ and each $\beta\in B$, the dagger kernel $k_{\alpha,\beta} :K_{\alpha,\beta}\to X_\alpha$ of $f_{\alpha}^\beta:X_\alpha\to Y_\beta$ in $\mathbf Q$ exists. Fix $\alpha\in A$. Since $\mathrm{KSub}(X_\alpha)$ is a complete orthomodular lattice, it follows that $k_\alpha:=\inf_{\beta\in B}k_{\alpha,\beta}$ in $\mathrm{KSub}(X_\alpha)$ exists. Let $K_\alpha$ be the domain of $k_\alpha$, so $k_\alpha:K_\alpha\to X_\alpha$. Then for each $\beta\in B$, we have the inequality $[k_\alpha]\leq [k_{\alpha,\beta}]$ in $\mathrm{KSub}(X_\alpha)$, so there exists a $g_\beta: K_\alpha\to K_{\alpha,\beta}$ such that $k_\alpha=k_{\alpha,\beta}\circ g_\beta$. As a consequence, for each $\beta\in B$, we have 
\begin{equation}\label{eq:daggerkernelMatrQ}f_{\alpha}^\beta\circ k_\alpha=f_\alpha^\beta\circ k_{\alpha,\beta}\circ g_\beta=0_{K_{\alpha,\beta},Y_\beta}\circ g_\beta=0_{K_\alpha,Y_\beta}.
\end{equation}

Let $K=(K_\alpha)_{\alpha\in A}$, and let $k:K\to X$ in $\mathrm{Matr}(\mathbf Q)$ be given by 
\[ k_{\alpha}^{\alpha'} = \begin{cases} k_\alpha, & \alpha=\alpha',\\
 0_{K_\alpha,X_{\alpha'}}, & \alpha\neq\alpha'
 \end{cases}\]
 for each $\alpha,\alpha'\in A$.

For the following, we will frequently apply  Lemma \ref{lem:quantaloid with zero}, which states that in a quantaloid with a zero object, $\perp_{A,B}=0_{A,B}$ for any two objects. Using this, it is straightforward to see that $(k^\dag\circ k)_\alpha^{\alpha'}=(\id_{K})_\alpha^{\alpha'}$ for each $\alpha,\alpha'\in A$, so $k^\dag\circ k=\id_{K}$, i.e., $k$ is a dagger mono.
Using Equation (\ref{eq:daggerkernelMatrQ}) it also readily follows that $(f\circ k)_\alpha^\beta=\bigvee_{\alpha'\in A}f_{\alpha'}^\beta\circ k_\alpha^{\alpha'}=0_{K_\alpha,Y_\beta}$ for each $\alpha\in A$ and each $\beta\in B$. So $f\circ k=0_{K,Y}$.

 Let $Z=(Z_\gamma)_{\gamma\in C}$ be an object of $\mathrm{Matr}(\mathbf Q)$, and let  $h=(h_\gamma^\alpha)_{(\gamma,\alpha)\in C\times A}:(Z_\gamma)_{\gamma\in C}\to (X_\alpha)_{\alpha\in A}$ be  morphism such that $f\circ h=0_{Z,Y}$. Fix $\alpha\in A$,  and $\gamma\in C$. Using Lemma \ref{lem:daggerkernel-zeroepi-factorization}, we find some $m_\gamma^\alpha:N_\gamma^\alpha\to X_\alpha$ in $\mathrm{KSub}(X_\alpha)$ and some zero-epi $e_\gamma^\alpha:Z_\gamma\to N_\gamma^\alpha$ such that $m_\gamma^\alpha\circ e_\gamma^\alpha=h_\gamma^\alpha$.
Then for each $\beta\in B$, we have

 \[ \begin{tikzcd}
  & K_{\alpha,\beta}\ar{dr}{k_{\alpha,\beta}} &&
  \\
 K_\alpha\ar{rr}{k_\alpha}\ar{ur}{g_\beta} &&  X_\alpha\ar{r}{f_\alpha^\beta} & Y_\beta \\
 & Z_\gamma\ar{ur}{h_\gamma^\alpha}\ar{r}[swap]{e_\gamma^\alpha} & N_\gamma^\alpha\ar{u}[swap]{m_\gamma^\alpha}&
\end{tikzcd} \]

 Now, fix also $\beta\in B$. Then
 \[\perp_{Z_\gamma,Y_\beta}=0_{Z_\gamma,Y_\beta}=(f\circ h)_\gamma^\beta=\bigvee_{\alpha'\in A}f_{\alpha'}^\beta\circ h_\gamma^{\alpha'},\]
 which forces $f_\alpha^\beta\circ h_\gamma^\alpha=\perp_{Z_\gamma,Y_\beta}=0_{Z_\gamma,Y_\beta}$.
  
Now, since $e_\gamma^\alpha$ is a zero-epi and $0_{Z_\gamma,Y_\beta}=f_\alpha^\beta\circ h_\gamma^\alpha=f_\alpha^\beta\circ m_\gamma^\alpha\circ e_\gamma^\alpha$, we obtain $f_\alpha^\gamma\circ m_\gamma^\alpha=0_{N_{\gamma}^\alpha,Y_\beta}$. Since $k_{\alpha,\beta}=\ker(f_\alpha^\beta)$,  there must be some $n_\gamma^\alpha:N_\gamma^\alpha\to K_{\alpha,\beta}$ such that $m_\gamma^\alpha=k_{\alpha,\beta}\circ n_\gamma^\alpha$. This precisely expresses that $[m_\gamma^\alpha]\leq [k_{\alpha,\beta}]$ in $\mathrm{KSub}(X_\alpha)$. By definition of $k_\alpha$, we obtain $[m_\gamma^\alpha]\leq [k_\alpha]$, hence there must be some $s_\gamma^\alpha:N_\gamma^\alpha\to K_\alpha$ such that $m_\gamma^\alpha=k_\alpha\circ s_\gamma^\alpha$.
Let $r_\gamma^\alpha:Z_\gamma\to K_\alpha$ be given by $r_\gamma^\alpha=s_\gamma^\alpha\circ e_\gamma^\alpha$. Then $r=(r_\gamma^\alpha)_{(\gamma,\alpha)\in C\times A}:Z\to K$ is a morphism in $\mathrm{Matr}(\mathbf Q)$ such that for each $\gamma\in C$ and $\alpha\in A$ we have
\[ (k\circ r)_\gamma^\alpha=\bigvee_{\alpha'\in A}k_{\alpha'}^\alpha\circ r_{\gamma}^{\alpha'}=k_\alpha\circ r_\gamma^\alpha=k_\alpha\circ s_\gamma^\alpha\circ e_\gamma^\alpha=m_\gamma^\alpha\circ e_\gamma^\alpha=h_\gamma^\alpha,\]
so $h=k\circ r$. We conclude that $k$ is the dagger kernel of $f$ in $\mathrm{Matr}(\mathbf Q)$.
\end{proof}

\begin{proposition}\label{prop:zero mono effects in qRel}\label{prop:qRel is a dagger kernel category}
    The category $\mathbf{qRel}$ is a binary dagger compact quantaloid with dagger kernels such that each quantum set $\X$ has precisely one $\bot$-monic effect, namely $\top_{\X,\mathbf 1}$.
\end{proposition}
\begin{proof}
By definition, we have $\mathbf{qRel}=\mathrm{Matr}(\mathbf{FdOS})$, which is a binary dagger compact quantaloid by Example \ref{ex:qRel-dagger-compact}. It follows from Proposition \ref{prop:MatrQ0isMatrQ} that $\mathbf{qRel}\cong \mathrm{Matr}(\mathbf{FdOS}_0)$. Since $\mathbf{FdOS}_0$ is a dagger kernel quantaloid (cf. Lemma \ref{lem:FdOS-dagger-kernels}) such that for each object $X$ in $\mathbf{FdOS}_0$ there is only one $\bot$-monic effect $X\to\mathbb C$ (cf. Lemma \ref{lem:zero mono effects in FdOS0}), the statement follows from Corollary \ref{cor:QdaggerkernelsimplyMatrQdaggerkernels}.
\end{proof}

Since $\Tr(s)=s$ for any scalar $s$ in a dagger compact category  (cf. Proposition \ref{prop:trace properties}), it follows that the definition of $\perp$ and $\neg$ in the next theorem generalizes the definition of $\perp $ and $\neg$ on sets of effects in Proposition \ref{prop:effects on one object form OML}.

\begin{theorem}\label{thm:dagger kernels imply orthomodularity}
    Let $\mathbf Q$ be a dagger compact quantaloid with dagger kernels such that every object has precisely one $\bot$-monic effect. Then $\mathbf Q$ is affine, and for any two objects $X$ and $Y$ in $\mathbf Q$, the homset $\mathbf Q(X,Y)$ is a complete orthomodular lattice with orthocomplementation $r\mapsto \neg r$ given by $\neg r=\bigvee\{s\in \mathbf Q(X,Y):r\perp s\}$, where the orthogonality relation $\perp$ on $\mathbf Q(X,Y)$ is given by $r\perp s$ if and only if $\Tr(r\circ s^\dag)=\perp_I$.  Moreover, the map $\mathbf Q(X,Y)\to\mathbf Q(X\otimes Y^*,I)$, $r\mapsto \coname {r}$ is an ortho-isomorphism.
\end{theorem}
\begin{proof}
We first assume that $\mathbf Q$ has a zero object. By Proposition \ref{prop:dagger kernel quantaloid}, $\mathbf Q$ is a dagger kernel category, hence it follows from 
Proposition \ref{prop:effects on one object form OML} that $\mathbf Q(X\otimes Y^*,I)$ is a complete orthomodular lattice. We consider the order isomorphism $\mathbf Q(X,Y)\to\mathbf Q(X\otimes Y^*,I)$, $r\mapsto\coname{r}=\epsilon_Y\circ(r\otimes \id_{Y^*})$ from Lemma \ref{lem:counit in compact quantaloid is order iso}. Then for $r,s:X\to Y$, we find
$\Tr(r\circ s^\dag)=\epsilon_Y\circ((r\circ s^\dag)\otimes \id_{Y^*})\circ\epsilon_Y^\dag=\epsilon_Y\circ(r\otimes\id_{Y^*})\circ(s^\dag\circ\id_{Y^*})\circ\epsilon_Y^\dag=\epsilon_Y\circ(r\otimes\id_{Y^*})\circ(\epsilon_Y\circ(s\otimes\id_{Y^*}))^\dag=\coname{r}\circ\coname{s}^\dag$. It follows that $r\perp s$ if and only if $\Tr(r\circ s^\dag)=\perp_I$ if and only if $\coname{r}\circ\coname{s}^\dag=\perp_I$. 
Hence, using that $\coname{-}$ is an order isomorphism, we obtain
\begin{align*}\coname{\neg r} & =\raisebox{-0.3em}{\(\llcorner\)}\bigvee\{s\in\mathbf Q(X, Y):r\perp s\}\raisebox{-0.3em}{\(\lrcorner\)} =\bigvee\{\coname{s}:s\in\mathbf Q(X,Y),r\perp s\}\\
& =\bigvee\{\coname{s}:s\in\mathbf Q(X,Y),\coname{r}\circ\coname{s}^\dag=\perp_I\} =\bigvee\{t\in\mathbf Q(X\otimes Y^*,I):\coname{r}\circ t^\dag=\perp_I\}\\
& = \bigvee\{t\in \mathbf Q(X\otimes Y^*),\coname{r}\perp t\} =\neg\phantom{.}{\coname{r}}.
\end{align*}
So, $\coname{-}$ preserves the orthocomplementation, hence it is an ortho-isomorphism. It follows that $\mathbf Q(X,Y)$ inherits the structure of a complete orthomodular lattice from $\mathbf Q(X\otimes Y^*,I)$. It remains to be shown that $\perp$ is the associated orthogonality relation of the orthocomplementation $\neg$ on $\mathbf Q(X,Y)$, but this follows directly from the result that $r\mapsto \coname{r}$ is an ortho-isomorphism with respect to $\neg$, and our previous calculation that $r\perp s$ if and only if $\coname{r}\circ\coname{s}^\dag=\perp_I$, which is equivalent to $\coname{r}\perp\coname{s}$. 

Now assume that $\mathbf Q$ does not have a zero object. Then $(\mathbf Q_0,\otimes,I)$ is also a dagger compact quantaloid, and by Proposition \ref{prop:dagger kernel quantaloid}, it is a dagger kernel quantaloid and every object has precisely one $\bot$-monic effect. Hence, the statement applies to $\mathbf Q_0$. As a consequence, for any two object $X$ and $Y$ of $\mathbf Q$, we have $\mathbf Q(X,Y)=\mathbf Q_0(X,Y)$ is orthomodular, and the map $\mathbf Q(X,Y)\to\mathbf Q(X\otimes Y^*,I)$, $r\mapsto\coname r$ is an ortho-isomorphism. Since we showed that the homsets of $\mathbf Q$ are orthomodular, it follows from Lemma \ref{lem:orthomodularity+unique-perp-monic-effects-implies-affine-and-boolean-scalars} that $\mathbf Q$ is affine.
\end{proof}

\begin{corollary}\label{cor:orthomodular quantaloid}
    Let $(\mathbf R,\otimes,I)$ be a dagger compact category with dagger kernels, all small dagger biproducts, with precisely two scalars, and such that every object has precisely one zero-monic effect. Then $\mathbf R$ is a binary dagger compact quantaloid such that every homset $\mathbf R(X,Y)$ is a complete orthomodular lattice with orthogonality relation $r\perp s$ if and only if $\Tr(r\circ s^\dag)=0_I$ and orthocomplementation $\neg r=\bigvee\{s\in\mathbf R(X,Y):r\perp s\}$.
\end{corollary}
\begin{proof}
    Combine Theorems \ref{thm:dagger biproducts imply dagger compact quantaloid} and \ref{thm:dagger kernels imply orthomodularity}.
\end{proof}

\section{Internal maps}\label{sec:internal maps}

From this section on, we will focus on internalizing structures in dagger quantaloids. We will regard morphisms in dagger quantaloids as generalizations of relations, and will also refer to them as `relations'. Then we can generalize properties of ordinary endorelations as follows:

\begin{definition}\label{def:endorelations}
Let $X$ be an object of a dagger quantaloid $\mathbf Q$ and let $r:X\to X$ be an endorelation on $X$. Then we call $r$:
\begin{itemize}
    \item \emph{reflexive} if $\id_X\leq r$;
    \item \emph{transitive} if $r\circ r\leq r$;
    \item \emph{idempotent} if $r\circ r=r$;
    \item \emph{symmetric} if $r^\dag=r$;
    \item \emph{anti-symmetric} if $r\wedge r^\dag\leq\id_X$;
    \item a \emph{preorder} if $r$ is reflexive and transitive;
    \item an \emph{order} if $r$ is an antisymmetric preorder;
    \item a \emph{partial equivalence relation (PER)} if $r$ is symmetric and transitive;
    \item an \emph{equivalence relation} if $r$ is a reflexive PER, or equivalently, if $r$ is a symmetric preorder;
    \item a \emph{projection} if it is a symmetric idempotent.
\end{itemize}
If, in addition, $\mathbf Q$ can be equipped with a dagger-compact monoidal structure, we say that $r$ is:
\begin{itemize}
        \item \emph{irreflexive} if $\Tr(r)=\perp_I$.
\end{itemize}
\end{definition}

\subsection{Definition and properties of internal maps}
We proceed with introducing internal maps in dagger quantaloids, whose definition is similar to the definition of an internal map in an allegory.
\begin{definition}
    Let $\mathbf Q$ be a dagger quantaloid. We call a morphism $f:X\to Y$ in $\mathbf Q$ a \emph{map} if $f^\dag\circ f\geq \id_X$ and $f\circ f^\dag\leq \id_Y$.
\end{definition}

\begin{lemma}
 Let $X,Y,Z$ be objects of a dagger quantaloid $\mathbf Q$. Then:
    \begin{itemize}
        \item[(1)] for any two maps $f:X\to Y$ and $g:Y\to Z$ in $\mathbf Q$, also $g\circ f:X\to Z$ is a map;
        \item[(2)] $\id_X$ is a map.
    \end{itemize}
\end{lemma}
\begin{proof}
    Let $f:X\to Y$ and $g:Y\to Z$ be maps, so $f^\dag\circ f\geq \id_X$ and $f\circ f^\dag\leq \id_Y$, and $g^\dag\circ g\geq \id_Y$ and $g\circ g^\dag\leq\id_Z$. Then $(g\circ f)^\dag\circ(g\circ f)=f^\dag\circ g^\dag\circ g\circ f\geq f^\dag\circ\id_Y\circ f=f^\dag\circ f\geq\id_X$, and $(g\circ f)\circ (g\circ f)^\dag=g\circ f\circ f^\dag\circ g^\dag\leq g\circ\id_Y\circ g^\dag=g\circ g^\dag\leq\id_Z$, so $g\circ f$ is indeed a map. Finally, we have $\id_X^\dag=\id_X$, hence $\id_X^\dag\circ\id_X=\id_X=\id_X\circ\id_X^\dag$, showing that $\id_X$ is a map.
\end{proof}
The previous lemma assures that the following category is well defined.
\begin{definition}
    Let $\mathbf Q$ be a dagger quantaloid. Then by $\mathrm{Maps}(\mathbf Q)$ we denote the wide subcategory of $\mathbf Q$ of maps. 
\end{definition}

Note that $\mathrm{Maps}(\mathbf{Rel})=\mathbf{Set}$. This lead to:
\begin{definition}We define $\mathbf{qSet}:=\mathrm{Maps}(\mathbf{qRel})$, the category of quantum sets and functions.
\end{definition}
It was shown in \cite{Kornell18} that $\qSet$ is dually equivalent to the category of hereditarily atomic von Neumann algebras and normal unital $*$-homomorphisms. We recall that a hereditarily atomic von Neumann algebra is an operator algebra isomorphic to an $\ell^\infty$-sum of matrix algebras. Because of this dual equivalence, $\Set$ embeds dually into the category of hereditarily atomic von Neumann algebras and normal unital $*$-homomorphisms, which is the reason why working with quantum sets might feel more natural than working with hereditarily atomic von Neumann algebras.

\begin{definition}
A map $f:X\to Y$ in a dagger quantaloid $\mathbf Q$ is called
\begin{itemize}
        \item \emph{injective} if $f^\dag\circ f=\id_X$;
        \item \emph{surjective} if $f\circ f^\dag=\id_Y$;
        \item \emph{bijective} if it is both injective and surjective.
    \end{itemize}
\end{definition}
We note that given the dagger biproduct $X$ of a set-indexed family $(X_\alpha)_{\alpha\in A}$ of objects in a dagger quantaloid with small biproducts, the canonical injection $i_\alpha:X_\alpha\to X$ is indeed an injection in the above sense, since $i_\alpha^\dag\circ i_\alpha=p_\alpha\circ i_\alpha=\id_{X_\alpha}$, whereas $i_\alpha\circ i_\alpha^\dag\leq\bigvee_{\beta\in A}i_\beta\circ i_\beta^\dag=\bigvee_{\beta\in A}i_\beta\circ p_\beta=\id_X$ (cf. Proposition \ref{prop:biproducts in quantaloids}). 

\begin{lemma}\label{lem:bijective function vs dagger iso}
Let $f:X\to Y$ be a morphism in a dagger quantaloid $\mathbf Q$. Then the following are equivalent:
\begin{itemize}
    \item[(a)] $f$ is a bijective map;
    \item[(b)] $f$ is a dagger isomorphism in $\mathbf Q$;
    \item[(c)] $f$ is an isomorphism in $\mathrm{Maps}(\mathbf Q)$.
\end{itemize}
\end{lemma}
\begin{proof}
    The equivalence between (a) and (b) is trivial. Let $f$ be a bijective map, so $f^\dag\circ f=\id_X$ and $f\circ f^\dag=\id_Y$. It follows immediately that $f^\dag$ is also a map that is the inverse of $f$, hence $f$ is an isomorphism in $\mathrm{Maps}(\mathbf Q)$. Conversely, assume that $f$ is an isomorphism in $\mathrm{Maps}(\mathbf Q)$, so there is a map $g:Y\to X$ such that $g\circ f=\id_X$ and $f\circ g=\id_Y$. 
    Using that $g$ is a map, it follows that
$g^\dag=g^\dag\circ g\circ f\geq f$ and $g^\dag=f\circ g\circ g^\dag\leq f$. Thus $g^\dag=f$, hence $f^\dag=g$. It follows that $f^\dag\circ f=\id_X$ and $f\circ f^\dag=\id_Y$, so $f$ is a bijective map.
\end{proof}

\begin{lemma}\label{lem:order between maps is trivial}
    Let $\mathbf Q$ be a dagger quantaloid, and let $f,g:X\to Y$ be parallel maps in $\mathbf Q$. If $f\leq g$ in $\mathbf Q$, then $f=g$.
\end{lemma}
\begin{proof}
Since $f\leq g$, we have $f^\dag\leq g^\dag$.
    We have $g=g\circ \id_X\leq g\circ f^\dag\circ f\leq g\circ g^\dag\circ f\leq\id_Y\circ f=f$, which yields equality between $f$ and $g$.
\end{proof}

\begin{lemma}\label{lem:Maps_functor}
    Any homomorphism of dagger quantaloids $F:\mathbf Q\to \mathbf R$ restricts and corestricts to a functor $\mathrm{Maps}(F):\mathrm{Maps}(\mathbf Q)\to\mathrm{Maps}(\mathbf R)$, which is (fully) faithful if $F$ is (fully) faithful.
\end{lemma}
\begin{proof}
Let $X$ and $Y$ be objects of $\mathbf Q$. By definition of a homomorphism of dagger quantaloids, the map $F_{X,Y}:\mathbf{Q}(X,Y)\to\mathbf R(FX,FY)$, $f\mapsto Ff$ preserves the dagger and suprema, hence it is in any case a monotone map, from which it is straightforward that $F_{X,Y}$ restricts and corestricts to a map $\mathrm{Maps}(F)_{X,Y}:\mathrm{Maps}(\mathbf Q)(X,Y)\to \mathrm{Maps}(\mathbf R)(FX,FY)$. It is also evident that $\mathrm{Maps}(F)$ is functorial, for $F$ is a functor. Clearly, $\mathrm{Maps}(F)$ is faithful if $F$ is faithful. Assume, in addition, that $F$ is full. By Lemma \ref{lem:fully faithful homomorphism of quantaloids}, it follows that $F_{X,Y}$ is an order isomorphism. Let $g\in\mathrm{Maps}(\mathbf R)(FX,FY)$, so $g^\dag\circ g\geq \id_{FX}$ and $g\circ g^\dag\leq \id_{FY}$. Then $g\in \mathbf R(FX,FY)$, so there is some $f\in\mathbf Q(X,Y)$ such that $g=Ff$. Since $F_{X,Y}$ preserves daggers, we have $F(f^\dag\circ f)=g^\dag\circ g\geq\id_{FX}=F\id_{X}$, and $F(f\circ f^\dag)=g\circ g^\dag\leq \id_{FY}=F\id_Y$, and since $F_{X,Y}$ is an order isomorphism, it follows that $f^\dag\circ f\geq\id_X$ and $f\circ f^\dag\leq\id_Y$, i.e., $f\in\mathrm{Maps}(\mathbf Q)(X,Y)$. Hence, also $\mathrm{Maps}(F)$ is full. 
\end{proof}

\begin{lemma}\label{lem:Maps form monoidal subcategory}
    Let $(\mathbf Q,\otimes,I)$ be a dagger symmetric monoidal quantaloid. Then $\mathrm{Maps}(\mathbf Q)$ is a symmetric monoidal subcategory of $\mathbf Q$.
\end{lemma}
\begin{proof}
In order to show that $\mathrm{Maps}(\mathbf Q)$ is a symmetric monoidal subcategory of $\mathbf Q$, we only have to verify that the associator, unitors and symmetry are maps, but this follows immediately because in the definition of a dagger symmetric monoidal category, these morphisms are required to be dagger isomorphisms.
\end{proof}

\begin{proposition}\label{prop:Maps_monoidal}
    Let $F:(\mathbf Q,\odot,J)\to (\mathbf R,\otimes,I)$ be a homomorphism of dagger symmetric monoidal quantaloids. Then $ \mathrm{Maps}(F):(\mathrm{Maps}(\mathbf Q),\odot,J)\to(\mathrm{Maps}(\mathbf R),\otimes,I)$ is a strong symmetric monoidal functor with the same coherence morphisms as $F$.   
\end{proposition}
\begin{proof}
Let $\psi:J\to FI$ and $\psi_{X,Y}:FX\otimes FY\to F(X\odot Y)$ be the coherence morphisms for $X,Y\in\mathbf Q$, which are dagger isomorphisms, since any homomorphism of dagger symmetric monoidal quantaloids in particular is a dagger strong symmetric monoidal functor. Since $\mathbf Q$ and $\mathrm{Maps}(\mathbf Q)$ share the same objects, the statement follows if we can prove that the coherence maps are bijective functions. But this follows from Lemma  \ref{lem:bijective function vs dagger iso}.
\end{proof}

\begin{proposition}\label{prop:coproducts of Maps}
    Let $(\mathbf R,\otimes,I)$ be a dagger symmetric monoidal quantaloid with all small dagger biproducts. Then the embedding $\mathrm{Maps}(\mathbf R)\to\mathbf R$ creates all coproducts, i.e., if $(X_\alpha)_{\alpha\in A}$ is a collection of objects in $\mathrm {Maps}(\mathbf R)$, then their dagger biproduct in $\mathbf R$ is their coproduct in $\mathrm{Maps}(\mathbf R)$, and the canonical injections in $\mathbf R$ are maps.
\end{proposition}
\begin{proof}
Let $X$ be the dagger biproduct in $\mathbf R$ of a collection $(X_\alpha)_{\alpha\in A}$ of objects in $\mathrm{Maps}(\mathbf R)$ with canonical injections $i_\alpha:X_\alpha\to X$ for each $\alpha\in A$. Fix $\alpha\in A$. Since $X$ is a dagger biproduct, the canonical projection $p_\alpha:X\to X_\alpha$ satisfies $p_\alpha= i_\alpha^\dag$. By definition of a biproduct, we have $i_\alpha^\dag\circ i_\alpha=p_\alpha\circ i_\alpha=\id_{X_\alpha}.$ It follows from Corollary \ref{cor:sum of ip is id} that $i_\alpha\circ i_\alpha^\dag=i_\alpha\circ p_\alpha\leq\bigvee_{\beta\in A}i_\beta\circ p_\beta=\id_X$,  so $i_\alpha$ is indeed a map.

Now, let $Y$ be another object of $\mathrm{Maps}(\mathbf R)$, and for each $\alpha\in A$, let $f_\alpha:X_\alpha\to Y$ be a map. To show that $X$ is the coproduct of $(X_\alpha)_{\alpha\in A}$ in $\mathrm{Maps}(\mathbf R)$, we have to show that  $f:=[f_\alpha]_{\alpha\in A}:X\to Y$ is a map. Using Proposition \ref{prop:dagger nabla and delta}, we find
\[ f\circ f^\dag=\bigvee_{\alpha\in A}f_\alpha\circ f_\alpha^\dag\leq\id_Y,\]
for $f_\alpha\circ f_\alpha^\dag\leq \id_Y$ because $f_\alpha$ is a map. By the same proposition, we obtain $(f^\dag\circ f)_{\alpha,\beta}=f_\beta^\dag\circ f_\alpha$.
Let $\alpha,\beta\in A$. First assume that $\alpha\neq\beta$. By Lemma \ref{lem:matrix identity}, we have $(\id_X)_{\alpha,\beta}=0_{X_\alpha,X_\beta}$ for $\alpha\neq\beta$, so $(f^\dag\circ f)_{\alpha,\beta}\geq (\id_{X})_{\alpha,\beta}$. Now, let $\alpha=\beta$. Since $f_\alpha$ is a map, it follows that $(f^\dag\circ f)_{\alpha,\beta}=f_\alpha^\dag\circ f_\alpha\geq\id_{X_\alpha}=(\id_X)_{\alpha,\beta}$, where the last identity also follows from Lemma \ref{lem:matrix identity}. So $(f^\dag\circ f)_{\alpha,\beta}\geq(\id_X)_{\alpha,\beta}$ for each $\alpha,\beta\in A$. It now follows from Proposition \ref{prop:biproducts in quantaloids are monotone} that $f^\dag\circ f\geq\id_X$, so $f$ is a map.\end{proof}

\subsection{The families construction}
Given a dagger quantaloid $\mathbf Q$, Proposition \ref{prop:coproducts of Maps} begs the question whether $\mathrm{Maps}(\mathrm{Matr}(\mathbf Q))$ is the \emph{free coproduct completion} of $\mathrm{Maps}(\mathbf Q)$. We will see that for $\mathbf 2$, this is indeed the case. However, in general the statement is false.

For an arbitrary category $\mathbf C$, the free coproduct completion is the category $\mathrm{Fam}(\mathbf C)$ of set-indexed families $(X_\alpha)_{\alpha\in A}$. A morphism $(\varphi,f):(X_\alpha)_{\alpha\in A}\to(Y_\beta)_{\beta\in B}$ in $\mathrm{Fam}(\mathbf C)$ consists of a function $\varphi:A\to B$ and for each $\alpha\in A$, a morphism $f_\alpha:X_\alpha\to Y_{\varphi(\alpha)}$.

Now in $\mathbf 2$, there is only one map, which corresponds to the element $1$ in $2=\{0,1\}$. Hence, as a category, we have $\mathrm{Maps}(\mathbf 2)\cong\mathbf 1$, whence $\mathrm{Maps}(\mathrm{Matr}(\mathbf 2))\cong\mathrm{Maps}(\mathbf{Rel})=\mathbf{Set}=\mathrm{Fam}(\mathbf 1)\cong\mathrm{Fam}(\mathrm{Maps}(\mathbf 2))$.

For an arbitrary dagger quantaloid $\mathbf Q$, we  have an inclusion $I:\mathrm{Fam}(\mathrm{Maps}(\mathbf Q))\to\mathrm{Maps}(\mathrm{Matr}(\mathbf Q))$. Indeed, the objects in both categories are clearly the same, so $I$ is the identity on objects. Given a morphism $(\varphi,f):(X_\alpha)_{\alpha\in A}\to(Y_\beta)_{\beta\in B}$ in $\mathrm{Fam}(\mathrm{Matr}(\mathbf Q))$, we define $I(f,\varphi):(X_\alpha)_{\alpha\in A}\to (Y_\beta)_{\beta\in B}$ in $\mathrm{Maps}(\mathrm{Matr}(\mathbf Q))$ by
\[ I(\varphi,f)_\alpha^\beta=\begin{cases}
f_{\alpha}, & \beta=\varphi(\alpha),\\
\perp_{X_\alpha,Y_\beta}, & \text{otherwise}.
\end{cases}
\]
Using that each $f_\alpha$ is a map, so a morphism of $\mathrm{Maps}(\mathbf Q)$, it
is straightforward to see that $I(\varphi,f)$ is a map, so a morphism of
$\mathrm{Maps}(\mathrm{Matr}(\mathbf Q))$. However, in general, $I$ is not full, so not an equivalence. For instance, let $\mathbf Q=\mathbf{FdOS}$. Let $H$ be a two-dimensional Hilbert space, and let $p,q:H\to H$
be two orthogonal projections that span the identity $1_H$ of $H$ in $\mathbf{FdHilb}$. For instance, if $H=\CC^2$, let $p=\begin{pmatrix}
    1 & 0\\
    0 & 0
\end{pmatrix}$ and $q=\begin{pmatrix}
    0 & 0\\
    0 & 1
\end{pmatrix}$. Let $P,Q\in\mathbf{FdOS}(H,H)$ be the morphisms $P=\CC p$ and $Q=\CC q$. Then, in $\mathbf{FdOS}$, we have $p,q\in P\vee Q$, hence $\id_H=\CC 1_H=\CC(p+q)H\subseteq P\vee Q$. Let $\H,\X$ in $\qRel=\mathrm{Matr}(\mathbf{FdOS})$ be given by $\H=(H)$, and $\X=(X_\alpha)_{\alpha\in 2}$, where $X_0=X_1=\CC$. Let $F:\H\to\X$ be given by $F_*^0=\top_{H,\CC}\cdot P$ and $F_*^1=\top_{H,\CC}\cdot Q$.
Then both $F_*^0$ and $F_*^1$ are  not maps in $\mathbf{FdOS}$, but $F$ is a map in $\mathbf{qRel}$. To see that $F$ is a map in $\mathrm{Matr}(\mathbf{FdOS})$, it is easiest to apply Theorem
\ref{thm:functions with classical codomain} below, which we are allowed to use by Theorem \ref{thm:qRel-properties} below. We have $P\cdot Q^\dag=P\cdot Q=0$ because $pq=0$, hence $\mathrm{Tr}(F_*^0\cdot (F_*^1)^\dag)=\mathrm{Tr}(\top_{H,\CC}\cdot P\cdot Q^\dag\cdot \top_{H,\CC}^\dag)=\mathrm{Tr}(0)=0=\perp_\CC$, so $F_*^0\perp F_*^1$. Furthermore, we have $F_*^0\vee F_*^1=\top_{H,\CC}\cdot (P\vee Q)\geq \top_{H,\CC}\cdot\id_H=\top_{H,\CC}$. We have $\top_{H,\CC}=B(H,\CC)=H^*$, the Banach space dual of $H$, which is also a Hilbert space. Hence, $F_*^0=\top_{H,\CC}\cdot P=H^*p=p^*H^*$, where $p^*:H^*\to H^*$, $\varphi\mapsto \varphi(-)p$ is the Banach space dual of $p$. Also $p^*$ is a projection, but on the Hilbert space $H^*$. Since the map $r\mapsto rK$ is a bijection between projections $r$ on a finite-dimensional Hilbert space $K$ and subspaces $rK$ of $K$, it follows that $F_*^0$ is a proper subspace of $\top_{H,\CC}$. In a similar way, we find that $F_*^1<\top_{H,\CC}$ in $\mathbf{FdOS}$. Since $\top_{H,\CC}$ is the only map $H\to\CC$ in $\mathbf{FdOS}$ by  Proposition \ref{prop:PER} below, it follows that $F_*^0$ and $F_*^1$ are not maps in $\mathbf{FdOS}$. Hence, $F$ is not of the form $I(\varphi,f)$ for some morphism $(\varphi,f):\H\to\X$ in $\mathrm{Fam}(\mathbf{FdOS})$. 

\subsection{$\bot$-monic partial equivalence relations}
We recall that a morphism $p$ in a dagger quantaloid $\mathbf R$ is called a \emph{partial equivalence relation (PER)} if it is symmetric, i.e., $p^\dag=p$, and transitive, i.e., $p\circ p\leq p$. An important property of a dagger symmetric monoidal quantaloid $\mathbf R$ will be whether $\bot$-monic PERs are equivalence relations, as this is one of the conditions that assures that the internal maps of $\mathbf R$ form a \emph{semicartesian} category, i.e., a symmetric monoidal category with terminal monoidal unit. We will see that both in $\mathbf{Rel}$ and $\mathbf{qRel}$ $\bot$-monic PERs are equivalence relations.

\begin{proposition}\label{prop:PER}
    Let $(\mathbf Q,\otimes,I)$ be an affine dagger symmetric monoidal quantaloid with dagger kernels. Assume that for each object $X$ of $\mathbf Q$: 
    \begin{itemize}
           \item[(1)] there is a $\bot$-monic effect $e:X\to I$;
                \item[(2)] any $\bot$-monic PER on $X$ is an equivalence relation on $X$.
    \end{itemize}
    Then:
    \begin{itemize}
        \item[(a)] $I$ is terminal in $\mathrm{Maps}(\mathbf Q)$, where for each object $X\in\mathrm{Maps}(\mathbf Q)$ the unique map $X\to I$ is $\top_{X,I}$;
        \item[(b)] for each object $X$ of $\mathbf Q$, the morphism $\top_{X,I}:X\to I$ is the unique effect that is $\bot$-monic.
    \end{itemize}
\end{proposition}
\begin{proof}
Let $e:X\to I$ be a $\bot$-monic effect. Then $e\circ e^\dag$ is a scalar, and since $\mathbf Q$ is affine, it follows that $e\circ e^\dag\leq\id_I$.  Consider $p=e^\dag\circ e$. Then $p^\dag=p$ and $p\circ p=e^\dag\circ e\circ e^\dag\circ e\leq e^\dag\circ\id_I\circ e=p$, so $p$ is a PER. It follows from from Lemma \ref{lem:r-zero-mono-iff-rdagr-zero-mono} that $p$ is also a $\bot$-mono in $\mathbf Q$. Hence, by assumption, we have that $p$ is an equivalence relation, so $p\geq\id_X$. It follows that $e$ is a map. Let $f:X\to I$ be another map. Since $\mathbf Q$ is affine, we have $e\circ f^\dag\leq\id_I$, hence $e\leq e\circ f^\dag\circ f\leq\id_I\circ f=f$, hence it follows from Lemma \ref{lem:order between maps is trivial} that $e=f$.  By Lemma \ref{lem:zeromonoeffect implies topzeromono} it follows that $\top_{X,I}$ is a $\bot$-monic effect, which is therefore a map, and any other map $X\to I$ must be equal to $\top_{X,I}$, proving that $I$ is terminal in $\mathrm{Maps}(\mathbf Q)$. For (b), if $e:X\to I$ is another $\bot$-monic effect, it follows that $e$ is a map which necessarily equals $\top_{X,I}$, so $\top_{X,I}:X\to I$ is the unique effect on $X$ that is $\bot$-monic.
\end{proof}

\begin{lemma}\label{lem:2}
The only nontrivial commutative quantale $(V,\cdot,e)$ whose associated quantaloid $\mathbf V$ (cf. Example \ref{ex:quantale-induced quantaloid}) is affine, has dagger kernels, and that has the property that every $\bot$-monic PER is an equivalence relation is the two-element Boolean algebra $2$.
\end{lemma}
\begin{proof}We first note that because $V$ is affine, it immediately follows that there is a $\bot$-monic effect $1\to 1$ in $\mathbf V$, namely the identity $e$. In order to see that $V$ must be $2$, we first note that the requirement that $\mathbf V$ is a dagger kernel quantaloid forces $V$ to  satisfy $v\cdot w\neq\perp$ for any $v,w\neq\perp$ in $V$ (cf. Lemma \ref{lem:dagger-kernel-quantale}), so any $v\neq\perp$ in $V$ is a $\bot$-mono in $\mathbf V$. Moreover, $V$ is affine, hence for each $v\in V$, we have $v\cdot v\leq v\cdot \top=v\cdot e=v$. Since we always have $v^\dag=v$, it follows that every $v\in V$ is a PER, hence for every $v\neq\perp$, we must have $\top=e\leq v$, which forces $v=\top$ for all $v\neq\perp$. 
\end{proof}

Also $\bot$-monic PERs in $\mathbf{FdOS}$ are equivalence relations. We first need a lemma. We recall that in the finite-dimensional setting, a C*-algebra is a subalgebra of the algebra of all linear operators $B(H)$ on a finite-dimensional Hilbert space that is \emph{selfadjoint}, i.e.,  closed under the dagger on $\mathbf{FdHilb}$. 

\begin{lemma}\label{lem:nondegenerate acting subalgebra is unital}
    Let $H$ be a finite-dimensional Hilbert space and let $A\subseteq B(H)$ be a C*-subalgebra that acts on $H$ in a nondegenerate way, i.e., for each nonzero $x\in H$, there is some $a\in A$ such that $ax\neq 0$. Then $1_H\in A$.
\end{lemma}
\begin{proof}
    Since $A$ is a C*-subalgebra of a finite-dimensional C*-algebra, it must be finite-dimensional itself, hence it should contain a unit element $e$ \cite[Lemma 11.1]{takesaki:oa1}. We will show that $e=1_H$.

    Since $A$ acts in a nondegenerate way on $H$, we have that the span of $\{ax:a\in A,x\in H\}$ equals $H$ \cite[Proposition 9.2]{takesaki:oa1}. Hence, for each $x\in H$, there are $a_1,\ldots,a_n\in A$ and $x_1,\ldots,x_n\in H$ such that $x=a_1x_1+\ldots+a_nx_n$. Then $ex=e(a_1x_1+\ldots+a_nx_n)=(ea_1)x_1+\ldots+(ea_n)x_n=a_1x_1+\ldots+a_nx_n=x$. Thus $e$ is indeed the identity $1_H$ on $H$.
\end{proof}

\begin{proposition}\label{prop:PERs in FdOS}
    Let $H$ be a nonzero finite-dimensional Hilbert space, and let $P:H\to H$ be a PER in $\mathbf{FdOS}$. If $P$ is a $\bot$-mono, then it is an equivalence relation in $\mathbf{FdOS}$. 
\end{proposition}
\begin{proof}
Since $P$ is a PER, we have $P\cdot P\subseteq P$ and $P^\dag=P$. From the latter identity it follows that $P$ is a self-adjoint subspace of $B(H)$. Furthermore, for each $a,b\in P$, we have $ab\in P\cdot P\subseteq P$, so $P$ is a self-adjoint subalgebra of $B(H)$. Since $B(H)$ is finite-dimensional, it follows that $P$ is a C*-subalgebra of $B(H)$. Let $x\in H$ be nonzero, and let $\check x:\CC\to H$ be the map $\lambda\mapsto\lambda x$. Define the morphism $R\in\mathbf{FdOS}(\CC,H)$ by $R=\CC\check x$. Then $R\neq\perp_{\CC,H}$, and since $P$ is a $\bot$-mono, we must have $P\cdot R\neq\perp_{\CC,H}$, so there is some $a\in P$ such that $a\check x\neq 0$. This means that $\lambda ax=a(\lambda x)=(a\check x)(\lambda)\neq 0$ for some $\lambda\in\CC$, which is only possible if $ax\neq 0$. Thus, $P$ is a C*-subalgebra of $B(H)$ that acts in a nondegenerate way on $H$, so  Lemma \ref{lem:nondegenerate acting subalgebra is unital} implies that $1_H\in P$. As a consequence $\id_H=\CC 1_X\subseteq H$, showing that $P$ is reflexive, hence an equivalence relation on $H$ in $\mathbf{FdOS}$.
    \end{proof}

\begin{proposition}\label{prop:PER-Matr}
    Let $\mathbf R$ be a dagger quantaloid such that every $\bot$-monic PER in $\mathbf R$ is an equivalence relation. Then every $\bot$-monic PER in $\mathrm{Matr}(\mathbf R)$ is an equivalence relation. 
\end{proposition}
\begin{proof}
    Let $X=(X_\alpha)_{\alpha\in A}$ be an object in $\mathrm{Matr}(\mathbf R)$ and let $p=(p_\alpha^\beta)_{(\alpha,\beta)\in A\times A}:X\to X$ be a $\bot$-monic PER on $X$, $p^\dag=p$ and $p\circ p\leq p$.
    Fix $\alpha\in A$. Let $r:Y\to X_\alpha$ be a morphism in $\mathbf R$ such that $p_\alpha^\alpha\circ r= \perp_{Y,X_\alpha}$. Let $f:(Y)\to X$ be the morphism in $\mathrm{Matr}(\mathbf R)$ given by \[f_*^\beta = \begin{cases} r, & \beta=\alpha
    \\
    \perp_{Y,X_\beta}, & \beta\neq\alpha
    \end{cases}\]
    for each $\beta\in A$.
    Then $p\circ f=\perp_{(Y),X}$, and since $p$ is $\bot$-monic, it follows that $f=\perp_{(Y),X}$, which forces $r=f_*^\alpha=(\perp_{(Y),X})_*^\alpha=\perp_{Y,X_\alpha}$. So $p_\alpha^\alpha$ is a $\bot$-mono.
    
Moreover, since $p$ is a PER, it follows that $(p_\alpha^\alpha)^\dag=(p^\dag)_\alpha^\alpha=p_\alpha^\alpha$ and $p_\alpha^\alpha\circ p_\alpha^\alpha\leq \bigvee_{\gamma\in A}p_\gamma^\alpha\circ p_\alpha^\gamma=(p\circ p)_\alpha^\alpha\leq p_\alpha^\alpha$, so $p_\alpha^\alpha$ is a $\bot$-monic PER on the object $X_\alpha$ in $\mathbf R$. It follows that $p_\alpha^\alpha$ is an equivalence relation, hence $\id_{X_\alpha}\leq p_\alpha^\alpha$. As a consequence, $(\id_X)_\alpha^\alpha=\id_{X_\alpha}\leq p_\alpha^\alpha$. Since for each distinct $\alpha,\beta\in A$ we have $(\id_X)_\alpha^\beta=\perp_{X_\alpha,X_\beta}$, it follows that $\id_X\leq p$, so $p$ is an equivalence relation.    
\end{proof}

The following theorems on $\mathbf{Rel}$ and $\mathbf{qRel}$ summarize all their relevant properties in this article. We note that stronger properties for $\mathbf{Rel}$ that actually characterize it as a dagger category are presented in \cite{Kornell23}.

\begin{theorem}\label{thm:Rel-properties}
    $\mathbf{Rel}$ is a binary dagger compact quantaloid with dagger kernels, and small dagger biproducts. Moreover, for each set $X$, there is a unique $\bot$-monic effect, namely $\top_{X,1}$, and any $\bot$-monic PER on $X$ is an equivalence relation. 
\end{theorem}
\begin{proof}
By Example \ref{ex:V-Rel is dagger compact quantaloid}, $\mathbf{Rel}$ is a dagger compact quantaloid that is equivalent to $\mathrm{Matr}(\mathbf 2)$ as a dagger symmetric monoidal quantaloid. Clearly, $\mathbf{Rel}$ is binary. Combining   Lemma \ref{lem:2} with Corollary \ref{cor:QdaggerkernelsimplyMatrQdaggerkernels} and Proposition \ref{prop:PER-Matr}, it follows that $\mathbf{Rel}$ is a dagger kernel quantaloid,  $\top_{X,1}$ is the unique $\bot$-monic effect $X\to 1$ for each set $X$, and that every $\bot$-monic PER in $\mathbf{Rel}$ is an equivalence relation.
\end{proof}

\begin{theorem}\label{thm:qRel-properties}
    $\mathbf{qRel}$ is a binary dagger compact quantaloid with dagger kernels and small dagger biproducts such that for each quantum set $\X$ the effect $\top_{\X,\mathbf 1}$ is the only $\bot$-monic effect, and such that any $\bot$-monic PER on $\X$ is an equivalence relation.
\end{theorem}
\begin{proof}
By definition of $\mathbf{qRel}=\mathrm{Matr}(\mathbf{FdOS})$, it follows that $\mathbf{qRel}$ has small dagger biproducts. By Proposition \ref{prop:qRel is a dagger kernel category}, $\mathbf{qRel}$ is a binary dagger compact quantaloid with dagger kernels, such that each quantum set $\X$ has a unique $\bot$-monic effects, namely $\top_{\X,\mathbf 1}$. Since every $\bot$-monic PER in $\mathbf{FdOS}$ is an equivalence relation (cf Proposition \ref{prop:PERs in FdOS}), it follows from  Proposition \ref{prop:PER-Matr} that every $\bot$-monic PER in $\mathbf{qRel}$ is an equivalence relation.
\end{proof}

\section{Internal preorders}\label{sec:internal preorders}
\subsection{Preordered objects}

In this section, we investigate internal preorders in dagger quantaloids (cf. Definition \ref{def:endorelations}). 

\begin{lemma}\label{lem:opposite relation}
    Let $r:X\to X$ be an endorelation on an object $X$ of a dagger quantaloid $\mathbf R$. Then:
    \begin{itemize}
        \item[(a)] $r^\dag$ is reflexive if $r$ is reflexive;
        \item[(b)] $r^\dag$ is transitive if $r$ is transitive;
        \item[(c)] $r^\dag$ is symmetric if $r$ is symmetric;
        \item[(d)] $r^\dag$ is anti-symmetric if $r$ is anti-symmetric;
        \item[(e)] $r^\dag$ is irreflexive if $r$ is irreflexive (under the additional assumptions that $\mathbf R$ is a dagger compact quantaloid).
    \end{itemize}
\end{lemma}
\begin{proof}
Statements (a)-(d) follow because $(-)^\dag$ is a functor whose action on homsets is an involutive order isomorphism. Statement (e) follows from \cite[Lemma 3.63(f)]{heunenvicary}.
\end{proof}

\begin{lemma}\label{lem:dual relation}
    Let $r:X\to X$ be an endorelation on an object $X$ of a dagger compact quantaloid $\mathbf R$. Then $r^*$ on $X^*$ satisfies the following properties:
    \begin{itemize}
        \item[(a)] $r^*$ is reflexive if $r$ is reflexive;
        \item[(b)] $r^*$ is transitive if $r$ is transitive;
        \item[(c)] $r^*$ is symmetric if $r$ is symmetric;
        \item[(d)] $r^*$ is anti-symmetric if $r$ is anti-symmetric;
        \item[(e)] $r^*$ is irreflexive if $r$ is irreflexive (under the additional assumptions that $\mathbf R$).
    \end{itemize}
\end{lemma}
\begin{proof}
Statements (a), (b), and (d) hold because $(-)^*$ is a functor whose action on homsets is an order isomorphism (cf. Lemma \ref{lem:counit in compact quantaloid is order iso}). Statement (c) follows from Lemma \ref{lem:dagger_and_dual} and the symmetry of $r$. Statement (e) follows from$\Tr_{X^*}(r^*)=\Tr_X(r)$ by \cite[Exercise 3.12(c)]{heunenvicary}.
\end{proof}

\begin{example}\label{ex:trivial order}
    Let $X$ be an object in a dagger quantaloid $\mathbf R$. Then the identity morphism $\id_X$ on $X$ is reflexive, transitive, symmetric, and anti-symmetric. We call $\id_X$ the \emph{trivial} or \emph{flat} order on  $X$.
\end{example}

\begin{definition}
An \emph{preorder} on an object $X$ of a dagger quantaloid $\mathbf R$ is a reflexive and transitive endomorphism $\qo:X\to X$. We call the pair $(X,\qo)$ a \emph{preordered} object.
If, in addition, $\qo$ is anti-symmetric, we call $\qo$ a \emph{partial order} and $(X,\qo)$ a \emph{partially ordered object}, or with a slight abuse of terminology a \emph{poset}. Sometimes, we will say that $X$ is a preordered object or poset without mentioning the (pre)order $\qo$ explicitly. 
\end{definition}
We will often formulate inequalities between morphisms in a dagger quantaloid $\mathbf R$ involving preorders $\qo$ on objects $X$ of $\mathbf R$. In order to increase the readability of those expressions, we will sometimes write $(\qo)$ instead of $\qo$.

Given a preorder $\qo$ on an object $X$ in a dagger quantaloid $\mathbf R$, it follows from Lemma \ref{lem:opposite relation} that the dagger $\qo^\dag$ of $\qo$ is again a preorder. Similarly, if $(\mathbf R,\otimes,I)$ is a dagger compact quantaloid, it follows from Lemma \ref{lem:dual relation} that the dual $\qo^*$ of $\qo$ is a preorder (on $X^*$). In both cases, the resulting preorders are even orders when $\qo$ is an order. This leads to the following definition:  
\begin{definition}
    Let $(X,\qo)$ be a preordered object in a dagger quantaloid $\mathbf R$. 
    \begin{itemize}
        \item We call the preorder $\qop:=\qo^\dag$ the \emph{opposite} preorder, and the pair $(X,\qop)$ the \emph{opposite} preordered objects, also denoted by $(X,\qo)^\op$, or simply $X^\op$ if it is clear that $X$ is preordered by $\qo$.
        \item If $(\mathbf R,\otimes,I)$ is a dagger compact quantaloid, we call the preorder $\qo^*$ the \emph{dual} preorder, and the pair $(X^*,\qo^*)$ the preordered object \emph{dual} to $(X,\qo)$, also denoted by $(X,\qo)^*$, or simply $X^*$ if it is clear that $X$ is preordered by $\qo$.
    \end{itemize}
    If $\qo$ is an order, we call $X^\op$ and $X^*$ the \emph{opposite} poset and the \emph{dual} poset of $X$, respectively. 
\end{definition}
For $\mathbf R=\mathbf{Rel}$, the opposite preorder on an object coincides with the dual preorder. However, for $\mathbf{R}=\mathbf{qRel}$ both concepts differ, since objects are not naturally isomorphic to their dual in this category, let alone equal as in the case of $\mathbf{Rel}$. 

We further note that by Lemma \ref{lem:dagger_and_dual}, we have $(X^\op)^*=(X^*)^\op$ for each preordered object $X$.

For the next definition, recall that a \emph{map} from an object $X$ to an object $Y$ in a dagger quantaloid $\mathbf R$ is a morphism $f:X\to Y$ such that $f^\dag\circ f\geq\id_X$ and $f\circ f^\dag\leq \id_Y$.
\begin{definition}
    Let $(X,\qo_X)$ and $(Y,\qo_Y)$ be preordered objects of a dagger quantaloid $\mathbf R$. Then a map $f:X\to Y$ is called:
    \begin{itemize}
        \item \emph{monotone} if it satisfies satisfies one of the following equivalent conditions (hence all):
  \begin{itemize}
    \item[(1)] $f\circ (\qo_X)\leq (\qo_Y)\circ f$;
    \item[(2)] $f\circ \qo_X\circ f^\dag\leq (\qo_Y)$;
    \item[(3)] $(\qo_X)\leq f^\dag\circ \qo_Y\circ f$.
\end{itemize}
\item \emph{an order embedding} if $\qo_X=f^\dag\circ\qo_Y\circ f$;
\item \emph{an order isomorphism} if it is a monotone map that has an inverse which is also monotone.
    \end{itemize}
    \end{definition}
We verify that the conditions in the definition are indeed equivalent. Assume that (1) holds. We show that (2) holds: 
\[ f\circ \qo_X\circ f^\dag\leq (\qo_Y)\circ f\circ f^\dag\leq (\qo_Y)\circ \id_Y=(\qo_Y).\]
Now assume that (2) holds. We show that (3) holds:
\[ (\qo_X)= \id_\X\circ \qo_X\circ \id_X\leq f^\dag\circ f\circ \qo_X\circ f^\dag\circ f\leq f^\dag\circ \qo_Y\circ f.\]
Finally we show that (3) implies (1), so assume that $(\qo_X)\leq f^\dag\circ\qo_Y\circ f$. Then:
\[f\circ (\qo_X)=f\circ f^\dag\circ\qo_Y\circ f\leq\id_Y\circ \qo_\Y\circ f=(\qo_\Y)\circ f.\]

It follows directly from the definitions that an order embedding is monotone. 

\begin{lemma}\label{lem:functions with trivially ordered domain are monotone}
    Let $f:X\to Y$ be a map in a dagger quantaloid $\mathbf R$, and let $\qo$ be a preorder on $Y$. Then $f:(X,\id_X)\to(Y,\qo)$ is monotone.
\end{lemma}
\begin{proof}
    By a direct calculation: $f\circ\id_X=\id_Y\circ f\leq\qo\circ f$.
\end{proof}

\begin{lemma}
    Let $(X,\qo_X)$, $(Y,\qo_Y)$ and $(Z,\qo_Z)$ be preordered objects in a dagger quantaloid $\mathbf R$ and let $f:X\to Y$ and $g:Y\to Z$ be monotone maps. Then $g\circ f$ is a monotone map. 
\end{lemma}
\begin{proof}By monotonicity of $f$ and $g$, we have $f\circ(\qo_X)\leq (\qo_Y)\circ f$ and $g\circ (\qo_Y)\leq (\qo_Z)\circ g$, hence $g\circ f\circ (\qo_X)\leq g\circ(\qo_Y)\circ f\leq (\qo_Z)\circ g\circ f$. 
\end{proof}

It follows from the previous lemma that the following categories are well defined.
\begin{definition}Let $\mathbf R$ be a dagger quantaloid. Then:
\begin{itemize}
    \item  $\mathrm{PreOrd}(\mathbf R)$ is defined as the category of preordered objects and monotone maps. The identity morphism on an object $(X,\qo)$ of $\mathrm{PreOrd}(\mathbf R)$ is the identity $\id_X$ on $X$.
    \item $\mathrm{Pos}(\mathbf R)$ is defined as the full subcategory of $\mathrm{PreOrd}(\mathbf R)$ of partially ordered objects.
    \end{itemize}
If $\mathbf R=\mathbf{Rel}$, we have $\mathrm{PreOrd}(\mathbf R)=\mathbf{PreOrd}$ and $\mathrm{Pos}(\mathbf R)=\mathbf{Pos}$.
\end{definition}

\begin{lemma}\label{lem:functions and the opposite order}
Let $X$ and $Y$ be preordered objects in a dagger quantaloid $\mathbf R$ and let $f:X\to Y$ be a map. Then $f:X\to Y$ is monotone if and only if $f:X^\op\to Y^\op$ is monotone.
\end{lemma}

\begin{proof}
Let $\qo_X$ and $\qo_Y$ be the preorders on $X$ and $Y$, respectively. Assume that $f:(X,\qo_X)\to(Y,\qo_Y)$ is a monotone map. Then $f\circ (\qo_X)\leq (\qo_Y)\circ f$, hence $(\qop_X)\circ f^\dag=(\qo_X)^\dag\circ f^\dag\leq f^\dag\circ (\qo_Y)^\op=f^\dag\circ(\qop_Y)$. Using the properties of a map, we obtain
\[f\circ (\qop_X)=f\circ \qop_X\circ \id_X\leq  f\circ \qop_X\circ f^\dag\circ f\leq f\circ f^\dag\circ \qop_Y\circ f\leq \id_Y\circ \qop_Y\circ f=(\qop_Y)\circ f,\] 
so $f$ is indeed a monotone map $X^\op\to Y^ \op$. Now assume that $f:X^\op\to Y^\op$ is a monotone map. Then it follows that $f:X^{\op\op}\to Y^{\op\op}$ is a monotone map, and since $X^{\op\op}=X$ and $Y^{\op\op}=Y$, the statement follows.
\end{proof}

\begin{proposition}\label{functor opposite preorder functor}
Let $\mathbf R$ be a dagger quantaloid. Then the assignment $X\mapsto X^\op$ extends to a functor $(-)^\mathrm{op}:\mathrm{PreOrd}(\mathbf R)\to\mathrm{PreOrd}(\mathbf{R})$ that is the identity on morphisms.
\end{proposition}
\begin{proof}Let $X$, $Y$ and $Z$ be preordered objects in $\mathbf R$, and let $f:X\to Y$ and $g:Y\to Z$ be monotone maps. 
Using Lemma \ref{lem:functions and the opposite order}, we have $\id_X^\op=\id_X=\id_{X^\op}$, because the underlying object of $X$ and $X^\op$ it the same. The lemma also yields $(f\circ g)^\op=f\circ g=f^\op\circ g^\op$, so $(-)^\op$ is functorial.
\end{proof}

Next, we provide an alternative description of order isomorphisms.

\begin{lemma}\label{lem:characterization of order isomorphisms}
Let $\mathbf R$ be a dagger quantaloid and let $(X,\qo_X)$ and $(Y,\qo_Y)$ be preordered objects in $\mathbf R$. Then a map $f:X\to Y$ is an order isomorphism if and only if it is a bijection such that $f\circ \qo_X=\qo_Y\circ f$.
\end{lemma}
\begin{proof}
Let $f$ be an order isomorphism. Then it is a monotone map, and there is a monotone
map $g:Y\to X$ such that $g\circ f=\id_X$ and $f\circ g=\id_Y$. Then $g=\id_Y\circ
g\leq f^\dag\circ f\circ g=f^\dag\circ \id_Y=f_Y$, and $g=g\circ \id_Y\geq g\circ
f\circ f^\dag=\id_Y\circ f^\dag=f^\dag$, so $f^\dag=g$. Thus $f^\dag\circ f=I_X$ and
$f\circ f^\dag=\id_Y$ expressing that $f$ is both injective and surjective, hence
bijective. Moreover, since $f$ and $g$ are monotone, we obtain $f\circ (\qo_X)\leq
(\qo_Y)\circ f$ and $g\circ (\qo_Y)\leq (\qo_X)\circ g$. From the latter inequality
we obtain $\qo_Y\circ f=\id_\Y\circ \qo_Y\circ f=f\circ g\circ \qo_Y\circ f\leq
f\circ \qo_X\circ g\circ f=f\circ \qo_X\circ \id_X=f\circ \qo_X$, whence $f\circ
\qo_X=\qo_Y\circ f$.

Conversely, assume that $f$ is a bijection such that $f\circ \qo_X=\qo_Y\circ f$. It follows immediately that $f$ is monotone. Since $f$ is a bijection, we have $f^\dag\circ f=\id_X$ and $f\circ f^\dag=\id_Y$, whence $f^\dag:Y\to X$ is also a map. Then $f^\dag\circ \qo_Y=f^\dag\circ \qo_Y\circ \id_Y=f^\dag\circ \qo_Y\circ f\circ f^\dag=f^\dag\circ f\circ \qo_X\circ f^\dag=\id_X\circ \qo_X\circ f^\dag=\qo_X\circ f^\dag$, so also $f^\dag$ is monotone.
\end{proof}

\begin{lemma}\label{lem:PreOrd_functor}
 Let $F:\mathbf Q\to \mathbf R$ be a homomorphism of dagger quantaloids.  Then for each preordered object $(X,\qo)$ in $\mathbf Q$, the pair $(FX,F(\qo))$ is a preordered object in $\mathbf R$, and the assignment $(X,\qo)\mapsto (FX,F(\qo))$ extends to a functor $\mathrm{PreOrd}(F):\mathrm{PreOrd}(\mathbf Q)\to\mathrm{PreOrd}(\mathbf R)$ that acts on morphism  by $f\mapsto Ff$.
 Moreover, $\mathrm{PreOrd}(F)$ is (fully) faithful if $F$ is (fully) faithful.
\end{lemma}
\begin{proof}
    Let $(X,\qo_X)$ be a preordered object in $\mathbf Q$. The map $F_{X,Y}:\mathbf Q(X,Y)\to\mathbf R(FX,FY)$, $f\mapsto Ff$ preserves daggers and suprema, hence is monotone. Hence, $\id_{FX}=F(\id_X)\leq F(\qo_X)$, and $F(\qo_X)\circ F(\qo_X)=F(\qo_X\circ\qo_X)\leq F(\qo_X)$, so $F(\qo_X)$ is a preorder on $FX$ in $\mathbf R$, which assures that $\mathrm{PreOrd}(F)$ is well defined on objects. Let $(Y,\qo_Y)$ be another preordered object of $\mathbf Q$, and let $f:(X,\qo_X)\to(Y,\qo_Y)$ be a monotone map, i.e., $f\circ (\qo_X)\leq(\qo_Y)\circ f$. Firstly, since the action of $\mathrm{PreOrd}(F)$ on morphisms coincides with the action of $\mathrm{Maps}(F)$ on morphisms, it follows from Lemma \ref{lem:Maps_functor} that $Ff:FX\to FY$ is a map in $\mathbf R$. We also obtain $Ff\circ F(\qo_X)=F(f\circ \qo_X)\leq F(\qo_Y\circ f)=F(\qo_Y)\circ Ff$, hence $Ff:(FX,F(\qo_X))\to (FY,F(\qo_Y))$ is monotone. Thus, $\mathrm{PreOrd}(F)$ is also well defined on morphisms, and clearly it is functorial. 
    
    Clearly, if $F$ is faithful, so is $\mathrm{PreOrd}(F)$. Assume, in addition, that $F$ is full, and let $g:(FX,F(\qo_X))\to (FY,F(\qo_Y))$ be a monotone map. In particular, it follows that $g:FX\to FY$ is a map in $\mathbf R$, and by Lemma \ref{lem:Maps_functor}, it follows that there is a map $f:X\to Y$ in $\mathbf Q$ such that $Ff=g$. Since $g$ is monotone, and $F$ functorial, we have $F(f\circ \qo_X)=g\circ F(\qo_X)\leq F(\qo_Y)\circ g=F(\qo_Y\circ f)$. Now, since $F$ is full, it follows from Lemma \ref{lem:fully faithful homomorphism of quantaloids} that the map $F_{X,Y}:\mathbf Q(X,Y)\to\mathbf R(FX,FY)$, $f\mapsto Ff$ is an order isomorphism, whence $f\circ (\qo_X)\leq (\qo_Y)\circ f$, so $f$ is monotone. We conclude that $\mathrm{PreOrd}(F)$ is indeed full.
\end{proof}

\begin{lemma}\label{lem:monoidal product of preordered objects}
    Let $(X,\qo_X)$ and $(Y,\qo_Y)$ be preordered objects in a dagger symmetric monoidal quantaloid $(\mathbf R,\otimes,I)$. Then  $(X,\qo_X)\otimes(Y,\qo_Y):=(X\otimes Y,\qo_X\otimes \qo_Y)$ is a preordered object as well.
\end{lemma}
\begin{proof}
Since $(\mathbf R,\otimes,I)$ is a dagger symmetric monoidal  quantaloid, the order relation on morphisms respects daggers and the monoidal product. Hence, we have $\id_{X\otimes Y}=\id_X\otimes\id_Y\leq \qo_X\otimes\qo_Y$, so $\qo_X\otimes\qo_Y$ is reflexive. We also have $(\qo_X\otimes\qo_Y)\circ(\qo_X\otimes\qo_Y)=(\qo_X\circ\qo_X)\otimes(\qo_Y\circ\qo_Y)\leq \qo_X\otimes\qo_Y$. 
\end{proof}

\begin{lemma}\label{lem:monoidal product of monotone functions}
    Let $(X,\qo_X)$, $(Y,\qo_Y)$, $(W,\qo_W)$ and $(Z,\qo_Z)$ be preordered objects in a dagger symmetric monoidal quantaloid $\mathbf R$. Let $f:(X,\qo_X)\to (W,\qo_W)$ and $g:(Y,\qo_Y)\to (Z,\qo_Z)$ be monotone maps. Then $f\otimes g:(X,\qo_X)\otimes(Y,\qo_Y)\to (W,\qo_W)\otimes (Z,\qo_Z)$ is a monotone map.
\end{lemma}
\begin{proof}Using that the monoidal product in a symmetric monoidal quantaloid preserves the order in both arguments separately, we obtain
$(f\otimes g)\circ (\qo_X\otimes \qo_Y)=(f\circ\qo_X)\otimes(g\circ\qo_Y)\leq (\qo_W\circ f)\otimes(\qo_Z\circ g)=(\qo_W\otimes\qo_Z)\circ(f\otimes g)$.
\end{proof}

\begin{theorem}\label{thm:monoidal structure PreOrd}
Let $(\mathbf R,\otimes,I)$ be a dagger symmetric monoidal quantaloid. The category $\mathrm{PreOrd}(\mathbf R)$ becomes a symmetric monoidal category as follows:
\begin{itemize}
    \item We define the monoidal product by $(X,\qo_X)\otimes(Y,\qo_Y):=(X\otimes Y,\qo_X\otimes\qo_Y)$ on objects, and on monotone maps by the monoidal product of their underlying morphisms in $\mathbf R$;
    \item The monoidal unit is $(I, \id_I)$;
    \item the associator, unitors and symmetry between preordered objects are the respective associator, unitors and symmetry between the underlying objects of $\mathbf R$.
\end{itemize}   Moreover, the inclusion functor $J:\mathrm{Maps}(\mathbf R)\to\mathrm{PreOrd}(\mathbf{R})$, $X\mapsto (X,\id_X)$ is strict monoidal, and left adjoint to the forgetful functor $U:\mathrm{PreOrd}(\mathbf R)\to\mathrm{Maps}(\mathbf R)$, $(X,\qo)\mapsto X$.
\end{theorem}

\begin{proof}
It follows from Lemmas \ref{lem:monoidal product of preordered objects} and \ref{lem:monoidal product of monotone functions} that $\otimes:\mathrm{PreOrd}(\mathbf R)\times\mathrm{PreOrd}(\mathbf R)\to\mathrm{PreOrd}(\mathbf R)$ is a well defined bifunctor.
By Lemma \ref{lem:Maps form monoidal subcategory}, $\mathrm{Maps}(\mathbf R)$ inherits its monoidal structure from $\mathbf R$. Let $(X,\qo_X)$, $(Y,\qo_Y)$ and $(Z,\qo_Z)$ be preordered objects in $\mathbf R$. We need to show that the associator $\alpha_{X,Y,Z}:(X\otimes Y)\otimes Z\to X\otimes (Y\otimes Z)$, the left unitor $\lambda_X:I\otimes X\to X$, the right unitor $\rho_X:X\otimes I\to X$ and the symmetry $\sigma_{X,Y}:X\otimes Y\to Y\otimes X$ are order isomorphisms, which in the light of Lemma \ref{lem:characterization of order isomorphisms} means that we have to show that
\begin{align*}
\alpha_{X,Y,Z}\circ((\qo_X\otimes\qo_Y)\otimes\qo_Z) & =(\qo_X\otimes(\qo_Y\otimes\qo_Z))\circ \alpha_{X,Y,Z},\\
\lambda_{X}\circ (\id_I\otimes\qo_X) & = \qo_X\circ\lambda_X\\
\rho_X\circ (\qo_X\otimes\id_I) & =\qo_X\circ\rho_X,\\
\sigma_{X,Y}\circ(\qo_X\otimes\qo_Y) & = (\qo_Y\otimes\qo_X)\circ\sigma_{Y,X},
\end{align*}
but this follows directly because $\alpha$, $\lambda$, $\rho$, and $\sigma$ are natural isomorphisms in $\mathbf R$.

 Finally, it also follows from Example \ref{ex:trivial order} and Lemma \ref{lem:functions with trivially ordered domain are monotone} that the assignment $X\mapsto(X,\id_X)$ extends to an inclusion functor $\mathrm{Maps}(\mathbf R)\to\mathrm{PreOrd}(\mathbf R)$. For any two object $X$ and $Y$ of $\mathrm{Maps}(\mathbf R)$, we have $JX\otimes JY=(X,\id_X)\otimes(Y,\id_Y)=(X\otimes Y,\id_X\otimes\id_Y)=(X\otimes Y,\id_{X\otimes Y})=J(X\otimes Y)$, and $JI=(I,\id_I)$, from which follows that $J$ is strict monoidal. To show that $J$ is left adjoint to $U$, let $X$ be an object of $\mathrm{Maps}(\mathbf R)$, we need a candidate unit for the adjunction, so a map $X\to UJX$. Since $UJX=X$, we can choose this map to be the identity $\id_X$. Now let $(Y,\qo)$ be a preordered object of $\mathbf R$, and let $f:X\to U(Y,\qo)=Y$ be a map. We need to show that there is a unique monotone map $g:JX\to (Y,\qo)$ such that the following diagram commutes:
 \[ \begin{tikzcd}
 X\ar{r}{\id_X}\ar{dr}[swap]{f} & UJX\ar{d}{Ug}  \\
& U(Y,\qo). 
\end{tikzcd} \]
Since $UJX=X$ and $U(Y,\qo)=Y$, the only possible choice would be $g=f$, for which we have to verify that $f$ is a monotone map $JX\to (Y,\qo)$. But this follows directly from Lemma \ref{lem:functions with trivially ordered domain are monotone}.
\end{proof}

\begin{proposition}\label{prop:PreOrd_monoidal}
    Let $F:(\mathbf Q,\odot,J)\to (\mathbf R,\otimes,I)$ be be a homomorphism of dagger symmetric monoidal quantaloids with coherence dagger isomorphisms $\psi:J\to FI$ and $\psi_{X,Y}:FX\otimes FY\to F(X\odot Y)$ for each $X,Y\in\mathbf Q$. Then $\mathrm{PreOrd}(F)$ defined in Lemma \ref{lem:PreOrd_functor} is strong symmetric monoidal, with coherence isomorphisms $\psi:(J,\id_J)\to F(I,\id_{I})$ and $\psi_{(X,\qo_X),(Y,\qo_Y)}:F(X,\qo_X)\otimes F(Y,\qo_Y)\to F(X,\qo_X)\odot F(Y,\qo_Y)$ induced directly by by the coherence dagger isomorphisms of $F$, namely $\psi:J\to FI$ and $\psi_{X,Y}$ above.
\end{proposition}
    
\begin{proof}
Let $(X,\qo_X)$ and $(Y,\qo_Y)$ be preorderd objects of $\mathbf Q$. It follows from Proposition \ref{prop:Maps_monoidal} that $\psi$ and $\psi_{X,Y}$ are bijective maps in $\mathbf R$. We need to show that these maps are actually order isomorphisms. For $\psi$ this is trivial. It remains to show that $\psi_{X,Y}:F(X,\qo_X)\otimes F(Y,\qo_Y)\to F((X,\qo_X)\odot(Y,\qo_Y))$ is an order isomorphism, , or equivalently, that $\psi_{X,Y}:(FX\otimes FY, F(\qo_X)\otimes F(\qo_Y))\to(F(X\odot Y),F(\qo_X\odot \qo_Y))$ is an order isomorphism. By Lemma  \ref{lem:characterization of order isomorphisms} it is sufficient to show that $F(\qo_X\odot \qo_Y)\circ\psi_{X,Y}=\psi_{X,Y}\circ (F(\qo_X)\otimes F(\qo_Y))$, but this follows from naturality of $\psi$. 
\end{proof}

\subsection{Monotone relations}

\begin{definition}
Let $(X,\qo_X)$ and $(Y,\qo_Y)$ be preordered objects in a dagger quantaloid $\mathbf R$. We say that a morphism  $v:X\to Y$ in $\mathbf R$ is a \emph{monotone relation} $(X,\qo_X)\to(Y,\qo_Y)$ if it satisfies one of the following two equivalent conditions (hence both):
\begin{itemize}
    \item[(1)] $(\qop_Y)\circ v\leq v$ and $v\circ (\qop_X)\leq v$.
    \item[(2)] $(\qop_Y)\circ v=v=v\circ (\qop_X)$. 
\end{itemize}
\end{definition}
Clearly, (2) implies (1). For the other direction, we have $v=\id_Y\circ v\leq (\qop_Y)\circ v$, and $v=v\circ \id_X\leq v\circ (\qop_X)$. 

\begin{example}
    Let $(X,\qo_X)$ be a preordered object in a dagger quantaloid $\mathbf R$. Then $\qop_X$ is a monotone relation $(X,\qo_X)\to(X,\qo_X)$ as follows from the transitivity of $\qop_X$. 
\end{example}

\begin{lemma}\label{lem:composition monotone relations}
    Let $(X,\qo_X)$, $(Y,\qo_Y)$ and $(Z,\qo_Z)$ be preordered objects in a dagger quantaloid $\mathbf R$ and let $r:X\to Y$ and $s:Y\to Z$ be monotone relations. Then $s\circ r:X\to Z$ is a monotone relation.
    \end{lemma}
\begin{proof}The monotonicity of $r$ and $s$ implies that $r=r\circ\qop_X$ and $\qop_Z\circ s=s$, hence $\qop_Z\circ s\circ r=s\circ r=s\circ r\circ\qop_X$. 
\end{proof}

It follows from the previous lemma that the following categories are well defined.
\begin{definition}Let $\mathbf R$ be a dagger quantaloid. Then $\mathrm{MonRel}(\mathbf R)$ is defined as the category of preordered objects in $\mathbf R$ and monotone relations. The identity monotone relation  $\id_{(X,\qo)}$ on a preordered object $(X,\qo)$ is the monotone relation $\qop$. Instead of $\mathrm{MonRel}(\mathbf{Rel})$ we write $\mathbf{MonRel}$.
\end{definition}

\begin{lemma}\label{lem:MonRel is quantaloid}
Let $\mathbf R$ be a dagger quantaloid. Then $\mathrm{MonRel}(\mathbf R)$ is a quantaloid: If $(X,\qo_X)$ and $(Y,\qo_Y)$ are preordered objects in $\mathbf R$, then the supremum of a collection $(v_\alpha)_{\alpha\in A}$ monotone relations $(X,\qo_X)\to (Y,\qo_Y)$ is given by the supremum $\bigvee_{\alpha\in A}v_\alpha$ of $(v_\alpha)_{\alpha\in A}$ in $\mathbf R$. 
\end{lemma}
\begin{proof}
We need to show that $\bigvee_{\alpha\in A}r_\alpha$ is a monotone relation. Since $\mathbf R$ is a quantaloid, we find 
    \[ \qop_Y\circ\left(\bigvee_{\alpha\in A}v_\alpha\right)=\bigvee_{\alpha\in A}(\qop_Y\circ v_\alpha)=\bigvee_{\alpha\in A}v_\alpha=\bigvee_{\alpha\in A}(v_\alpha\circ\qop_X)=\left(\bigvee_{\alpha\in A}v_\alpha\right)\circ\qop_X.\qedhere\]
\end{proof}

\begin{proposition}\label{prop:biproducts in MonRel}
    Let $\mathbf R$ be a dagger quantaloid with small dagger biproducts. Then $\mathrm{MonRel}(\mathbf R)$ is a quantaloid with small biproducts. 
    
    More specifically, let   $(X_\alpha,\qo_\alpha)_{\alpha\in A}$ be a set-indexed family of preordered objects of $\mathbf R$. Then $\bigoplus_{\alpha\in A}(X_\alpha,\qo_\alpha)=(X,\qo_X)$ where $X=\bigoplus_{\alpha\in A}X_\alpha$ and $\qo_X=\bigoplus_{\alpha\in A}\qo_\alpha$. Moreover, if $p_{X_\beta}:X\to X_\beta$ and $i_{X_\beta}:X_\beta\to X$ denote the respective canonical projection and the canonical injection for each $\beta\in A$, then:
    \begin{itemize}
        \item the canonical projection map $p_{(X_\beta,\qo_\beta)}:\bigoplus_{\alpha\in A}(X_\alpha,\qo_\alpha)\to(X_\beta,\qo_\beta)$ is given by 
        \begin{equation}\label{eq:p_ordered}
        \qop_{X_\beta}\circ p_{X_\beta}=p_{X_\beta}\circ\qop_X;    
        \end{equation}
        
        \item the canonical injection map $i_{(X_\beta,\qo_\beta)}:(X_\beta,\qo_\beta)\to\bigoplus_{\alpha\in A}(X_\alpha,\qo_\alpha)$ is given by 
        \begin{equation}\label{eq:i_ordered}
            i_{X_\beta}\circ\qop_{X_\beta}=\qop_X\circ i_{X_\beta}.
        \end{equation}
            \end{itemize}
\end{proposition}

\begin{proof}
 Firstly, by  Proposition \ref{prop:biproducts in quantaloids are monotone} we have $\id_X=\bigoplus_{\alpha\in A}\id_{X_\alpha}\leq \bigoplus_{\alpha\in A}\qo_\alpha=\qo_X$, and $\qo_X\circ\qo_X=\left(\bigoplus_{\alpha\in A}\qo_\alpha\right)\circ\left(\bigoplus_{\alpha\in A}\qo_\alpha\right)=\bigoplus_{\alpha\in A}\qo_\alpha\circ\qo_\alpha\leq\bigoplus_{\alpha\in A}\qo_\alpha=\qo_X$. Thus $(X,\qo_X)$ is a preordered object in $\mathbf R$.
By Lemma \ref{lem:dagger preserves dagger biproducts}, we have $\qop_X=\bigoplus_{\alpha\in A}\qop_\alpha$, whence (\ref{eq:p_ordered}) and (\ref{eq:i_ordered}) hold. These two equalities immediately imply that $p_{(X_\alpha,\qo_\alpha)} $ and $i_{(X_\alpha,\qo_\alpha)}$ are monotone relations for each $\alpha\in A$. 

By Lemma \ref{lem:MonRel is quantaloid}, $\mathrm{MonRel}(\mathbf R)$ is a quantaloid.
Using the characterization of biproducts in quantaloids in 
    Proposition \ref{prop:biproducts in quantaloids}, we have $p_{X_\beta}\circ i_{X_\alpha}=\delta_{\alpha,\beta}$ for each $\alpha,\beta\in A$, and $\bigvee_{\alpha\in A}i_{X_\alpha}\circ p_{X_\alpha}=\id_X$. From the first identity it follows that for each $\alpha,\beta\in $ that
    \[p_{(X_\beta,\qo_\beta)}\circ i_{(X_\alpha,\qo_\alpha)}=\qop_\beta\circ p_{X_\beta}\circ i_{X_\alpha}\circ\qop_\alpha=\qop_\beta\circ\delta_{X_\alpha,X_\beta}\circ\qop_\alpha=\delta_{(X_\alpha,\qo_\alpha),(X_\beta,\qo_\beta)}.\]
        From the second identity, and using that $\mathbf R$ is a quantaloid, so pre- and postcomposition preserve suprema, we obtain
    \begin{align*}\bigvee_{\alpha\in A}i_{(X_\alpha,\qo_\alpha)}\circ p_{(X_\alpha,\qo_\alpha)}& =\bigvee_{\alpha\in A}\qop_X\circ i_{X_\alpha}\circ p_{X_\alpha}
        \circ \qop_X=\qop_X\circ\left(\bigvee_{\alpha\in A}i_{X_\alpha}\circ p_{X_\alpha}\right)\circ\qop_X\\
    & =\qop_X\circ\id_X\circ\qop_X=\qop_X=\id_{(X,\qo_X)}.\end{align*}
    Since suprema of parallel morphisms in $\mathrm{MonRel}(\mathbf R)$ coincide with the suprema of these morphisms in $\mathbf R$, it follows from Proposition \ref{prop:biproducts in quantaloids} that $(X,\qo_X)=\bigoplus_{\alpha\in A}(X_\alpha,\qo_\alpha)$ in $\mathrm{MonRel}(\mathbf R)$ with projection and injection morphisms $p_{(X_\alpha,\qo_\alpha)}$ and $i_{(X_\alpha,\qo_\alpha)}$, respectively.
\end{proof}

\begin{lemma} Let $\mathbf R$ be a dagger quantaloid. The assignment $X\mapsto X^\op$ extends to a functor $(-)^{\op}:\mathrm{MonRel}(\mathbf R)\to\mathrm{MonRel}(\mathbf R)^\op$, which acts on monotone relations $v:X\to Y$ by $v^\op=v^\dag$. Moreover, this functor $(-)^\op$ is involutory, hence an isomorphism of categories.
\end{lemma}
\begin{proof}
    Let $(X,\qo_X)$, $(Y,\qo_Y)$ and $(Z,\qo_Z)$ be preordered objects in $\mathbf R$. Let $v:(X,\qo_X)\to(Y,\qo_Y)$ be a monotone relation. Then $v\circ \qop_X=v=\qop_Y\circ v$, hence $\qo_X\circ v^\dag=v^\dag=v^\dag\circ \qo_Y$, showing that $v^\dag:(Y,\qop_Y)\to (X,\qop_X)$ is a monotone relation. We check functoriality. We have $\id_{(X,\qo_X)}^\op=(\qop_X)^\op=\qo_X=\id_{(X,\qop_X)}=\id_{(X,\qo_X)^\op}$, and if $w:(Y,\qo_Y)\to(Z,\qo_Z)$ is another monotone relation, we have $(w\circ v)^\op=(w\circ v)^\dag=v^\dag\circ w^\dag=v^\op\circ w^\op$.
    Finally, we have $((X,\qo_X)^{\op})^\op=(\X,\qop_X)^\op=(\X,\qo_X)$, and $(v^{\op})^\op=(v^\dag)^\op=v^{\dag\dag}=v$, so $(-)^\op$ is involutory.
\end{proof}

\begin{definition}\label{def:diamond}
    Let $\mathbf R$ be a dagger quantaloid, let $X\in\mathbf R$ be an object, and let  $(Y,\qo_Y)$ be a preordered object in $\mathbf R$. For any morphism $r:X\to Y$ in $\mathbf R$, we define $r_\diamond:X\to Y$ and $r^\diamond:Y\to X$ as the morphisms in $\mathbf R$ given by 
    \begin{align*}
        r_\diamond :=( \qop_Y)\circ r;\qquad
        r^\diamond  := r^\dag\circ(\qop_Y)
    \end{align*}
\end{definition}

\begin{lemma}\label{lem:embedding qPOS into qRelPos}
Let $\mathbf R$ be a dagger quantaloid. There are functors $(-)_\diamond:\mathrm{PreOrd}(\mathbf R)\to\mathrm{MonRel}(\mathbf R)$ and $(-)^\diamond:\mathrm{PreOrd}(\mathbf R)\to\mathrm{MonRel}(\mathbf R)^\op$, which are the identity on objects, and which acts on monotone maps $f:(X,\qo_X)\to(Y,\qo_Y)$ by $f\mapsto f_\diamond$ and $f\mapsto f^\diamond$, respectively (cf. Definition \ref{def:diamond}). Moreover, for each monotone map $f:X\to Y$, the following identities hold:
\[f_\diamond\circ f^\diamond\leq\id_{(Y,\qo_Y)},\qquad f^\diamond\circ f_\diamond\geq \id_{(X,\qo_X)}.\]
\end{lemma}

\begin{proof}
We first check that $f_\diamond$ is a monotone relation if $f:(X,\qo_X)\to(Y,\qo_Y)$ is a monotone map between preordered objects of $\mathbf R$. We immediately find     
    $(\qop_Y)\circ f_\diamond = (\qop_Y)\circ (\qop_Y)\circ f\leq (\qop_Y)\circ f=f_\diamond$. By Lemma \ref{lem:functions and the opposite order}, $f$ is also a monotone map $(X,\qop_X)\to(Y,\qop_Y)$, whence $f\circ (\qop_X)\leq (\qop_Y)\circ f$. 
Moreover, we have
hence $f_\diamond\circ (\qop_X)=(\qop_Y)\circ f\circ (\qop_X)\leq (\qop_Y)\circ (\qop_Y)\circ f\leq (\qop_Y)\circ f=f_\diamond$.
Next, we check functoriality. For  $\id_X:(X,\qo_X)\to(X,\qo_Y)$, we have $(\id_X)_\diamond=(\qop_X)\circ \id_X=(\qop_X)$, which is indeed the identity monotone relation $\id_{(X,\qo_X)}$ on $(X,\qo_X)$. Furthermore, given another preordered object $(Z,\qo_Z)$ and monotone map $g:(Y,\qo_Y)\to(Z,\qo_Z)$, we have 
\begin{align*}
    (g\circ f)_\diamond= (\qop_Z)\circ g\circ f=g_\diamond\circ f=g_\diamond\circ (\qop_Y)\circ f=g_\diamond\circ f_\diamond,
\end{align*}
so $(-)_\diamond$ is indeed a functor. In a similar way, one proves that $(-)^\diamond$ is a contravariant functor.

Finally, given a monotone map $f:X\to Y$, two direct calculations yield:  
$f_\diamond\circ f^\diamond =\qop_Y\circ f\circ f^\dag\circ \qop_Y\leq \qop_Y\circ \qop_Y\leq \qop_Y=\id_{(Y,\qo_Y)}$, and
$f^\diamond\circ f_\diamond= f^\dag\circ \qop_Y\circ \qop_Y\circ f=f^\dag\circ \qop_Y\circ f\geq f^\dag\circ f\circ \qop_X\geq \qop_X=\id_{(X,\qo_X)}$. 
\end{proof}

We note that if $\mathbf{R}=\mathbf{Rel}$, then any monotone relation $r:X\to Y$ that has an upper adjoint $s:Y\to X$ in $\mathbf{Rel}$, i.e., $s\circ r\geq \id_X$ and $r\circ s\leq \id_Y$, must be of the form $r=f_\diamond$ for some monotone map $f:X\to Y$, in which case $s=f^\diamond$ \cite[Footnote 3]{KurzMoshierJung}. This does not hold in general. For instance, take $\mathbf{R}=\mathbf{qRel}$. In Example \ref{ex:qRel}, we construct an invertible binary relation $R:\X\to\X$ in $\mathbf{qRel}$ that is not a dagger isomorphism, i.e., the inverse $S$ of $R$ does not equal $R^\dag$. Since $S$ is the inverse of $R$, it is its upper adjoint in $\mathbf{qRel}$. When we equip $\X$ with the trivial order, then $R$ becomes a monotone relation. If a monotone map $F:\X\to\X$ such that $F_\diamond=R$ exists, then it must be equal to $R$ for the order on $\X$ is trivial. In order for $F$ to be a map, $F^\dag$ must be its upper adjoint in $\mathbf{qRel}$, but since upper adjoints are unique, we would obtain $R^\dag=F^\dag=S$, which is a contradiction.

\begin{proposition} For any dagger quantaloid $\mathbf R$ the following diagram commutes:
\[\begin{tikzcd}
\mathrm{PreOrd}(\mathbf R)\ar{r}{(-)^\op}\ar{d}[swap]{(-)_\diamond} & \mathrm{PreOrd}(\mathbf R)\ar{d}{(-)^\diamond}\\
\mathrm{MonRel}(\mathbf R)\ar{r}[swap]{(-)^\op} & \mathrm{MonRel}(\mathbf R)^\op
\end{tikzcd}\]
\end{proposition}
\begin{proof}
    Let $(X,\qo_X)$ be a preordered object in $\mathbf R$. Then
    \[ ((X,\qo_X)^\op)^\diamond=(X,\qop_X)^\diamond=(X,\qop_X)=(X,\qo_X)^\op=((X,\qo_X)^\op)_\diamond. \]
    Let $(Y,\qo_Y)$ be another preordered object in $\mathbf R$ and let $f:(X,\qo_X)\to(Y,\qo_Y)$ be a monotone map. Then:

\begin{align*}
    [[f:(X,\qo_X)\to(Y,\qo_Y)]^\op]^\diamond & = [f:(X,\qop_X)\to(Y,\qop_Y)]^\diamond \\ 
    &  = f^\dag\circ \qo_Y:(Y,\qop_Y)\to (X,\qop_X)\\
    & = [ \qop_Y\circ f:(X,\qo_X)\to(Y,\qo_Y)]^\op\\
    & = [[f:(X,\qo_X)\to(Y,\qo_Y)]_\diamond]^\op. \qedhere
\end{align*}
\end{proof}

Recall Lemma \ref{lem:monoidal product of preordered objects} that states that the monoidal product $(X,\qo_X)\otimes(Y,\qo_Y):=(X\otimes Y,\qo_X\otimes\qo_Y)$ of preordered objects in a dagger symmetric monoidal quantaloid $(\mathbf Q,\otimes,I)$ is again a preordered object. We now define the monoidal product of monotone relations between preordered objects.

\begin{lemma}\label{lem:monoidal product of monotone relations}
    Let $(X,\qo_X)$, $(Y,\qo_Y)$, $(W,\qo_W)$ and $(Z,\qo_Z)$ be preordered objects in a dagger symmetric monoidal quantaloid ($\mathbf R,\otimes,I)$. Let $r:(X,\qo_X)\to (W,\qo_W)$ and $s:(Y,\qo_Y)\to (Z,\qo_Z)$. Then $r\otimes s:(X,\qo_X)\otimes(Y,\qo_Y)\to (W,\qo_W)\otimes (Z,\qo_Z)$ is a monotone relation.
\end{lemma}
\begin{proof}
By a direct calculation: $(\qop_W\otimes\qop_Z)\circ (r\otimes s)=(\qop_W\circ r)\otimes(\qop_Z\circ r)=r\otimes s=(r\circ\qop_X)\otimes(s\circ\qop_Y)=(r\otimes s)\circ(\qop_X\otimes \qop_Y)$.
\end{proof}

\begin{proposition}\label{prop:MonRel is symmetric monoidal}
   Let $(\mathbf R,\otimes,I)$ be a dagger symmetric monoidal quantaloid. Then
   $\mathrm{MonRel}(\mathbf R)$ is a symmetric monoidal quantaloid if we equip it with a monoidal product $\otimes$ as follows:
   \begin{itemize}
       \item The monoidal product $\otimes$ coincides with the monoidal product on $\mathrm{PreOrd}(\mathbf R)$ as defined in Theorem \ref{thm:monoidal structure PreOrd}, i.e., $(X,\qo_X)\otimes (Y,\qo_Y)=(X\otimes Y,\qo_X\otimes\qo_Y)$;
       \item the monoidal product $r\otimes s$ of monotone relations $r$ and $s$ is given by the monoidal product of $r$ and $s$ in $\mathbf R$;
       \item the monoidal unit is given by $(I,\id_I)$;
       \item if $\alpha$, $\lambda$, $\rho$ and $\sigma$ denote the respective associator, left unitor, right unitor and symmetry of $(\mathrm{PreOrd}(\mathbf R),\otimes,(I,\id_I))$, then  the associator, left unitor, right unitor and symmetry of $\mathrm{MonRel}(\mathbf R)$ are given by $\alpha_\diamond$, $\lambda_\diamond$, $\rho_\diamond$, and $\sigma_\diamond$, respectively.
   \end{itemize}
\end{proposition}
\begin{proof}
 Since $\otimes$ is a symmetric monoidal product on $\mathbf{R}$, it follows from Lemma \ref{lem:monoidal product of monotone relations} that it induces a bifunctor on $\mathrm{MonRel}(\mathbf{R})$. By Theorem \ref{thm:monoidal structure PreOrd}, the underlying morphisms in $\mathbf R$ of the components of the associator, left unitor, right unitor and symmetry of $\mathrm{PreOrd}(\mathbf R)$ are monotone maps, hence the components of $\alpha_{\diamond}$, $\lambda_\diamond$, $\rho_\diamond$ and $\sigma_\diamond$ are indeed monotone relations. We verify the naturality of these morphisms. So for $i=1,2$, let $(X_i,\qo_{X_i})$, $(Y_i,\qo_{Y_i})$ and $(Z_i,\qo_{Z_i})$ be preordered objects in $\mathbf R$, and let $u:X_1\to X_2$, $v:Y_1\to Y_2$ and $w:Z_1\to Z_2$ be monotone relations. 
  Then
 \begin{align*}
     (\alpha_{X_2,Y_2,Z_2})_\diamond\circ ((u\otimes v)\otimes w) &=(\qop_{X_2}\otimes (\qop_{Y_2}\otimes \qop_{Z_2}))\circ\alpha_{X_2,Y_2,Z_2}\circ ((u\otimes v)\otimes w) \\
     & = (\qop_{X_2}\otimes (\qop_{Y_2}\otimes \qop_{Z_2}))\circ  (u\otimes (v\otimes w))\circ\alpha_{X_1,Y_1,Z_1}\\
     & = (u\otimes (v\otimes w))\circ (\qop_{X_1}\otimes (\qop_{Y_1}\otimes \qop_{Z_1}))\circ\alpha_{X_1,Y_1,Z_1}\\
     & = (u\otimes (v\otimes w))\circ  (\alpha_{X_1,Y_1,Z_1})_\diamond,
 \end{align*}
 where we used naturality of $\alpha$ is associator of $\mathbf R$ in the second equality, and the fact that $u$, $v$ and $w$, hence also $u\otimes(v\otimes w)$ are monotone relations in the third equality. For the unitors and the symmetry the proof proceeds analogously. Then, since $\alpha$, $\lambda$, $\rho$, and $\sigma$ satisfy the coherence conditions for symmetric monoidal categories, and $(-)_\diamond$ is a functor, it follows that $\alpha_\diamond$, $\lambda_\diamond$, $\rho_\diamond$, and $\sigma_\diamond$ satisfy the same coherence conditions. So $\mathrm{MonRel}(\mathbf R)$ is indeed a symmetric monoidal category. Moreover, $\mathrm{MonRel}(\mathbf R)$ is a quantaloid where the supremum of parallel monotone relations is calculated in $\mathbf R$ by Lemma \ref{lem:MonRel is quantaloid}. Since also the monoidal product of morphism in $\mathrm{MonRel}(\mathbf R)$ is the same as the monoidal product of morphism in $\mathbf R$, which, by assumption is a symmetric monoidal quantaloid, it follows that the monoidal product on $\mathrm{MonRel}(\mathbf R)$ preserves suprema in both arguments separately. Thus  $(\mathrm{MonRel}(\mathbf{R}),\otimes,(I,\id_I))$ is a symmetric monoidal quantaloid.
\end{proof}

\begin{theorem}
    Let $(\mathbf R,\otimes,I)$ be a dagger compact quantaloid with respective unit and counit morphisms $\eta_X:I\to X^*\otimes X$ and $\epsilon_X:X\otimes X^*\to I$ for each object $X$. Then $(\mathrm{MonRel}(\mathbf R),\otimes,(I,\id_I))$ is a compact quantaloid with respective unit and counit morphism $\eta_{(X,\qo)}:(I,\id_I)\to (X,\qo)^*\otimes(X,\qo)$ and $\epsilon_{(X,\qo)}:(X,\qo)\otimes(X,\qo)^*\to(I,\id_I)$ given by $\eta_{(X,\qo)}:=(\qop^*\otimes\qop)\circ\eta_X$ and $\epsilon_{(X,\qo)}:=\epsilon_X\circ(\qop\otimes\qop^*)$.
\end{theorem}

\begin{proof}
By Proposition \ref{prop:MonRel is symmetric monoidal}, $\mathrm{MonRel}(\mathbf R)$ is a symmetric monoidal quantaloid. 

Let $(X,\qo)$ be a preordered object of $\mathbf R$. Since $\qo\circ\qo\leq \qo$, and $\id_X\leq\qo$, we have $\qo=\id_X\circ\qo\leq\qo\circ\qo\leq\qo$, whence $\qo\circ\qo=\qo$, so preorders are idempotent. We will use this in the remainder of proof without mentioning it. Note that by Lemma \ref{lem:dagger_and_dual}, we have $(\qo^*)^\dag=(\qo^\dag)^*=\qop^*$. Note furthermore that if $f:(X,\qo)\to(Y,\qo_Y)$ is an order isomorphism between preordered objects, functoriality of $(-)_\diamond$ yields that $f_\diamond$ is also invertible in $\mathrm{MonRel}(\mathbf R)$, and $(f_\diamond)^{-1}=(f^{-1})_\diamond=\qop\circ f$. We simply write $f_\diamond^{-1}$ instead of  $(f_\diamond)^{-1}$ or $(f^{-1})_\diamond$.
We have show to that the  unit and counit satisfy the zigzag identities of a compact closed category, for which we
used string diagrams. By definition, we have the following identities.
\begin{center}
\vcbox{\includegraphics{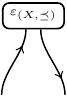}}
\vcbox{$\overset{(D1)}{=}$}
\vcbox{\includegraphics{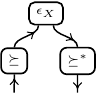}}
\qquad\qquad
\vcbox{\includegraphics{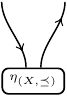}}
\vcbox{$\overset{(D2)}{=}$}
\vcbox{\includegraphics{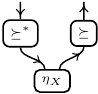}}
\qquad\qquad
\vcbox{\includegraphics{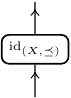}}
\vcbox{$\overset{(D3)}{=}$}
\vcbox{\includegraphics{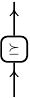}}
\end{center}
Furthermore, we have the following identities, where the first is the first zigzag identity for $\mathbf R$ as a compact category, the second by Lemma \ref{lem:dagger_and_dual}, and the third by combining reflexivity and transitivity in the definition of a preorder, respectively:
\begin{center}
\vcbox{\includegraphics{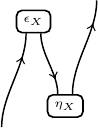}}
\vcbox{$\overset{(I1)}{=}$}
\vcbox{\includegraphics{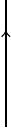}}
\qquad\qquad
\vcbox{\includegraphics{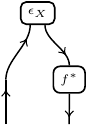}}
\vcbox{$\overset{(I2)}{=}$}
\vcbox{\includegraphics{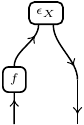}}
\qquad\qquad
\vcbox{\includegraphics{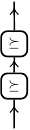}}
\vcbox{$\overset{(I3)}{=}$}
\vcbox{\includegraphics{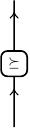}}
\end{center}
Then, suppressing associators and unitors, as usual with string diagrams, we find:
\begin{center}
\vcbox{\includegraphics{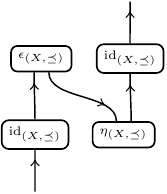}}
\vcbox{$\overset{(D1)-(D3)}{=}$}
\vcbox{\includegraphics{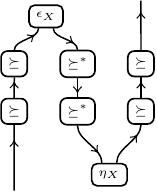}}
\vcbox{$\overset{(I3)}{=}$}
\vcbox{\includegraphics{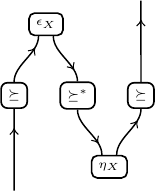}}
\vcbox{$\overset{(I2)}{=}$}\\
\end{center}
\begin{center}
\vcbox{\includegraphics{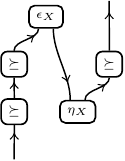}}
\vcbox{$\overset{(I3)}{=}$}
\vcbox{\includegraphics{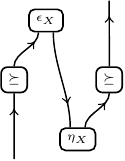}}
\vcbox{$\overset{(I1)}{=}$}
\vcbox{\includegraphics{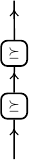}}
\vcbox{$\overset{(I3)}{=}$}
\vcbox{\includegraphics{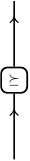}}
\vcbox{$\overset{(D3)}{=}$}
\vcbox{\includegraphics{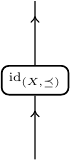} }
\end{center}
In a similar way, we show that $\mathrm{MonRel}(\mathbf R)$ satisfy the second zigzag identity of compact-closed categories, which proves the statement.
\end{proof}

\begin{proposition}\label{prop:MonRel-functor}
Let $F:(\mathbf Q,\odot,J)\to(\mathbf R,\otimes,I)$ be a homomorphism of dagger symmetric monoidal quantaloids. Then there is a homomorphism of symmetric monoidal quantaloids $\mathrm{MonRel}(F):\mathrm{MonRel}(\mathbf Q)\to\mathrm{MonRel}(\mathbf R)$ defined as follows: 
\begin{itemize}
    \item on objects by $(X,\qo)\mapsto (FX,F(\qo))$;
    \item on morphisms by $v\mapsto Fv$.
\end{itemize}
Moreover, $\mathrm{MonRel}(F)$ is (fully) faithful if $F$ is (fully) faithful.
\end{proposition}
\begin{proof}
By Lemma \ref{lem:PreOrd_functor}, the action of $\mathrm{MonRel}(F)$ is well defined on objects. Let $v:(X,\qo_X)\to(Y,\qo_Y)$ be a monotone relation in $\mathbf Q$, so $(\qo_Y)\circ v=v=v\circ(\qo_X)$. Then, by functoriality of $F$, we have $F(\qo_Y)\circ Fv=Fv=Fv\circ F(\qo_Y)$, so $Fv:(FX,F(\qo_X))\to (FY,F(\qo_Y))$ is a monotone relation in $\mathbf R$. The functoriality of $\mathrm{MonRel}(F)$ clearly follows from the functoriality of $F$. 

By Lemma \ref{lem:MonRel is quantaloid}, $\mathrm{MonRel}(\mathbf Q)$ and $\mathrm{MonRel}(\mathbf R)$ are quantaloids, and the suprema of parallel morphisms in these quantaloids can be calculated in $\mathbf Q$ and $\mathbf R$, respectively. Since $F$ is a homomorphism of dagger quantaloids, it preserves suprema of parallel morphisms, whence $\mathrm{MonRel}(F)$ is a homomorphism of quantaloids.

Assume $F$ is faithful. Then it follows immediately that also $\mathrm{MonRel}(F)$ is faithful. Now assume, in addition, that $F$ is full, and let $w:(FX,F(\qo_X))\to(FY,F(\qo_Y))$ be a monotone relation, so $F(\qo_Y)\circ w=w=w\circ F(\qo_X)$. In particular, we have that $w:FX\to FY$ is a morphism in $\mathbf R$, hence since $F$ is full, there must be some morphism $v:X\to Y$ in $\mathbf Q$ such that $Fv=w$. Then, by functoriality of $F$, we find $F(\qo_Y\circ v)=Fv=F(v\circ\qo_X)$, and since $F$ is faithful, we find $(\qo_Y)\circ v=v=v\circ(\qo_X)$, i.e., $v:(X,\qo_X)\to(Y,\qo_Y)$ is a monotone relation. We conclude that $\mathrm{MonRel}(F)$ is full.

Let $(X,\qo_X)$ and $(Y,\qo_Y)$ be objects of $\mathrm{MonRel}(\mathbf Q)$. Then they are objects of $\mathrm{PreOrd}(\mathbf Q)$ as well, and the monoidal product $(X,\qo_X)\odot(Y,\qo_Y)$ in $\mathrm{MonRel}(\mathbf Q)$ and in $\mathrm{PreOrd}(\mathbf Q)$ coincides by Proposition \ref{prop:MonRel is symmetric monoidal}. Similarly, $F(X,\qo_X)\otimes F(Y,\qo_Y)$ coincides in $\mathrm{MonRel}(\mathbf R)$ and $\mathrm{PreOrd}(\mathbf R)$. Since $\mathrm{PreOrd}(F)$ is strong symmetric monoidal by Proposition \ref{prop:PreOrd_monoidal}, we have coherence order isomorphisms $\psi_{(X,\qo_X),(Y,\qo_Y)}:F(X,\qo_X)\otimes F(Y,\qo_Y)\to F((X,\qo_X)\odot (Y,\qo_Y))$ and $\psi:(J,\id_J)\to F(I,\id_I)$. Then $(\psi_{(X,\qo_X),(Y,\qo_Y)})_\diamond:F(X,\qo_X)\otimes F(Y,\qo_Y)\to F((X,\qo_X)\odot (Y,\qo_Y))$ and $\psi_\diamond:(J,\id_J)\to F(I,\id_I)$ form coherence isomorphisms for $\mathrm{MonRel}(F)$, hence the latter functor is strong symmetric monoidal.
\end{proof}

\section{The embedding of sets}\label{sec:embedding-sets}

For the next lemma, recall the quantaloid $\mathbf 2$ (cf. Example \ref{ex:quantale-induced quantaloid}).

\begin{lemma}\label{lem:IQ}
Let $(\mathbf Q,\otimes,I)$ be a nontrivial dagger symmetric monoidal quantaloid. Then the functor $I_\mathbf Q:\mathbf 2\to\mathbf Q$ that acts on the single object by $1\mapsto I$, and on the morphisms by $1\mapsto \id_I$ and $0\mapsto 0_I$ is faithful. Moreover, if $\mathbf Q$ has precisely two scalars, then $I_\mathbf Q$ is full.
\end{lemma}

\begin{theorem}\label{thm:qRel quote}
    Let $(\mathbf R,\otimes,I)$ be a nontrivial dagger symmetric monoidal quantaloid with small dagger biproducts. Let $`(-):\mathbf{Rel}\to\mathbf{R}$ be the functor that maps any set $A$ to  $`A:=\bigoplus_{\alpha\in A}I$, and that maps any binary relation $r:A\to B$ between sets to the morphism $`r:`A\to `B$ whose matrix element $(`r)_{\alpha,\beta}$ is given by 
     \[(`r)_{\alpha,\beta}=\begin{cases}
    \id_I, & (\alpha,\beta)\in r,\\
    0_I, & (\alpha,\beta)\notin r,
    \end{cases}\] for each $\alpha\in A$ and each $\beta\in B$. Then $`(-)$ is a faithful homomorphism of dagger symmetric monoidal quantaloids that preserve all dagger biproducts. Moreover, if $\mathbf R$ has precisely two scalars, then $`(-)$ is full. 
    \end{theorem}
\begin{proof}
Since $\mathbf R$ has small dagger biproducts, it follows from Proposition \ref{prop:biproduct completion monoidal quantaloid} that the functors $P_\mathbf R:\mathrm{Matr}(\mathbf R)\to\mathbf R$ defined in Theorem \ref{thm:biproduct-completion} is an equivalence of dagger symmetric monoidal quantaloids.
By Proposition \ref{prop:Matr-homomorphism-of-dagger-symmetric-monoidal-quantaloids}, the functor $\mathrm{Matr}(I_\mathbf R):\mathrm{Matr}(\mathbf 2)\to\mathrm{Matr}(\mathbf R)$ defined in Theorem \ref{thm:biproduct-completion} is a homomorphism of dagger symmetric monoidal quantaloids that is faithful, because $I_\mathbf R$ is faithful, and that is full if $\mathbf R$ has precisely two scalars. Furthermore, by Example \ref{ex:V-Rel is dagger compact quantaloid}, the functor $F_2:\mathbf{Rel}\to \mathrm{Matr}(\mathbf 2)$ from Proposition \ref{prop:VRelisMatrV} is an equivalence of dagger symmetric monoidal quantaloids. We now precisely have  $`(-)=P_\mathbf Q\circ\mathrm{Matr}(I_\mathbf Q)\circ F_2$, which is a faithful homomorphism of dagger symmetric monoidal quantaloids since it is a composition of faithful homomorphisms of dagger symmetric monoidal quantaloids. We saw that if $\mathbf R$ has precisely two scalars, then $\mathrm{Matr}(I_\mathbf R)$ is also full, in which case clearly also $`(-)$ is full. Finally, by Proposition \ref{prop:quantaloid homomorphisms preserve biproducts} it follows that $`(-)$ preserves dagger biproducts.
\end{proof}

By the identifications $\mathbf{Set}=\mathrm{Maps}(\mathbf{Rel})$, $\mathbf{PreOrd}=\mathrm{PreOrd}(\mathbf{Rel})$ and $\mathbf{MonRel}=\mathrm{MonRel}(\mathbf{Rel})$, and the functoriality of $\mathrm{Maps}$ (cf. Proposition \ref{prop:Maps_monoidal}), $\mathbf{PreOrd}$ (cf. Proposition \ref{prop:PreOrd_monoidal}) and $\mathbf{MonRel}$ (Proposition \ref{prop:MonRel-functor}), Theorem \ref{thm:qRel quote} immediately yields the following corollaries:

\begin{corollary}\label{cor:qSet quote}
    Let $(\mathbf R,\otimes,I)$ be a nontrivial dagger symmetric monoidal quantaloid with small dagger biproducts. Then $`(-):\mathbf{Rel}\to\mathbf R$ restricts and corestricts to a faithful strong symmetric monoidal functor $\mathrm{Maps}(`(-)):\mathbf{Set}\to\mathrm{Maps}(\mathbf R)$. Moreover, if $\mathbf R$ has precisely two scalars, this functor is full as well. 
\end{corollary}

\begin{corollary}\label{cor:qPOS quote}   Let $(\mathbf R,\otimes, I)$ be a nontrivial dagger symmetric monoidal quantaloid with all small dagger biproducts. 
Then there is a  faithful strong symmetric monoidal functor $\mathrm{PreOrd}(`(-)):\mathbf{PreOrd}\to\mathrm{PreOrd}(\mathbf R)$, which:
\begin{itemize} \item is defined on objects by $(A,\sqsubseteq)\mapsto (`A,`{\sqsubseteq})$;
\item is defined on morphisms by  $f\mapsto `f$;
\item is full if $\mathbf R$ has precisely two scalars.
    \end{itemize}
\end{corollary}

\begin{corollary}\label{cor:qMonRel quote}Let $(\mathbf R,\otimes,I)$ be a nontrivial dagger symmetric monoidal quantaloid with small dagger biproducts. Then we have a faithful homomorphism of symmetric monoidal quantaloid  $\mathrm{MonRel}(`(-)):\mathbf{MonRel}\to\mathrm{MonRel}(\mathbf R)$ that preserves all biproduct that sends any preordered set $(A,\sqsubseteq_A)$ to $(`A,`{\sqsubseteq_A})$ and any monotone relation $v:(A,\sqsubseteq_A)\to(B,\sqsubseteq_B)$ to $`v$. Moreover, $\mathrm{MonRel}(`(-))$ preserves small biproducts. If $\mathbf R$ has precisely two scalars, then $\mathrm{MonRel}(`(-))$ is full.
\end{corollary}

\subsection{Convention}
With abuse of notation, we will denote the functors induced by $`(-):\mathbf{Rel}\to\mathbf R$ in Corollaries \ref{cor:qSet quote}, \ref{cor:qPOS quote} and \ref{cor:qMonRel quote} as well by $`(-)$.

\begin{proposition}
    Let $(\mathbf R,\otimes,I)$ be a nontrivial dagger symmetric monoidal quantaloid with all small dagger biproducts. Let $\mathbf S=\mathrm{Maps}(\mathbf R)$. Then the embedding $`(-):\mathbf{Set}\to\mathbf S$ has a right adjoint given by $\mathbf S(I,-)$. 
\end{proposition}
\begin{proof}
Given a set $A$, the $A$-component $\eta_A:A\to\mathbf S(I,`A)$ of the unit $\eta$ of the adjunction is defined by $\eta_A(\alpha)=i_\alpha$ for each $\alpha\in A$, where $i_\alpha:I\to`A$ is the canonical injection of $I$ into the $\alpha$-th factor of $`A=\bigoplus_{\alpha\in A}I_\alpha$. We note that $i_\alpha$ is a map in $\mathbf R$ by Proposition \ref{prop:coproducts of Maps}, hence $\eta_A$ is well defined. Now, let $X$ be an object of $\mathbf S$, and let $f:A\to\mathbf S(I,X)$ be a function.  We define $g:`A\to X$ as the morphism $[f(\alpha)]_{\alpha\in A}$. Since $f(\alpha):I\to X$ is a map in $\mathbf R$ for each $\alpha\in A$, it follows from Proposition \ref{prop:coproducts of Maps} that $g$ is a map in $\mathbf R$, so a morphism in $\mathbf S$. Then for each $\beta\in A$, we have $\mathbf S(I,g)\circ\eta_A(\beta)=g\circ\eta_A(\beta)=g\circ i_\beta=[f(\alpha)]_{\alpha\in A}\circ i_\beta=f(\beta)$, so $\mathbf S(I,g)\circ\eta_A=f$.
Given any other map $h:`A\to X$ such that $\mathbf S(I,h)\circ\eta_A=f$, we have $f(\alpha)=S(I,h)\circ\eta_A(\alpha)=h\circ\eta_A(\alpha)=h\circ i_\alpha$ for each $\alpha\in A$, which shows that $h=[f(\alpha)]_{\alpha\in A}=g$.
\end{proof}

In the following, we investigate when there is an object $\Omega$ in a dagger symmetric monoidal quantaloid $(\mathbf R,\otimes,I)$ for which there is a bijection between $\mathbf{R}(X,I)\cong \mathrm{Maps}(\mathbf R)(X,\Omega)$. This will be of importance for the next section in which we investigate the existence of power objects. 

\begin{lemma}\label{lem:characterization functions to IoplusI}
    Let $(\mathbf R,\otimes,I)$ be an affine dagger symmetric monoidal  quantaloid with small dagger biproducts. Let $X$ be an object of $\mathbf R$, let $A$ be a set, and let $(f_\alpha)_{\alpha\in A}$ be a set-indexed family in $\mathbf R(X,I)$. Write $f=\langle f_\alpha\rangle_{\alpha\in A}:X\to `{A}$. Then:
    \begin{itemize}
        \item[(a)] $f\circ f^\dag\leq\id_{`A}$ if and only if $f_\alpha\perp f_\beta$ for every district $\alpha,\beta\in A$;
        \item[(b)] $f^\dag\circ f\geq\id_X$ if and only if $\bigvee_{\alpha\in A}f_\alpha^\dag\circ f_\alpha\geq\id_X$. 
    \end{itemize}
\end{lemma}
\begin{proof}
We denote the embedding of $I$ onto the $\alpha$-th factor of $`A$ by $i_\alpha$. Its corresponding projection is denoted by $p_\alpha$, which satisfies  $p_\alpha=i_\alpha^\dag$.  Proposition \ref{prop:dagger nabla and delta} yields $(f\circ f^\dag)_{\alpha,\beta}=f_\beta\circ f_\alpha^\dag$ for each $\alpha,\beta\in A$.
By Lemma \ref{lem:matrix identity}, we have $(\id_{`A})_{\alpha,\beta}=\delta_{\alpha,\beta}$ for each $\alpha,\beta\in A$. 
It now follows from Proposition \ref{prop:biproducts in quantaloids are monotone} that $f\circ f^\dag\leq\id_{`A}$ if and only if 
$f_\alpha\circ f^\dag_\alpha\leq \id_I$ for each $\alpha\in A$ and $f_\alpha\circ f^\dag_\beta\leq 0_I=\perp_I$ (using Lemma \ref{lem:quantaloid with zero}) for each $\alpha\neq\beta$. Since by assumption, $(\mathbf R,\otimes,I)$ is affine, $\id_I$ is the largest scalar, hence the former condition always holds, whereas the second condition translates to $f_\alpha\perp f_\beta$ for $\alpha\neq\beta$.

For (b), let $g_\alpha=f_\alpha^\dag:I\to X$, and let $g=[g_\alpha]_{\alpha\in A}:`A\to X$. Then $f=g^\dag$ by Proposition \ref{prop:dagger nabla and delta}, which also yields $f^\dag\circ f=g\circ g^\dag=\bigvee_{\alpha\in A}g_\alpha\circ g_\alpha^\dag=\bigvee_{\alpha\in A}f_\alpha^\dag\circ f_\alpha$,
from which the statement follows.
\end{proof}

\begin{proposition}\label{prop:zero-mono effect iff function}
    Let $(\mathbf R,\otimes,I)$ be an affine dagger compact quantaloid with dagger kernels such that every object has exactly one $\bot$-mono effect. Then an effect $r:X\to I$ is a $\bot$-mono if and only if $r^\dag\circ r\geq \id_X$.   
\end{proposition}
\begin{proof}
By Theorem \ref{thm:dagger kernels imply orthomodularity}, the homsets of $\mathbf R$ are orthomodular lattices under the orthogonality relation $\perp$ given by $f\perp g$ if and only if $\Tr(f\circ g^\dag)=\perp_I$.
    Assume $r^\dag\circ r\geq \id_X$. Let $s:Z\to X$ be a morphism such that $r^\dag\circ r\circ s=\perp_{Z,X}$. Then $\perp_{Z,X}=r^\dag\circ r\circ s\geq \id_X\circ s=s$, forcing $s=\perp_{Z,X}$. Hence, $r^\dag\circ r$ a $\bot$-mono and it follows from Lemma \ref{lem:r-zero-mono-iff-rdagr-zero-mono} that also $r$ is a $\perp$-zero in $\mathbf R$. Conversely, assume $r$ is a $\bot$-mono. Let $p=r^\dag\circ r$. Then $p^\dag=p$. 
    Let $f:X\to X$ such that $p\perp f$. Then $\perp_I=\Tr(f\circ p^\dag)=\Tr(f\circ p)=\Tr(f\circ r^\dag\circ r)=\Tr(r\circ f\circ r^\dag)=r\circ f\circ r^\dag$, since $r\circ f\circ r^\dag$ is a scalar. Since $r$ is a $\bot$-mono, it follows that $f\circ r^\dag=\perp_{I,X}$, hence $r\circ f^\dag=\perp_{X,I}$. Again, since $r$ is a $\bot$-mono, it follows $f^\dag=\perp_X$, so also $f=\perp_X$. Thus we have shown that $f\perp p$ implies $f=\perp_X$. Since $\mathbf R$ is a dagger compact quantaloid, it follows from Lemma \ref{lem:trace preserves suprema} that the trace preserves arbitrary suprema. Hence,  $\Tr(\neg p\circ\id_X^\dag)=\Tr(\neg p)=\bigvee\{\Tr(f):f\perp p\}=\Tr(\perp_X)=\perp_I$. We conclude that $\neg p\perp\id_X$, i.e., $\id_X\leq \neg\neg p=p$. 
\end{proof}

\begin{theorem}\label{thm:functions with classical codomain}
    Let $(\mathbf R,\otimes,I)$ be an affine dagger compact quantaloid with small dagger biproducts and dagger kernels such that for each object $X$ of $\mathbf R$:
    \begin{itemize}
        \item[(1)] there is a $\bot$-monic effect $X\to I$;
        \item[(2)] every $\bot$-monic PER on $X$ is a equivalence relation on $X$.
        \end{itemize}
        Then for each object $X$ of $\mathbf R$ and each set $A$, any morphism $f=\langle f_\alpha\rangle_{\alpha\in A}:X\to `A$ is a map if and only if $f_\alpha\perp f_\beta$ for each distinct $\alpha,\beta\in A$ and $\bigvee_{\alpha\in A}f_\alpha=\top_{X,I}$.
\end{theorem}
\begin{proof}
For any object $X$ it follows from Proposition \ref{prop:PER} that $\top_{X,I}$ is the unique $\bot$-monic effect on $X$. 
Now, if $f:X\to`A$ is a map, it follows from Lemma \ref{lem:characterization functions to IoplusI} that $f_\alpha\perp f_\beta$ for distinct $\alpha,\beta\in A$ and that $\bigvee_{\alpha\in A}f^\dag_\alpha\circ f_\alpha\geq\id_X$. Then, using that $\mathbf R$ is a dagger quantaloid, we obtain
\[\left(\bigvee_{\alpha\in A}f_{\alpha}\right)^\dag\circ \left(\bigvee_{\beta\in A}f_\beta\right)=\bigvee_{\alpha,\beta\in A}f_\alpha^\dag\circ f_\beta\geq \bigvee_{\alpha\in A}f_\alpha^\dag\circ f_\alpha\geq\id_X,\]
hence $\bigvee_{\alpha\in A}f_\alpha$ is a $\bot$-monic effect on $X$ by Proposition \ref{prop:zero-mono effect iff function}.  Since $\top_{X,I}$ is the unique $\bot$-monic effect on $X$, we conclude that $\bigvee_{\alpha\in A}f_{\alpha}=\top_{X,I}$.

Conversely, assume that $f:X\to`A$ satisfies $f_\alpha\perp f_\beta$ for distinct $\alpha,\beta\in A$ and that $\bigvee_{\alpha\in A}f_\alpha= \top_{X,I}$. It follows from Lemma \ref{lem:characterization functions to IoplusI} that $f\circ f^\dag\leq\id_{`A}$. Moreover, since $\top_{X,I}$ is a $\bot$-mono, it follows that $\bigvee_{\alpha\in A}f_\alpha$ is a $\bot$-mono.

For each $\alpha\in A$, let $s_\alpha=f_\alpha^\dag\circ f_\alpha$. Let $s=\bigvee_{\alpha\in A}s_\alpha$. We are done if we can show that $s\geq\id_X$. Firstly, we clearly have $s_\alpha^\dag=s_\alpha$ for each $\alpha\in A$, whence $s^\dag=s$ for $\mathbf R$ is a dagger quantaloid. For each $\alpha\in A$, since $f_\alpha\circ f_\alpha^\dag$ is a scalar and $\mathbf R$ is affine, we have $f_\alpha\circ f_\alpha^\dag\leq\id_I$, hence $s_\alpha\circ s_\alpha=f_\alpha^\dag\circ f_\alpha\circ f_\alpha^\dag\circ f_\alpha\leq f_\alpha^\dag\circ\id_I\circ f_\alpha=s_\alpha$.
Now assume that $\alpha$ and $\beta$ in $A$ are distinct. By assumption, $f_\alpha\perp f_\beta$, so $f_\alpha\circ f_\beta^\dag=\Tr(f_\alpha\circ f_\beta^\dag)=\perp_I$, using Proposition \ref{prop:trace properties} in the first equality. Hence, $s_\alpha\circ s_\beta=f_\alpha^\dag\circ f_\alpha\circ f_\beta^\dag\circ f_\beta=\perp_X$. Then $s\circ s=\left(\bigvee_{\alpha\in A}s_\alpha\right)\circ\left(\bigvee_{\beta\in A}s_\beta\right)=\bigvee_{\alpha,\beta\in A}s_\alpha\circ s_\beta\leq \bigvee_{\alpha\in A}s_\alpha=s$, so $s$ is symmetric and transitive, hence a PER.

We claim that $s$ is a $\bot$-mono. So let $r:Y\to X$ be a morphism such that $s\circ r=\perp_{Y,X}$. Since $\mathbf R$ is a quantaloid, this implies $\bigvee_{\alpha\in A}s_\alpha\circ r=\perp_{Y,X}$, which is only possible if $s_\alpha\circ r=\perp_{Y,X}$ for each $\alpha\in A$. Then $r^\dag\circ f_\alpha^\dag\circ f_\alpha\circ r=r^\dag\circ s_\alpha\circ r=\perp_Y=0_Y$ for each $\alpha\in A$, where we used Lemma \ref{lem:quantaloid with zero} in the last equality. Since $\mathbf R$ has small dagger biproducts, it has a zero object, hence it is a dagger kernel category by Proposition \ref{prop:dagger kernel quantaloid}. As a consequence, we may apply Lemma \ref{lem:nondegeneracy} to conclude that $f_\alpha\circ r=0_{Y,I}=\perp_{Y,I}$ for each $\alpha\in A$. Hence, $\perp_{Y,I}=\bigvee_{\alpha\in A}f_\alpha\circ r=\top_{X,I}\circ r$, which implies $r=\perp_{Y,X}$ for $\top_{X,I}$ is a $\bot$-mono. So $s$ is indeed a $\bot$-mono. It now follows from assumption (2) that $s\geq\id_X$, i.e., $\bigvee_{\alpha\in A}f_\alpha^\dag\circ f_\alpha\geq\id_X$.
\end{proof}

\begin{corollary}\label{cor:power objects condition}
   Let $(\mathbf R,\otimes,I)$ be an affine dagger compact quantaloid with small dagger biproducts and dagger kernels such that for each object $X$ of $\mathbf R$:
    \begin{itemize}
        \item[(1)] there is a $\bot$-monic effect $X\to I$;
        \item[(2)] every $\bot$-monic PER on $X$ is a equivalence relation on $X$.
        \end{itemize}
 Let $\Omega=I\oplus I$, and denote the projection $\Omega\to I$ on the first factor by $p_0$, and the projection on the second factor by $p_1$. Then for each object $X$, we have a bijection
 \[\mathrm{Maps}(\mathbf R)(X,\Omega)\to\mathbf R(X,I),\qquad f\mapsto p_1\circ f\]
 whose inverse is given by $r\mapsto\langle \neg r,r\rangle$, where 
$\neg r$ is defined as in Lemma \ref{lem:characterizations of negation}.
\end{corollary}
\begin{proof}
        For any object $X$ it follows from Proposition \ref{prop:PER} that $\top_{X,I}$ is the unique $\bot$-monic effect on $X$. Hence, we can apply Theorem \ref{thm:dagger kernels imply orthomodularity} to conclude that homsets in $\mathbf R$ are orthomodular lattices with respect to the orthocomplementation $\neg$.
Let $\varphi$ be the map $\mathbf R(X,I)\to\mathrm{Maps}(\mathbf R)(X,\Omega)$, $r\mapsto\langle \neg r,r\rangle$. It follows directly from Theorem \ref{thm:functions with classical codomain} that $\varphi$ is well defined. Denote the map $\mathrm{Maps}(\mathbf R)(X,\Omega)\to\mathbf R(X,I)$, $f\mapsto p_1\circ f$ by $\psi$. Let $r\in\mathbf R(X,I)$. Then $\psi\circ\varphi(r)=\psi(\langle \neg r,r\rangle)=p_1\circ\langle \neg r,r\rangle=r$. Let $f:X\to\Omega$ be a map. Let $f_1=p_1\circ f$ and $f_0=p_0\circ f$, so $f=\langle f_0,f_1\rangle$ as morphism in $\mathbf R$. By Theorem \ref{thm:functions with classical codomain}, we have $f_0\perp f_1$ and $f_0\vee f_1=\top_{X,I}$, so $f_0=\neg f_1$. 
As a consequence, $\varphi\circ\psi(f)=\varphi(p_1\circ f)=\varphi(f_1)=\langle \neg f_1,f_1\rangle=\langle f_0,f_1\rangle=f$.
We conclude that $\varphi$ and $\psi$ are each other's inverses, which proves the statement.
\end{proof}

\begin{lemma}\label{lem:eval-identities}
Let $(\mathbf R,\otimes,I)$ be an affine dagger symmetric monoidal quantaloid with dagger biproducts. Let $2$ be the two-point set $\{0,1\}$ ordered by $\sqsubseteq$ via $0\sqsubseteq 1$. Let $\Omega=`2$ with projections $p_0$ and $p_1$ on the respective zero-th and first component of $\Omega$, and let $\qo_\Omega=`(\sqsubseteq)$. 
Then the following identities hold:
\begin{align*}
& p_0\circ \qo_\Omega   = p_0, &p_1\circ \qo_\Omega  = \top_{\Omega,I},\\    
       &  p_1\circ \qop_\Omega = p_1, &
    p_0\circ \qop_\Omega  = \top_{\Omega,I}.
        \end{align*}
\end{lemma}
\begin{proof}

We first prove the statement for $\mathbf{Rel}$, that is, we prove that
\begin{align*}
& q_0\circ (\sqsubseteq)  = q_0,
 &   q_1\circ (\sqsubseteq)  = \top_{2,1},\\    
       & q_1\circ (\sqsupseteq)  = q_1, & 
    q_0\circ (\sqsupseteq) = \top_{2,1}.
        \end{align*}
Here, for $\alpha=0,1$, $q_\alpha:2\to 1$ denotes the canonical projection on the $\alpha$-th factor, $q_0=\{(0,*)\}$ and $q_1=\{(1,*)\}$ if we regard $q_0,q_1$ as subsets of $2\times 1$.
Then, for each $\alpha\in 2$, we have $(\alpha,*)\in q_0\circ\sqsubseteq$ if and only if there is some $\beta\in 2$ such that $\alpha\sqsubseteq\beta$ and $(\beta,*)\in q_0$. The latter condition forces that $\beta=0$, which forces $\alpha=0$, i.e., $q_0\circ(\sqsubseteq)=q_0$, and in a similar way, we obtain $q_1\circ(\sqsupseteq)=q_1$. For each $\alpha\in 2$, we have $(\alpha,*)\in q_1\circ(\sqsubseteq)$ if and only if there is some $\beta\in 2$ such that $\alpha\sqsubseteq\beta$ and $(\beta,*)\in q_1$. The latter condition forces $\beta=1$, and since $\alpha\sqsubseteq 1$ for each $\alpha\in 2$, it follows that $q_1\circ(\sqsubseteq)=\top_{2,1}$. In a similar way, we find $q_0\circ(\sqsupseteq)=\top_{2,1}$.

For the general case, we use that $`2=\Omega$, $`1=I$ and $`(\sqsubseteq)=\qo_\Omega$. By Theorem \ref{thm:qRel quote}, $`(-)$ preserves daggers and dagger biproducts, hence we have $\qop_\Omega=\qo_\Omega^\dag=(`{\sqsubseteq})^\dag=`(\sqsubseteq^\dag)=`(\sqsupseteq)$ and $`q_0=p_0$ and $`q_1=p_1$. Since $\mathbf R$ is affine, the theorem also assures that $`\top_{2,1}=\top_{\Omega,I}$. The statement now follows from functoriality of $`(-)$.
\end{proof}

\begin{proposition}\label{prop:condition downset monad}
     Let $(\mathbf R,\otimes,I)$ be an affine dagger compact quantaloid with small dagger biproducts and dagger kernels such that for each object $X$ of $\mathbf R$:
    \begin{itemize}
        \item[(1)] there is a unique $\bot$-monic effect $X\to I$;
        \item[(2)] every $\bot$-monic PER on $X$ is a equivalence relation on $X$.
        \end{itemize}
Let $2$ be the ordinary set $\{0,1\}$ ordered by $\sqsubseteq$ defined by $0\sqsubseteq 1$. Let $\Omega=`2$ with projection on the second factor denoted by $p_1$, and let $\qo_\Omega=`(\sqsubseteq)$. 
Then for each preorderd objects $(X,\qo_X)$ in $\mathbf R$, the map 
\[\mathrm{PreOrd}(\mathbf R)\big( (X,\qo_X),(\Omega,\qo_\Omega)\big)\to\mathrm{MonRel}( \mathbf R)\big((X,\qo_X),(I,\id_I)\big),\qquad f\mapsto p_1\circ f\]
is a bijection.
\end{proposition}
\begin{proof}
    Note that $\Omega$ coincides the $\Omega$ in Corollary \ref{cor:power objects condition}. By that same corollary, we have a bijection
    \[\mathrm{Maps}(\mathbf R)(X,\Omega)\to\mathbf R(X,I),\qquad f\mapsto p_1\circ f.\]
Let $f:(X,\qo_X)\to(\Omega,\qo_\Omega)$ be a monotone map. By Lemma \ref{lem:functions and the opposite order} also $f:(X,\qop_X)\to(\Omega,\qop_\Omega)$ is monotone, i.e., $f\circ \qop_X\leq \qop_\Omega\circ f$. 
Then, using Lemma \ref{lem:eval-identities}, we find $\id_I\circ p_1\circ f=p_1\circ f =p_1\circ\qop_\Omega\circ f\geq p_1\circ f\circ \qop_X$, which shows that $p_1\circ f:(X,\qo_X)\to (I,\id_I)$ is a monotone relation. If $g:(X,\qo_X)\to(\Omega,\qo_\Omega)$ is another monotone map such that $p_1\circ f=p_1\circ g$, then it follows from Corollary \ref{cor:power objects condition} that $f=g$, so the map $f\mapsto p_1\circ f$ in the statement is injective. 
We proceed with showing surjectivity. So let $v:(X,\qo_X)\to(I,\id_I)$ be a monotone relation. In particular, $v\in\mathbf R(X,I)$, hence by Corollary \ref{cor:power objects condition} there is a map $f:X\to\Omega$ such that $p_1\circ f=v$. We only need to show that $f$ is monotone. First, we show that $\top_{\Omega,I}\circ f$ is a $\bot$-mono. So let $r:Y\to X$ be a morphism in $\mathbf R$ such that $\top_{\Omega,I}\circ f\circ r=\perp_{Y,I}$. Since $\top_{\Omega,I}$ is a $\bot$-mono, we obtain $f\circ r=\perp_{Y,\Omega}$.
Then $\perp_{Y,X}=f^\dag\circ \perp_{Y,\Omega}=f^\dag\circ f\circ r\geq r$ for $f$ is a map, whence  $r=\perp_{Y,X}$. So, $\top_{\Omega,I}\circ f$ is a $\bot$-mono $X\to I$, and since by Proposition \ref{prop:PER} $\top_{X,I}$ is the unique $\bot$-monic effect $X\to I$, we must have $\top_{X,I}=\top_{\Omega,I}\circ f$.
Then, again using Lemma \ref{lem:eval-identities}, we obtain
\[ p_0\circ f\circ\qop_X \leq \top_{X,I}=\top_{\Omega,I}\circ f=p_0\circ\qop_\Omega\circ f\]
and
\[ p_1\circ f\circ \qop_X =v\circ \qop_X=v=p_1\circ f=p_1\circ \qop_\Omega\circ f.\] It now follows from (a) of Proposition \ref{prop:biproducts in quantaloids are monotone} that $f\circ\qop_X\leq \qop_\Omega\circ f$, so $f$ is monotone. 
\end{proof}

\section{Power objects}\label{sec:power objects}

Dagger quantaloids can be regarded as categorical generalizations of the category $\mathbf{Rel}$, which allows different examples than allegories - other categorical generalizations of $\mathbf{Rel}$. In the theory of allegories, the notion of power objects is very important, since there is a relation between allegories with power objects and topoi - categorical generalizations of the category $\mathbf{Set}$. The following definition is inspired by the definition of power objects in allegories.

\begin{definition}
We say that a dagger quantaloid  $\mathbf R$ has \emph{power objects} if the embedding $\mathrm{Maps}(\mathbf R)\to\mathbf R$ has a right adjoint.
\end{definition}

\subsection{Existence of power objects}
In this subsection, we explore conditions that assure the existence of power objects in dagger quantaloids. We first state a more general theorem for which neither a quantaloid structure nor daggers are necessary. We note that in the theorem below, the object $\Omega$ can be interpreted as an object of truth values, and $\omega$ can be interpreted as the dagger of a morphism that represents the element `true' in $\Omega$. The proof of the theorem is heavily inspired by the proof of \cite[Theorem 9.2]{klm-quantumposets}.

\begin{theorem}\label{thm:existence power objects}
    Let $(\mathbf R,\otimes,I)$ be a compact-closed category and let $(\mathbf S,\odot,J)$ be a symmetric monoidal closed category with internal hom $[-,-]$ and evaluation morphism $\Eval_{A,B}:[A,B]\otimes A\to B$ for objects $A,B$ of $\mathbf S$.  Let $E:\mathbf S\to\mathbf R$ be a strict monoidal functor that is bijective on objects. Assume that there is an object $\Omega\in\mathbf S$ and an $\mathbf R$-morphism $\omega:E(\Omega)\to I$ such that for each object $A\in\mathbf S$, we have a bijection
    \begin{equation}\label{eq:assumption-bijection}
\mathbf S(A,\Omega)\xrightarrow{\cong}\mathbf R(E(A),I), \qquad f\mapsto \omega\circ E(f).
\end{equation}
For each object $X\in\mathbf R$, let $P(X):= [E^{-1}(X^*),\Omega]$. Then the assignment $X\mapsto P(X)$ extends to a functor $P:\mathbf R\to\mathbf S$ that is right adjoint to $E$. 
The $X$-component of the co-unit $\ni$ of this adjunction is the unique $\mathbf R$-morphism $\ni_X:EP(X)\to X$ such that
\begin{equation}\label{eq:defni}
    \omega\circ E(\Eval_{E^{-1}(X^*),\Omega})=\coname{\ni_X}.
\end{equation}
\end{theorem}

Before we prove the theorem, we need some lemmas. The first one follows directly from the monoidal closure of $\mathbf S$:
\begin{lemma}\label{lem:eval}
For any $A$ in $\mathbf S$ and each $X$ in $\mathbf R$, we have a bijection
\[ \mathbf S(A,P(X))\xrightarrow{\cong} \mathbf S(A\otimes E^{-1}(X^*),\Omega), \qquad f\mapsto\Eval_{E^{-1}(X^*),\Omega}\circ (f\otimes \id_{E^{-1}(X^*)}).\]
\end{lemma}
We construct the counit of the theorem in the second lemma.
\begin{lemma}\label{lem:ni}
    For each $X$ in $\mathbf R$ there is a unique morphism $\ni_X:EP(X)\to X$ in $\mathbf R$ such that (\ref{eq:defni}) holds. 
\end{lemma}
\begin{proof}
    Since $\Eval_{E^{-1}(X^*),\Omega}$ is a $\mathbf S$-morphism $P(X)\otimes E^{-1}(X^*)\to\Omega$, it follows from the assumption (\ref{eq:assumption-bijection}) that $\omega\circ E(\Eval_{E^{-1}(X^*),\Omega})$ is an $\mathbf R$-morphism $E(P(X)\otimes E^{-1}(X^*))\to I$. Since $E$ is strict monoidal, we have that $\omega\circ J(\Eval_{E^{-1}(X^*),\Omega})$ is an $\mathbf R$-morphism $EP(X)\otimes X^*\to I$. The existence of $\ni_X$ such that (\ref{eq:defni}) holds follows now from Lemma \ref{lem:epsilon}.
\end{proof}

\begin{proof}[of Theorem \ref{thm:existence power objects}]
    Let $A$ be an object of $\mathbf S$ and let $X$ an object of $\mathbf S$. We need to show that for  each $\mathbf R$-morphism $v:E(A)\to X$ there is a unique $\mathbf S$-morphism $f_v:A\to P(X)$ such that the following diagram commutes:
    \[\begin{tikzcd}
E(A)\ar{rd}{v}\ar{d}[swap]{E(f_v)} & \\
EP(X)\ar{r}[swap]{\ni_X} & X.
\end{tikzcd}\]
    We define $f_v$ in steps. Since $v\in\mathbf R(E(A),X)$, it follows from Lemma \ref{lem:epsilon} that $\coname{v}\in\mathbf R(E(A)\otimes X^*,I)$. Since $E$ is strict monoidal and bijective on objects, we have $\coname{v}\in \mathbf R(E(A\otimes E^{-1}(X^*)),I)$. Hence, by the assumption (\ref{eq:assumption-bijection}), there is a unique $\mathbf S$-morphism $k_v\in \mathbf S(A\otimes E^{-1}(X^*),\Omega)$ such that 
    \begin{equation}\label{eq:kv}
    \omega\circ E(k_v)= \coname{v}.
    \end{equation}
    Now, by Lemma \ref{lem:eval}, there is a unique $f_v\in\mathbf S(A,P(X))$ such that
    \begin{equation}\label{eq:fv}
    k_v=\Eval_{E^{-1}(X^*),\Omega}\circ (f_v\otimes \id_{E^{-1}(X^*)}).
    \end{equation}
    We check that the diagram in the statement commutes.
We have
\begin{align*} 
\coname{\ni_X\circ E(f_v))} & = \epsilon_X\circ( (\ni_X\circ E(f_v))\otimes \id_{X^*})\\
& =
\epsilon_X\circ(\ni_X\otimes \id_{X^*})\circ(E(f_v)\otimes \id_{X^*})\\
& = \coname{\ni_X}\circ (E(f_v)\otimes\id_{X^*})\\
& = \omega\circ E(\Eval_{E^{-1}(X^*),\Omega})\circ (E(f_v)\otimes \id_{X^*})\\
& = \omega\circ E\big(\Eval_{E^{-1}(X^*),\Omega}\circ (f_v\otimes \id_{E^{-1}(X^*)})\big)\\
& = \omega\circ E(k_v)\\
&= \coname{v},
\end{align*}
where we used (\ref{eq:defni}) proven in Lemma \ref{lem:ni} in the fourth equality, functoriality of $E$ and the fact that $E$ is strict monoidal in the fifth equality (note that $E(\id_{E^{-1}(X^*)})=\id_{EE^{-1}(X^*)}=\id_{X^*}$), the definition of $f_v$, i.e., equation (\ref{eq:fv}), in the penultimate equality, and the definition of $k_v$, i.e., equation (\ref{eq:kv}) in the last equality. It now follows from Lemma \ref{lem:epsilon} that $\ni_X\circ E(f_v)=v$, i.e., the diagram commutes. Next, we check that $f_v$ is the unique $\mathbf S$-morphism for which the diagram commutes. So assume that $g:A\to P(X)$ is a $\mathbf S$-morphism such that $\ni_X\circ E(g)=v$. Then $\ni_X\circ E(g)=\ni_X\circ E(f_v)$, hence     
\begin{align*}
    \omega\circ E(\Eval_{E^{-1}(X^*),\Omega}\circ (g\otimes \id_{E^{-1}(X^*)}) & = \omega\circ E(\Eval_{E^{-1}(X^*),\Omega})\circ (E(g)\otimes \id_{X^*}) \\
    & = \coname{\ni}\circ (E(g)\otimes\id_{X^*})\\
    & = \epsilon_X\circ (\ni_X\otimes \id_{X^*})\circ (E(g)\otimes \id_{X^*})\\
    & = \epsilon_X\circ ((\ni_X\circ E(g))\otimes \id_{X^*})\\
      & = \epsilon_X\circ ((\ni_X\circ E(f_v))\otimes \id_{X^*})\\
          & = \epsilon_X\circ (\ni_X\otimes \id_{X^*})\circ (E(f_v)\otimes \id_{X^*})\\
          & = \coname{\ni_X}\circ (E(f_v)\otimes\id_{X^*})\\
          & = \omega\circ E(\Eval_{E^{-1}(X),\Omega})\circ (E(f_v)\otimes \id_{X^*}) \\
            & = \omega\circ E(\Eval_{E^{-1}(X^*),\Omega}\circ (f_v\otimes
            \id_{E^{-1}(X^*)})),
\end{align*}
where we used functoriality of $E$ and the fact that $E$ is strict monoidal in the first and last equalities, whereas we used Lemma \ref{lem:ni} in the second and penultimate equalities. It now follows from the assumption (\ref{eq:assumption-bijection}) that
\[ \Eval_{E^{-1}(X^*),\Omega}\circ (g\otimes \id_{E^{-1}(X^*)})=\Eval_{E^{-1}(X^*),\Omega}\circ (f_v\otimes \id_{E^{-1}(X^*)}).\]
We can now apply Lemma \ref{lem:eval} to conclude that $g=f_v$.
\end{proof}

\begin{corollary}\label{cor:power objects}
    Let $(\mathbf R,\otimes,I)$ be an affine dagger compact quantaloid with small dagger biproducts and dagger kernels such that for each object $X$ of $\mathbf R$:
    \begin{itemize}
        \item[(1)] there is a $\bot$-monic effect $X\to I$;
        \item[(2)] every $\bot$-monic PER on $X$ is an equivalence relation on $X$.
    \end{itemize}
    If $\mathbf S=\mathrm{Maps}(\mathbf R)$ is symmetric monoidal closed, then the embedding $E:\mathbf S\to\mathbf R$ has a right adjoint $P$.

    More precisely, if $\Omega=I\oplus I$, and $\omega:\Omega\to I$ be the projection of $\Omega$ onto the second factor, and if $[-,-]$ and $\Eval$ denote the internal hom and the evaluation of $\mathbf S$, respectively, then $P$ is defined on objects $X$ of $\mathbf S$  by $P(X)=[X^*,\Omega]$. The $X$-component of the counit $\ni$ is the unique morphism $\ni_X:P(X)\to X$ satisfying $\omega\circ\Eval_{X^*,\Omega}=\coname{\ni_X}$.
\end{corollary}
\begin{proof}
    This follows from combining Corollary \ref{cor:power objects condition} and Theorem \ref{thm:existence power objects}, taking $E$ to be the inclusion and $\omega=p_1$. 
\end{proof}

\begin{corollary}\label{cor:lower set functor}
    Let $(\mathbf R,\otimes,I)$ be an affine dagger compact quantaloid with small dagger biproducts and dagger kernels such that for each object $X$ of $\mathbf R$:
    \begin{itemize}
        \item[(1)] $\top_{X,I}$ is a $\bot$-monic effect;
        \item[(2)] every $\bot$-monic PER on $X$ is an equivalence relation on $X$.
    \end{itemize}
    If $\mathrm{PreOrd}(\mathbf R)$ is symmetric monoidal closed, then the functor $(-)_\diamond:\mathrm{PreOrd}(\mathbf R)\to\mathrm{MonRel}(\mathbf R)$ has a right adjoint $D$.

    More precisely, let $(\Omega,\qo_\Omega)=`(2,\sqsubseteq)$, where the order $\sqsubseteq$ on $2=\{0,1\}$ is determined by $0\sqsubseteq 1$. Let $\omega:\Omega\to I$ be the projection of $\Omega$ onto the second factor. Let $[-,-]$ and $\Eval$ denote the internal hom and the evaluation of $\mathrm{PreOrd}(\mathbf R)$. Then $D$ is defined on objects $(X,\qo_X)$ by $D(X,\qo_X)=[(X,\qo_X)^*,(\Omega,\qo_\Omega)]$. The $(X,\qo_X)$-component of the counit $\ni$ is the unique morphism $\ni_{(X,\qo_X)}:D(X,\qo_X)\to (X,\qo_X)$ satisfying $\omega\circ\Eval_{(X,\qo_X)^*,(\Omega,\qo_\Omega)}=\coname{\ni_{(X,
    \qo_X)}}$.
\end{corollary}
\begin{proof}
 For each monotone map $f:(X,\qo_X)\to(\Omega,\qo_\Omega)$, we have $\omega\circ E(f)=p_1\circ f_\diamond=p_1\circ \qop_\Omega\circ f=p_1\circ f$, where the last equality follows from Lemma \ref{lem:eval-identities}. Then the statement follows directly from Proposition \ref{prop:condition downset monad} and Theorem \ref{thm:existence power objects}, where we take $E=(-)_\diamond$ and $\omega=p_1$.
\end{proof}

We provide some examples of adjunctions obtained via Theorem \ref{thm:existence power objects} or one of its corollaries. The first example is a direct application of the theorem.

\begin{example}
    Let $V$ be a nontrivial commutative quantale. Let $\mathbf R:=V$-$\mathbf{Rel}$, which is dagger compact (cf. Example \ref{ex:V-Rel is dagger compact quantaloid}). Furthermore, we take $\mathbf{S}:=\mathbf{Set}$, which is cartesian closed. 
    Let $E:\mathbf{S}\to\mathbf{R}$ be the functor $`(-):\mathbf{Rel}\to\mathbf R$ of Theorem \ref{thm:qRel quote} restricted to $\mathbf{Set}$. The dagger biproducts of $V$-$\mathbf{Rel}$ are described in Example \ref{ex:biproducts in V-Rel}, and under the identification $\coprod_{\alpha\in A}1=A$ for any set $A$, it follows that $E$ is the identity on sets, and sends a function $f:A\to B$ to the $V$-relation $Ef:A\sto B$ given by \[(Ef)(\alpha,\beta)=\begin{cases} e, & f(\alpha)=\beta,\\
    \perp, & \text{otherwise}.
    \end{cases}\]
    Then $E$ has a right adjoint that sends every set $X$ to its $V$\emph{-valued powerset} $V^X$.

    This follows from Theorem \ref{thm:existence power objects} by taking $\Omega=V$, and by choosing $\omega:V\sto 1$ to be the function $V\times 1\to V$, $(v,*)\mapsto v$. 
We only need to show that $\mathbf{Set}(X,V)\to V$-$\mathbf{Rel}(X,1)$, $f\mapsto \omega\bullet Ef$ is a bijection.
Indeed,  for each set $X$, each function $f:X\to V$, and each $x\in X$, we have
    \[ (\omega\bullet Ef)(x,*)=\bigvee_{v\in V}\omega(v,*)\cdot Ef (x,v)=\omega(f(x),*)\cdot Ef (x,f(x))=f(x)\cdot e=f(x).\]
As a consequence, if $f,g:X\to V$ are distinct functions, then $f(x)\neq g(x)$ for some $x\in X$, hence $(\omega\bullet Ef)(x,*)=f(x)\neq g(x)=(\omega\bullet Eg)(x,*)$, showing that $\omega\bullet Ef\neq\omega\bullet Eg$, i.e., $f\mapsto\omega\bullet Ef$ is injective. For surjectivity, let $r:X\sto 1$ be a $V$-relation, so a function $X\times 1\to V$. Let $f:X\to V$ be the function $x\mapsto r(x,*)$. Then for each $(x,*)\in X\times 1$, we have $(\omega\bullet Ef)(x,*)=f(x)=r(x,*)$, so $r=\omega\bullet Ef$, hence we indeed have a bijection.
\end{example}

As the special case $V=2$ of the previous example, we obtain the ordinary power set functor. We can also obtain this functor by applying one of corollaries. 

\begin{example}\label{ex:power set}
By Theorem \ref{thm:Rel-properties},  we can apply Corollary \ref{cor:power objects} to conclude that the embedding $E:\mathbf{Set}\to\mathbf{Rel}$ has a right adjoint $P$, the covariant power set functor. 
\end{example}
In an almost similar way, we can derive the existence of a quantum power set functor. 
\begin{example}\label{ex:quantum power set}
Let $\mathbf{R}=\mathbf{qRel}$ and $\mathbf S=\mathrm{Maps}(\mathbf R)=\mathbf{qSet}$. The latter category is symmetric monoidal closed \cite[Theorem 9.1]{Kornell18}. By the properties of $\qRel$ stated in Theorem \ref{thm:qRel-properties}, we can apply Corollary \ref{cor:power objects} to conclude that the embedding $\E:\mathbf{qSet}\to\mathbf{qRel}$ has a right adjoint $\P$, which we call the \emph{quantum power set functor}.
\end{example}

\begin{example}
If $\mathbf R=\mathbf{Rel}$, then $\mathrm{PreOrd}(\mathbf R)=\mathbf{PreOrd}$ and $\mathrm{MonRel}(\mathbf R)=\mathbf{MonRel}$. It is well known that $\mathbf{PreOrd}$ is cartesian closed. The properties of $\mathbf{Rel}$ in Theorem \ref{thm:Rel-properties} allow us to apply Corollary \ref{cor:lower set functor}, assuring the existence of a right adjoint $D$ to the functor $(-)_\diamond:\mathbf{PreOrd}\to\mathbf{MonRel}$, which is the lower set functor.
\end{example}

\begin{example}
If $\mathbf R=\mathbf{qRel}$, then $\mathrm{PreOrd}(\mathbf R)=\mathbf{qPreOrd}$ and $\mathrm{MonRel}(\mathbf R)=\mathbf{qMonRel}$. In \cite[Theorem 8.3]{klm-quantumposets}, it was shown that the related category $\mathbf{qPOS}$ of quantum posets is symmetric monoidal closed. The proof of this theorem can be simplified to obtain a proof of the symmetric monoidal closure of $\mathbf{qPreOrd}$. The properties of $\mathbf{qRel}$ stated in Theorem \ref{thm:qRel-properties} allow us to apply Corollary \ref{cor:lower set functor}, assuring that the functor $(-)_\diamond:\mathbf{qPreOrd}\to\mathbf{qMonRel}$ has a right adjoint $\D$, which we call the \emph{quantum lower set functor}. 
\end{example}

\subsection{The reconstruction of internal homsets}
In \cite{Kornell18}, Kornell showed that $\mathbf{qSet}=\mathrm{Maps}(\mathbf{qRel})$ satisfies properties that strongly resemble the axioms of an elementary topos. 

Let $\mathbf R$ be a dagger compact quantaloid and let $\mathbf S=\mathrm{Maps}(\mathbf S)$. In the previous section, we explored conditions that assure that the embedding $\mathbf S\to \mathbf R$ has a right adjoint, which relied on the assumption that $\mathbf S$ is symmetric monoidal closed. In this section, assuming some additional mild conditions, we prove the converse, namely that $\mathbf S$ is symmetric monoidal closed provided that the embedding $\mathbf S\to\mathbf R$ has a right adjoint $P$. 

We first make the following definition:
\begin{definition}
    Let $(\mathbf S,\otimes,I)$ be a semicartesian category. For any two objects $X,Y\in\mathbf S$, let  $p^1_{X\otimes Y}:X\otimes Y\to X$ and $p^2_{X\otimes Y}:X\otimes Y\to Y$ be the canonical projections  given by $p^1_{X\otimes Y}:=\rho_X\circ (\id_X\otimes !_Y)$ and $p^2_{X\otimes Y}:=\lambda_Y\circ(!_X\otimes \id_Y)$. Then we call a morphism $f:X\to Y$ \emph{classical} if there is a morphism $\hat f:X\to X\otimes Y$ such that $f=p^2_{X\otimes Y}\circ \hat f$ and $\id_X=p_{X\otimes Y}^1\circ \hat f$.
\end{definition}

\begin{lemma}
    Let $(\mathbf S,\otimes,I)$ be a semicartesian category. Then a morphism $f:X\to Y$ is classical if and only if for each morphism $g:X\to Z$ there is a morphism $h:X\to Z\otimes Y$ with $p_{Z\otimes Y}^1\circ h=g$ and $p_{Z\otimes Y}^2\circ h=f$.
\end{lemma}
\begin{proof}
    Assume that for each $g:X\to Z$, there is a morphism $h:X\to Z\otimes Y$ with $p_X\circ h=g$. Take, $g=\id_X:X\to X$ yields $h:X\to X\otimes Y$ with $p_X\circ h=\id_X$ and $p_Y\circ h=f$, so $f$ is classical with $\hat f=h$. Conversely, assume that $f$ is classical, and let $g:X\to Z$ be a morphism. Let $h=(g\otimes\id_Y)\circ \hat f$. Then, $p_{Z\otimes Y}^1\circ h=p_{Z\otimes Y}^1\circ (g\otimes\id_Y)\circ\hat f=\rho_Z\circ(\id_Z\otimes !_Y)\circ (g\otimes\id_Y)\circ\hat f=\rho_Z\circ (g\otimes\id_I)\circ(\id_X\otimes !_Y)\circ\hat f=g\circ \rho_X\circ(\id_X\otimes !_Y)\circ\hat f=g\circ p_{X\otimes Y}^1\circ\hat f=g$, and $p_{Z\otimes Y}^2\circ h=p_{Z\otimes Y}^2\circ (g\otimes\id_Y)\circ\hat f=\lambda_Y\circ (!_Z\otimes\id_Y)\circ (g\otimes\id_Y)\circ \hat f=\lambda_Y\circ (!_X\otimes\id_Y)\circ\hat f=p_{X\otimes Y}^2\circ \hat f=f$.
\end{proof}

Let $F:\X\to\Y$ be a map between quantum sets. It was shown in \cite[Lemma 10.7 \& Proposition 10.8]{Kornell18} that there are canonical maps $Q:\X\to`\At(\X)$ and $J:`\qSet(\mathbf 1,\Y)\to\Y$ such that that $F$ is classical in $\mathbf{qSet}$ if and only there is an ordinary function $f:\At(\X)\to\qSet(\mathbf 1,\Y)$ such that $F=J\circ`f\circ Q$. This motivates the terminology `classical map'.

We will now state the main theorem of this section. For the remainder of this section, we will assume that the conditions in the theorem hold. Note that by Lemma \ref{lem:Maps form monoidal subcategory} $\mathbf S$ is a monoidal subcategory of $\mathbf R$ 

\begin{theorem}\label{thm:inner hom S}
Let $(\mathbf R,\otimes,I)$ be a dagger compact quantaloid, and let $\mathbf S=\mathrm{Maps}(\mathbf R)$. If
\begin{itemize}
\item[(1)] $\mathbf S$ is semicartesian when regarded as a monoidal subcategory of $\mathbf R$ (cf. Lemma \ref{lem:Maps form monoidal subcategory});
    \item[(2)] The embedding $J:\mathbf S\to\mathbf R$ has a right adjoint $P$ with unit $\{\cdot\}$ and counit $\ni$;
    \item[(3)] There exists an object $\Omega$ of $\mathbf S$ and an $\mathbf S$-morphism $\mathrm{true}:I\to \Omega$ such that $\mathbf S(X,\Omega)\to\mathbf R(X,I)$, $f\mapsto \mathrm{true}^\dag\circ f$ is a bijection;
    \item[(4)] $\mathbf S$ has pullbacks;
    \item[(5)] For each object $X$ and each subobject $m:A\to X$ in $\mathbf S$ there is a unique classical morphism $\chi_A:X\to\Omega$ such that the diagram below is a pullback square in $\mathbf S$.
        \[\begin{tikzcd}
A\ar{r}{!_A}\ar{d}[swap]{m} & I\ar{d}{\mathrm{true}}\\
X\ar{r}[swap]{\chi_A} & \Omega.
\end{tikzcd}\]
Then $\mathbf S$ is symmetric monoidal closed.
    \end{itemize}
\end{theorem}

Our proof is essentially the proof that power objects in a topos imply the existence of exponential objects, see for instance Section IV.2 of \cite{MacLane-Moerdijk}, which we followed quite closely. We first need some lemmas.

\begin{lemma}
    Let $X\in\mathbf S$. Then there exists a unique classical morphism $\sigma_X:P(X)\to\Omega$ such that
the following diagram is a pullback square:
       \[\begin{tikzcd}
X\ar{r}{!_X}\ar{d}[swap]{\{\cdot\}_X} & I\ar{d}{\mathrm{true}}\\
P(X)\ar{r}[swap]{\sigma_X} & \Omega.
\end{tikzcd}\]
\end{lemma}
\begin{proof}
   Since $\mathbf S$ is a subcategory of $\mathbf R$, the functor $J$ is faithful, hence the unit $\{\cdot\}$ of the adjunction $J\dashv P$ is a monomorphism (which follows from the dual of \cite{maclane}*{Theorem IV.3.1}.) Consequently, we can apply property (5) of Theorem \ref{thm:inner hom S}, which yields the statement.
\end{proof}

\begin{lemma}\label{lem:bar ni}
    Let $X\in\mathbf S$. Then there is a unique morphism $\bar\ni_X:P(X)\otimes X^*\to\Omega$ such that $\mathrm{true}^\dag\circ\bar\ni_X=\coname{\ni_X}$. 
\end{lemma}
\begin{proof}
Since $\mathbf R$ is dagger compact, we can take the coname $\coname{\ni_X}:P(X)\otimes X^*\to I$ of $\ni_X:P(X)\to X$. Then, by the bijection $\mathbf S(X,\Omega)\to \mathbf R(X,I)$, $f\mapsto \mathrm{true}^\dag\circ f$ there is a unique morphism $\bar\ni_X:P(X)\otimes X^*\to\Omega$ in $\mathbf S$ such that $\mathrm{true}^\dag\circ\bar\ni_X=\coname{\ni_X}$.    
\end{proof}

\begin{lemma}\label{lem:P-transpose}
    Let $f:X\otimes Y^*\to\Omega$ be a morphism in $\mathbf S$. Then there exists a unique morphism $\hat f:X\to P(Y)$ in $\mathbf S$ such that $\bar\ni_Y\circ(\hat f\otimes\id_{Y^*})=f$.
\end{lemma}
\begin{proof}
By the third assumption in Theorem \ref{thm:inner hom S}, we have a bijection \[\mathbf S(X\otimes Y^*,\Omega)\to\mathbf R(X\otimes Y^*,I),\qquad f\mapsto\true^\dag\circ f.\]
By Lemma \ref{lem:epsilon}, we also have a bijection \[\mathbf R(X,Y)\to\mathbf R(X\otimes Y^*,I),\qquad g\mapsto\coname g.\] By the second assumption of Theorem \ref{thm:inner hom S}, we have a bijection \[\mathbf S(X,P(Y))\to\mathbf R(X,Y),\qquad h\mapsto \ni_Y\circ h.\] So any morphism $f:X\otimes Y^*\to\Omega$ in $\mathbf S$ corresponds to a unique morphism $\true^\dag\circ f:X\otimes Y^*\to I$ in $\mathbf R$, for which there is a unique morphism $g:X\to Y$ in $\mathbf R$ such that $\true^\dag\circ f=\coname g$. By the last bijection, there is a unique $\hat f:X\to P(Y)$ in $\mathbf S$ such that $\ni_Y\circ \hat f=g$. Hence, $\hat f$ is the unique morphism in $\mathbf S$ such that $\true^ \dag\circ f=\coname {\ni_Y\circ\hat f}$.
We have $\coname{\ni_Y\circ\hat f}=\epsilon_Y\circ ((\ni_Y\circ \hat f)\otimes\id_{Y^*}=\epsilon_Y\circ(\ni_Y\otimes\id_{Y^*})\circ (\hat f\otimes\id_{Y^*})=\coname{\ni_Y}\circ(\hat f\otimes\id_{Y^*})=\true^\dag\circ\bar\ni_Y\circ (\hat f\otimes\id_{Y^*})$, where we used Lemma \ref{lem:bar ni} in the last equality. Thus we obtain $\true^\dag\circ f=\true^\dag\circ\bar\ni_Y\circ (\hat f\otimes\id_{Y^*})$, hence the statement follows from the bijection $f\mapsto\true^\dag \circ f$.
\end{proof}

\begin{proof}[of Theorem \ref{thm:inner hom S}]
Following \cite{MacLane-Moerdijk}, we assume that the associativity isomorphisms are identities to simplify the notation.
Consider objects $X$ and $Y$ of $\mathbf S$. In order to construct an object $Y^X$ that will be the inner hom of $\mathbf S$, we apply Lemma \ref{lem:P-transpose} to define:
\begin{align*}
    v : = \widehat{\bar\ni_{X^*\otimes Y}},\qquad
    u : = \widehat{\sigma_Y\circ v},\qquad
    k  := \widehat{\mathrm{true}\circ !_{I\otimes X}}.
\end{align*}
That is, $v:P(X^*\otimes Y)\to P(Y)$, $u:P(X^*\otimes Y)\to P(X^*)$, and $k:I\to P(X^*)$ are the respective unique morphisms in $\mathbf S$ such that 
\begin{align*}
   \bar\ni_Y\circ(v\otimes\id_{Y^*} )& = \bar\ni_{X^*\otimes Y};\\
   \bar\ni_{X^*}\circ(u\otimes \id_{X}) & = \sigma_Y\circ v;
\\
\bar\ni_{X^*}\circ (k\otimes \id_{X}) & = \true\circ !_{I\otimes X}.
\end{align*}

We now define $Y^X$ as the pullback of $u$ and $k$, so the following diagram is a pullback square:
 \begin{equation}\label{eq:def-YX}\begin{tikzcd}
Y^X\ar{r}{!_{Y^X}}\ar{d}[swap]{m} & I\ar{d}{k}\\
P(X^*\otimes Y)\ar{r}[swap]{u} & P(X^*).
\end{tikzcd}\end{equation}
In order to construct the evaluation map $e:Y^X\otimes X\to Y$, we need to prove that 
\begin{equation}\label{eq:equality for pullback}
\true\circ !_{Y^X\otimes X}=\sigma_Y\circ v\circ (m\otimes\id_X).    
\end{equation}

Consider the following diagram:

 \[\begin{tikzcd}
Y^X\otimes X\ar{r}{m\otimes\id_X}\ar{dd}[swap]{!_{Y^X}\otimes \id_X} & P(X^*\otimes Y)\otimes X\ar{r}{v}\ar{d}{u\otimes\id_X} & P(Y)\ar{d}{\sigma_Y} \\
& P(X^*)\otimes X\ar{r}{\bar\ni_{X^*}} & \Omega    \\
I\otimes X\ar{rr}[swap]{!_{I\otimes X}}\ar{ru}{k\otimes\id_X} && I\ar{u}[swap]{\true}
\end{tikzcd}\]
Here, the left square commutes, since it diagram (\ref{eq:def-YX}) tensored with $X$. The upper right square commutes by definition of $u$, and the lower right diagram commutes by definition of $k$. Since $!_{I\otimes X}\circ (!_{Y^ X\otimes X}\otimes\id_X)=!_{Y^ X\otimes X}$, it follows that (\ref{eq:equality for pullback}) indeed holds. Thus, we have the following diagram:

 \[\begin{tikzcd}
Y^X\otimes X\arrow[bend right=30,swap]{ddr}{v\circ(m\otimes\id_X)}\arrow[bend left=30]{rrd}{!_{Y^X\otimes X}}\arrow[dotted]{dr}{e} &&\\
&Y\ar{d}[swap]{\{\cdot\}_Y}\ar{r}{!_Y} & I\ar{d}{\true}\\
&P(Y)\ar{r}[swap]{\sigma_Y} &\Omega.
\end{tikzcd}\]
and since the square is a pullback square in $\mathbf S$, there must be a unique morphism $e:Y^ X\otimes X\to Y$ such that the diagram commutes. 

Next, for another object $Z$ of $\mathbf S$ and a morphism $f:Z\otimes X\to Y$ in $\mathbf S$, we claim that there is a unique morphism $g:Z\to Y^X$ in $\mathbf S$ such that $e\circ (g\otimes\id_X)=f$. To construct $g$, we first consider the morphism
\[Z\otimes X\otimes Y^*\xrightarrow{f\otimes \id_{Y^*}}Y\otimes Y^*\xrightarrow{\{\cdot\}_Y\otimes\id_{Y^*}}P(Y)\otimes Y^*\xrightarrow{\bar\ni_\Y}\Omega,\]
which, by Lemma \ref{lem:P-transpose}, equals
\[Z\otimes X\otimes Y^*\xrightarrow{h\otimes\id_X\otimes\id_{Y^*}}P(X^*\otimes Y)\otimes X\otimes Y^*\xrightarrow{\bar\ni_{X^*\otimes Y}}\Omega\]
for some unique morphism $h:Z\to P(X^*\otimes Y)$ in $\mathbf S$. Thus we obtain:
\[\bar\ni_Y\circ (\{\cdot\}_Y\otimes\id_{Y^*})\circ (f\otimes\id_{Y^*})=\bar\ni_{X^*\otimes Y}\circ (h\otimes\id_X\otimes\id_{Y^*}).\]
By definition of $v$, we obtain
\[\bar\ni_Y\circ (\{\cdot\}_Y\otimes\id_{Y^*})\circ (f\otimes\id_{Y^*})=\bar\ni_{Y}\circ(v\otimes\id_{Y^*})\circ(h\otimes\id_X\otimes\id_{Y^*}),\]
whence, using Lemma \ref{lem:P-transpose},
\begin{equation}\label{eq:equation for curry}
    \{\cdot\}_Y\circ f=v\circ (h\otimes\id_{X}).
\end{equation}
Now, we obtain
\begin{align*}
    \bar\ni_{X^*}\circ (k\otimes\id_{X})\circ (!_Z\otimes\id_X) & = \true\circ !_{I\otimes X}\circ  (!_Z\otimes\id_X)\\
    & = \true\circ !_{Z\otimes X}=\true\circ !_Y\circ f\\
    & = \sigma_Y\circ \{\cdot\}_Y\circ f=\sigma_Y\circ v\circ (h\otimes\id_{X})\\
    & = \bar\ni_{X^*}\circ (u\otimes\id_X)\circ (h\otimes\id_X),
    \end{align*}
    where the first equality follows from the definition of $k$, the fourth equality from the definition of $\sigma_Y$, the fifth equality from (\ref{eq:equation for curry}), and the last equality from the definition of $u$. Lemma \ref{lem:P-transpose} now yields $k\circ !_Z=u\circ h$. Note that automatically we have $!_{Y^ X}\circ g=!_Z$, so by definition of $Y^X$ as the pullback of $u$ and $k$, it follows that there is a unique morphism $g:Z\to Y^X$ in $\mathbf S$ such that the following diagram commutes:
 \[\begin{tikzcd}
Z\arrow[bend right=30,swap]{ddr}{h}\arrow[bend left=30]{rrd}{!_Z}\arrow[dotted]{dr}{g} && \\
&Y^X\ar{d}[swap]{m}\ar{r}{!_{Y^X}} & I\ar{d}{k}\\
&P(X^*\otimes Y)\ar{r}[swap]{u} &P(X^*).
\end{tikzcd}\]
We verify that $e\circ (g\otimes\id_X)=f$.
\begin{align*}
    e\circ (g\otimes\id_X) & = (\ni_Y)\circ\{\cdot\}_Y\circ e\circ (g\otimes\id_X) = (\ni_Y)\circ v \circ (m\otimes \id_X)\circ(g\otimes\id_X)\\
    & = (\ni_Y)\circ v\circ(h\otimes\id_X)= (\ni_Y)\circ \{\cdot\}_Y\circ f = f,
\end{align*}
where in the first and last equalities we used the triangle identities of the adjunction $J\dashv P$, while the second equality follows by definition of $e$, the third equality follows by definition of $g$, and the penultimate equality follows from equality (\ref{eq:equation for curry}).
Finally, assume that $g':Z\to Y^X$ is another morphism in $\mathbf S$ such that $e\circ(g' \otimes\id_X)=f$. Then:
\begin{align*}
    \bar\ni_{X^*\otimes Y}\circ (m\otimes\id_X\otimes\id_{Y^*})\circ(g\otimes\id_X\otimes\id_{Y^*}) & = \bar\ni_Y\circ v\circ (m\otimes\id_X\otimes\id_{Y^*})\circ (g\otimes\id_X\otimes\id_{Y^*})\\
    & = \bar\ni_Y\circ (\{\cdot\}_Y\otimes \id_{Y^*})\circ (e\otimes\id_{Y^*})\circ (g\otimes\id_X\otimes\id_{Y^*})\\
    & = \bar\ni_Y\circ (\{\cdot\}_Y\otimes\id_{Y^*})\circ (f\otimes\id_{Y^*})\\
    & = \bar\ni_Y\circ (\{\cdot\}_Y\otimes \id_{Y^*})\circ (e\otimes\id_{Y^*})\circ (g' \otimes\id_X\otimes\id_{Y^*})\\
    & = \bar\ni_Y\circ v\circ (m\otimes\id_X\otimes\id_{Y^*})\circ (g' 
 \otimes\id_X\otimes\id_{Y^*})\\
    & =     \bar\ni_{X^*\otimes Y}\circ (m\otimes\id_X\otimes\id_{Y^*})\circ(g'\otimes\id_X\otimes\id_{Y^*})
\end{align*}
where we used the definition of $v$ in the first and last equalities, and the definition of $e$ in the second and penultimate equalities. The remaining equalities hold by the assumed properties of $g$ and $g'$. Using Lemma \ref{lem:P-transpose}, we obtain $m\circ g=m\circ g'$. Since we have  $!_{Y^ X}\circ g' =!_Z$, and by definition $g$ is the unique morphism in $\mathbf S$ such that $m\circ g=h$ and $!_{Y^X}\circ g=!_Z$, it follows that $g' =g$. 
\end{proof}

\section{Conclusions and future work}
We introduced symmetric monoidal quantaloids as a categorical structure that equips quantaloids with a symmetric monoidal structure, and that generalizes the category $\mathbf{Rel}$. Our prime example is the category $\mathbf{qRel}$ of quantum sets and binary relations; our main motivation is the internalization of mathematical structures in this category, which corresponds to the quantization of these structures. We showed that symmetric monoidal quantaloids form a framework in which one can internalize functions and partially ordered structures. For dagger symmetric monoidal quantaloids $\mathbf Q$, there are still other connections to be investigated  such as limits and subobjects in $\mathrm{Maps}(\mathbf{Q})$. It might be that these concepts are best investigated in a $2$-dimensional setting by combining $\mathbf Q$ and $\mathrm{Maps}(\mathbf{Q})$ in a double category.  Furthermore, we note that that $\mathbf{qSet}$ has properties resembling the axioms of topoi \cite{Kornell18}. $\mathbf{qSet}$ is the noncommutative generalization of the prime example of a topos, namely the category $\mathbf{Set}$ of sets and functions. This suggests the existence of notions of \emph{quantum allegories} and \emph{quantum topoi} with $\mathbf{qRel}$ and $\mathbf{qSet}$ as prime examples, respectively. The connection between power objects in $\mathbf{Q}$ and the monoidal closure of $\mathrm{Maps}(\mathbf Q)$ in the last section also points towards a notion that generalizes power allegories. We hope that this work eventually contributes to finding these notions and generalizations.

\section*{Acknowledgements}
We thank the anonymous reviewer, because their useful comments helped us to improve the final version of this article significantly. 
We thank Chris Heunen and Andre Kornell for their help with several questions that arose while writing this article. Furthermore, we are indebted to Dominik Lachman, Anna Jen\v{c}ová, Peter Sarkoci and Miloslav Štěpán, who provided us with ample comments. We also thank Isar Stubbe for various discussions that had a profound impact on this work. Gejza Jen\v{c}a is supported by  VEGA-2/0128/24, VEGA-1/0036/23 and APVV-20-0069. Bert Lindenhovius was partly supported by VEGA 2/0128/24 and APVV-22-0570, and is currently supported by the Austrian Science Fund (FWF) under Project DOI 10.55776/PAT6443523.


\begin{bibdiv}
\begin{biblist}

\bib{AndresMartinez-Heunen}{article}{
      author={Andrés-Martínez, P.},
      author={Heunen, C.},
       title={Categories of sets with infinite addition},
        date={2025},
        ISSN={0022-4049},
     journal={Journal of Pure and Applied Algebra},
      volume={229},
      number={2},
       pages={107872},
         url={https://www.sciencedirect.com/science/article/pii/S0022404925000118},
}



\bib{borceux:handbook2}{book}{
      author={Borceux, F.},
       title={Handbook of categorical algebra 2: Categories and structures},
   publisher={Cambridge University Press},
        date={1994},
}

\bib{BorceuxBossche1986}{article}{
  author    = {F. Borceux and G. Van den Bossche},
  title     = {Quantales and Their Sheaves},
  journal   = {Order},
  volume    = {3},
  date      = {1986},
  pages     = {61--87},
}

\bib{heunenjacobs}{article}{
      author={C.~Heunen, B.~Jacobs},
       title={Quantum logic in dagger kernel categories},
        date={2011},
     journal={Electronic Notes in Theoretical Computer Science},
      volume={270},
       pages={79\ndash 103},
}

\bib{BiCatRel}{article}{
      author={Carboni, A.},
      author={Walters, R.F.C.},
       title={{Cartesian bicategories I}},
        date={1987},
        ISSN={0022-4049},
     journal={Journal of Pure and Applied Algebra},
      volume={49},
      number={1},
       pages={11\ndash 32},
         url={https://www.sciencedirect.com/science/article/pii/0022404987901216},
}

\bib{connes:ncg}{book}{
      author={Connes, A.},
       title={Noncommutative geometry},
   publisher={Academic Press},
        date={1994},
}


\bib{eklund2018semigroups}{book}{
      author={Eklund, P.},
      author={García, J.~Gutiérrez},
      author={Höhle, U.},
      author={Kortelainen, J.},
       title={Semigroups in complete lattices: Quantales, modules and related topics},
      series={Developments in Mathematics},
   publisher={Springer},
     address={Cham},
        date={2018},
      volume={54},
}

\bib{allegories}{book}{
      author={Freyd, P.},
      author={Scedrov, A.},
       title={Categories, allegories},
   publisher={North-Holland},
        date={1990},
}

\bib{haghverdi_2000}{article}{
      author={Haghverdi, E.},
       title={Unique decomposition categories, geometry of interaction and combinatory logic},
        date={2000},
     journal={Mathematical Structures in Computer Science},
      volume={10},
      number={2},
       pages={205–230},
}



\bib{heunenkornell}{article}{
      author={Heunen, C.},
      author={Kornell, A.},
       title={{Axioms for the category of Hilbert spaces}},
        date={2022},
     journal={Proc Natl Acad Sci U S A},
}

\bib{heunenvicary}{book}{
      author={Heunen, C.},
      author={Vicary, J.},
       title={Categories for quantum theory an introduction:},
   publisher={Oxford University Press},
        date={2019},
}

\bib{Heymans-Stubbe}{article}{
      author={Heymans, H.},
      author={Stubbe, I.},
       title={{Grothendieck quantaloids for allegories of enriched categories}},
        date={2012},
     journal={Bulletin of the Belgian Mathematical Society - Simon Stevin},
      volume={19},
      number={5},
       pages={859 \ndash  888},
         url={https://doi.org/10.36045/bbms/1354031554},
}

\bib{Heymans-Stubbe-2}{article}{
      author={Heymans, H.},
      author={Stubbe, I.},
       title={Modules on involutive quantales: Canonical hilbert structure, applications to sheaf theory},
        date={2009},
     journal={Order},
      volume={26},
      number={2},
       pages={177\ndash 196},
}

\bib{monoidaltopology}{book}{
      author={Hofmann, D.},
      author={Tholen, W.},
      author={Seal, G.J.},
       title={Monoidal topology: A categorical approach to order, metric, and topology},
   publisher={Cambridge University Press},
        date={2014},
}

\bib{johnstone:elephant1}{book}{
      author={Johnstone, P.~T.},
       title={Sketches of an elephant: A topos theory compendium, volume {I}},
   publisher={Oxford University Press},
        date={2002},
}





\bib{Kelly-Laplaza}{article}{
      author={Kelly, G.M.},
      author={Laplaza, M.L.},
       title={Coherence for compact closed categories},
        date={1980},
     journal={Journal of Pure and Applied Algebra},
      volume={19},
       pages={193\ndash 213},
}

\bib{Kornell11}{article}{
      author={Kornell, A.},
       title={Quantum functions},
        date={2011},
      eprint={arXiv:1101.1694},
}

\bib{Kornell17}{article}{
      author={Kornell, A.},
       title={Quantum collections},
        date={2017},
     journal={Int. J. Math.},
      volume={28},
}

\bib{Kornell18}{article}{
      author={Kornell, A.},
       title={Quantum sets},
        date={2020},
     journal={J. Math. Phys.},
      volume={61},
}

\bib{Kornell20}{article}{
      author={Kornell, A.},
       title={Discrete quantum structures i: Quantum predicate logic},
        date={2024},
     journal={J. Noncommut. Geom.},
      volume={18},
      number={1},
       pages={337\ndash 382},
}

\bib{Kornell23}{article}{
      author={Kornell, A.},
       title={Axioms for the category of sets and relations},
        date={2025},
     journal={Theory and Applications of Categories},
      volume={44},
      number={10},
       pages={305\ndash 325},
}

\bib{klm-qpl}{article}{
      author={Kornell, A.},
      author={Lindenhovius, B.},
      author={Mislove, M.},
       title={{Quantum CPOs}},
        date={2021},
     journal={Proceedings 17th International Conference on Quantum Physics and Logic},
         url={https://arxiv.org/abs/2109.02196},
}

\bib{KLM20}{article}{
      author={Kornell, A.},
      author={Lindenhovius, B.},
      author={Mislove, M.},
       title={{A category of quantum posets}},
        date={2022},
     journal={{Indagationes Mathematicae}},
      volume={33},
       pages={1137\ndash 1171},
}

\bib{klm-quantumcpos}{article}{
      author={Kornell, A.},
      author={Lindenhovius, B.},
      author={Mislove, M.},
       title={Categories of quantum cpos},
        date={2026},
     journal={to appear in Mathematical Structures in Computer Science},
         url={https://arxiv.org/abs/2406.01816},
}

\bib{klm-quantumposets}{article}{
      author={Kornell, A.},
      author={Lindenhovius, B.},
      author={Mislove, M.},
       title={A category of quantum posets},
        date={2022},
        ISSN={0019-3577},
     journal={Indagationes Mathematicae},
      volume={33},
      number={6},
       pages={1137\ndash 1171},
         url={https://www.sciencedirect.com/science/article/pii/S0019357722000520},
}

\bib{kuperbergweaver:quantummetrics}{book}{
      author={Kuperberg, G.},
      author={Weaver, N.},
       title={A {V}on {N}eumann {A}lgebra {A}pproach to {Q}uantum {M}etrics: {Q}uantum {R}elations},
      series={Memoirs of the American Mathematical Society},
   publisher={American Mathematical Society},
        date={2012},
}

\bib{KurzMoshierJung}{incollection}{
      author={Kurz, A.},
      author={Moshier, A.},
      author={Jung, A.},
       title={Stone duality for relations},
        date={2023},
   booktitle={Samson abramsky on logic and structure in computer science and beyond},
      editor={Palmigiano, Alessandra},
      editor={Sadrzadeh, Mehrnoosh},
   publisher={Springer Verlag},
       pages={159\ndash 215},
}

\bib{Laird}{article}{
title = {Weighted models for higher-order computation},
journal = {Information and Computation},
volume = {275},
pages = {104645},
date = {2020},
author = {J. Laird}
}


\bib{maclane}{book}{
author={S. Mac Lane},
title={{Categories for the Working Mathematician (2nd edition)}},
year={1998},
publisher={Springer}
}



\bib{MacLane-Moerdijk}{book}{
      author={Lane, S.~Mac},
      author={Moerdijk, I.},
       title={Sheaves in geometry and logic: {A} first introduction to topos theory},
      series={Universitext},
   publisher={Springer New York},
     address={New York, NY},
        date={1992},
        ISBN={9781461209270},
}


\bib{Stubbe}{article}{
      author={Stubbe, I.},
       title={Categorical structures enriched in a quantaloid: categories, distributors and functors},
        date={2005},
     journal={Theory and Applications of Categories},
      volume={14},
      number={1},
       pages={1\ndash 45},
         url={http://eudml.org/doc/125469},
}

\bib{takesaki:oa1}{book}{
      author={Takesaki, M.},
       title={Theory of operator algebra I},
   publisher={Springer},
        date={2000},
}


\bib{Weaver10}{article}{
      author={Weaver, N.},
       title={Quantum relations},
        date={2012},
     journal={Mem. Amer. Math. Soc.},
      volume={215},
}



\end{biblist}
\end{bibdiv}


\end{document}